\documentclass[11pt, leqno]{amsart}

\usepackage{shuffle} %\[A \shuffle B\]

\usepackage{lscape}
\usepackage{graphicx}
\usepackage{amssymb,amsmath,amsthm,amscd}
\usepackage{mathrsfs}
\usepackage{enumerate}
\usepackage[usenames,dvipsnames]{color}
\usepackage[colorlinks=true, pdfstartview=FitV,
 linkcolor=blue,citecolor=blue,urlcolor=blue]{hyperref}

\usepackage[all]{xy}
\usepackage[normalem]{ulem}  % for strikeout \sout

\usepackage{enumitem}

\allowdisplaybreaks[2]
\usepackage{oldgerm}
\usepackage{xspace}

\usepackage{marginnote}
\usepackage{adjustbox}
\usepackage{tikz}
\usetikzlibrary{positioning}
\usetikzlibrary{arrows.meta}
\usetikzlibrary{patterns}
\usepackage{placeins}

\newcommand{\nc}{\newcommand}
\numberwithin{equation}{section}

\newenvironment{red}{\relax\color{red}}{\relax}
\newenvironment{blue}{\relax\color{blue}}{\hspace*{.5ex}\relax}
\newenvironment{yellow}{\relax\color{Dandelion}}{\hspace*{.5ex}\relax}

\newcommand{\beb}{\begin{blue}}
\newcommand{\eb}{\end{blue}}

\newcommand{\bey}{\begin{yellow}}
\newcommand{\ey}{\end{yellow}}

\newcommand{\ber}{\begin{red}}
\newcommand{\er}{\end{red}}

\newcommand{\berE}[1]{\begin{red}{}\marginnote{\fbox{\scshape\lowercase{E}}}%
#1}  % Euiyong

\newcommand{\berJ}[1]{\begin{red}{}\marginnote{\fbox{\scshape\lowercase{J}}}%
#1}  % Jung

\renewcommand{\le}{\leqslant}% change the font <
\renewcommand{\ge}{\geqslant}
\theoremstyle{plain}
\newtheorem{lemma}{Lemma}[section]
\newtheorem{prop}[lemma]{Proposition}
\newtheorem{theorem}[lemma]{Theorem}

\newcommand{\Prop}{\begin{prop}}
\newcommand{\enprop}{\end{prop}}
\newcommand{\Lemma}{\begin{lemma}}
\newcommand{\enlemma}{\end{lemma}}
\newcommand{\Th}{\begin{theorem}}
\newcommand{\enth}{\end{theorem}}
\newtheorem{corollary}[lemma]{Corollary}
\newcommand{\Cor}{\begin{corollary}}
\newcommand{\encor}{\end{corollary}}
\newtheorem{definition}[lemma]{Definition}
\newtheorem{conjecture}{Conjecture}
\newcommand{\Def}{\begin{definition}}
\newcommand{\edf}{\end{definition}}
\newtheorem{sublemma}[lemma]{Sublemma}
\newcommand{\Sublemma}{\begin{sublemma}}
\newcommand{\ensub}{\end{sublemma}}

\theoremstyle{definition}
\newtheorem{remark}[lemma]{Remark}

\newtheorem{example}[lemma]{Example}
\newtheorem{Convention}[lemma]{Convention}
\newcommand{\Conv}{\begin{Convention}}
\newcommand{\enconv}{\end{Convention}}
\nc{\Conj}{\begin{conjecture}}
\nc{\enconj}{\end{conjecture}}
\nc{\Rem}{\begin{remark}}
\nc{\enrem}{\end{remark}}

\newcommand{\Q}{\mathbb {Q}}

\newcommand{\Z}{{\mathbb Z}}

\newcommand{\D}{\mathscr{D}}

\newcommand{\Gr}{{\mathrm{Gr}}}

\newcommand{\one}{{\bf{1}}}

\newcommand{\seteq}{\mathbin{:=}}

\newcommand{\hd}{{\operatorname{hd}}}

\newcommand{\g}{{\mathfrak{g}}}

\newcommand{\Hom}{\operatorname{Hom}}

\newcommand{\pt}{\operatorname{pt}}

\newcommand{\M}{{\mathscr M}}

\newcommand{\hs}{\hspace*}
\newcommand{\ms}{\mspace}

\newenvironment{myequation}
{\relax\setlength{\arraycolsep}{1pt}\begin{eqnarray}}
{\end{eqnarray}}
\newenvironment{myequationn}
{\relax\setlength{\arraycolsep}{1pt}\begin{eqnarray*}}
{\end{eqnarray*}}

\nc{\eq}{\begin{myequation}}
\nc{\eneq}{\end{myequation}}
\nc{\eqn}{\begin{myequationn}}
\nc{\eneqn}{\end{myequationn}}

\newcommand{\on}{\operatorname}

\newcommand{\bni}{\be[label=\rm(\roman*)]}
\newcommand{\bnum}{\bni}
\newcommand{\bna}{\be[label=\rm(\alph*)]}

\newcommand{\ba}{\begin{array}}
\newcommand{\ea}{\end{array}}

\newcommand{\supp}{\operatorname{supp}}

\newcommand{\eqsub}{\begin{subequations}\begin{eqnarray}}
\newcommand{\eneqsub}{\end{eqnarray}\end{subequations}}

\nc{\la}{\lambda}
\nc{\lam}{\lambda}
\nc{\U}[1][\g]{U_q(#1)}
\nc{\te}{\tilde{e}}
\nc{\tem}{\tilde{e}^{\mathrm{max}}}
\nc{\tei}{\tilde{e}_i}
\nc{\tf}{\tilde{f}}
\nc{\tfm}{\tilde{f}^{\mathrm{max}}}
\nc{\tfi}{\tilde{f}_i}
\nc{\tU}{\widetilde U_q(\g)}
\nc{\tE}{\widetilde{E}}
\nc{\tF}{\widetilde{F}}
\nc{\tK}{\widetilde{K}}
\nc{\tEs}{\widetilde{E}^*}
\nc{\tFs}{\widetilde{F}^*}

\nc{\ttE}{\widetilde{\mathcal{E}}}
\nc{\ttF}{\widetilde{\mathcal{F}}}
\nc{\ttEs}{\ttE^*}
\nc{\ttFs}{\ttF^*}
\nc{\tfs}{\tf^*}
\nc{\tfss}[1]{\tf^{* \hskip 0.05em #1}}
\nc{\tess}[1]{\te^{* \hskip 0.05em #1}}
\nc{\tes}{\te^*}
\nc{\tesm}{\tilde{e}^{* \hskip 0.05em \mathrm{max}}}
\nc{\tfsm}{\tilde{f}^{* \hskip 0.05em \mathrm{max}}}

\nc{\tk}{\tilde{k}}
\nc{\tkone}{\tk_{\ol{1}}}
\nc{\teone}{\tilde{e}_{\ol{1}}}
\nc{\tfone}{\tilde{f}_{\ol{1}}}

\nc{\teibar}{\tilde{e}_{\ol{i}}} \nc{\tfibar}{\tilde{f}_{\ol{i}}}
\nc{\tki}{{\tk}_{\ol {i}}}

\nc{\BZ}{{\mathbb{Z}}}
\nc{\al}{\alpha}
\nc{\tal}{\widetilde{\al}}
\nc{\tch}{\widetilde{h}}
\nc{\qs}{{q}}
\nc{\lan}{\langle}
\nc{\ran}{\rangle}
\nc{\re}{{\mathrm{re}}}
\nc{\wt}{\operatorname{wt}}
\nc{\hwt}{\widehat{\wt}}
\nc{\Ht}{\mathrm{ht}}
\nc{\hHt}{\widehat{\Ht}}
\nc{\ch}{\operatorname{ch}}
\nc{\Um}[1][\g]{U^-_q(#1)}
\nc{\Ue}{U^+_q(\g)}
\nc{\ep}{\varepsilon}
\nc{\hep}{\widehat{\ep}}
\nc{\vphi}{\varphi}
\nc{\sphi}{\varphi^*}
\nc{\eps}{\ep^*}
\nc{\heps}{\hep^{ \hskip 0.2em  *}}

\nc{\nn}{\nonumber}
\def\max{{\mathop{\mathrm{max}}}}
\nc{\vp}{\varpi}
\nc{\cls}{{\operatorname{cl}}}
\nc{\Wt}{{\operatorname{Wt}}}
\nc{\Us}{U'_q(\g)}
\nc{\La}{\Lambda}
\nc{\tLa}{\widetilde\Lambda}
\nc{\ro}{{\rm(}}
\nc{\rf}{{\rm)}}
\nc{\norm}{{\mathrm{norm}}}
\nc{\qbox}{\quad\mbox}
\nc{\braid}{{\mathfrak{B}}}
\nc{\Ad}{\operatorname{Ad}}
\nc{\Aut}{\operatorname{Aut}}
\nc{\dt}[1]{\tilde{\tilde #1}}
\nc{\Sn}{S^{{\mathrm{norm}}}}
\nc{\aff}{{\mathrm{aff}}}
\nc{\rk}{{\mathrm{rk}}}
\nc{\tP}{\widetilde{P}}
\nc{\tW}{\widetilde{W}}
\nc{\Dyn}{\mathrm{Dyn}}
\nc{\tD}{\widetilde{\Delta}}
\nc{\height}[1]{{\operatorname{ht}}(#1)}
\nc{\bl}{\bigl(}
\nc{\br}{\bigr)}
\nc{\Hecke}{\mathrm{H}}
\nc{\HA}{\Hecke^{\mathrm{A}}}
\nc{\HB}{\Hecke^{\mathrm{B}}}
\newcommand{\scbul}{{\,\raise1pt\hbox{$\scriptscriptstyle\bullet$}\,}}
\nc{\vac}{{\phi}}
\nc{\Bt}{\B_\theta(\g)}
\nc{\be}{\begin{enumerate}}
\nc{\ee}{\end{enumerate}}
\nc{\low}{{\mathrm{low}}}
\nc{\upper}{{\mathrm{up}}}
\nc{\Zodd}{\Z_{\mathrm{odd}}}
\nc{\Ft}[1][n]{\mathbb{P}\mathrm{ol}_{#1}}
\nc{\Ftf}[1][n]{\widetilde{\mathbb{P}\mathrm{ol}}_{#1}}
\nc{\KA}{\on{K}^{\mathrm{A}}}
\nc{\KB}{\on{K}^{\mathrm{B}}}
\nc{\Res}{\on{Res}}
\nc{\Fc}[1][{n,m}]{\mathbf{F}_{#1}}
\nc{\tphi}{\tilde{\varphi}}
\nc{\CO}{\mathscr{O}}
\nc{\inte}{\mathrm{int}}
\nc{\Oint}{\mathcal{O}^{\ge0}_{\inte}}
\nc{\vs}{\vspace*}
\nc{\tLt}{\widetilde{L}}
\nc{\tL}{\widetilde{\Lambda}}
\nc{\tu}{\tilde{u}}
\nc{\noi}{\noindent}
\nc{\heigh}{\mathfrak{t}}
\nc{\lowest}{\mathfrak{l}}
\nc{\rootl}{\mathsf{Q}}
\nc{\rlQ}{\rootl}
\nc{\cl}{\colon}

\nc{\uqpg}{U'_q(\mathfrak g)}
\nc{\uq}{\uqpg}
\nc{\Oh}{\widehat{\mathcal{O}}}
\nc{\pn}{p_{\mathfrak{n}}}

\nc{\KLR}{KLR algebra}
\nc{\KLRs}{KLR algebras}
\nc{\cor}{\mathbf{k}}
\nc{\cora}{{\cor(A)}}
\nc{\haut}{\mathrm{ht}}
\nc{\tens}{\mathop\otimes}
\nc{\gmod}{\mbox{-$\mathrm{gmod}$}}
\nc{\gMod}{\mbox{-$\mathrm{gMod}$}}
\nc{\proj}{\mbox{-$\mathrm{proj}$}}
\nc{\gproj}{\mbox{-$\mathrm{gproj}$}}
\nc{\smod}{\mbox{-$\mathrm{mod}$}}
\nc{\Mod}{\mbox{-$\mathrm{Mod}$}}
\nc{\h}{\mathfrak h}
\nc{\Rnorm}{R^{\mathrm{norm}}}
\nc{\Runiv}{R^{\mathrm{univ}}}
\nc{\Rren}{R^{\mathrm{ren}}}

\nc{\Vhat}{\widehat{V}}
\nc{\F}{\mathcal{F}}

\def\T{{\mathcal T}}

\def\Gup{{\mathrm{G^{up}}}}
\def\tGup{{\widetilde{\mathrm{G}}^{\rm up} }}

\nc{\fd}[1][A]{\on{\mathrm{flat.dim}_{#1}}}
\nc{\bP}{{\mathbb{P}}}
\nc{\bPh}{\widehat{\mathbb{P}}}
\nc{\bK}[1][{n}]{\widehat{\mathbb{K}}_{#1}}
\nc{\bV}[1][{n}]{\widehat{V}^{\otimes{#1}}}
\nc{\bVK}[1][{n}]{\widehat{V}^{\otimes{#1}}_{\widehat{\mathbb{K}}}}
\nc{\hV}{\widehat{V}}
\nc{\opp}{\mathrm{opp}}
\nc{\col}{\colon}
\nc{\oep}{\epsilon}
\nc{\qtext}{\quad\text}
\nc{\qtextq}[1]{\quad\text{#1}\quad}
\nc{\longtwoheadrightarrow}[1][]{\xymatrix{\ar@{->>}[r]^-{{#1}}&}}
\nc{\epiTo}[1][]{\longtwoheadrightarrow[{#1}]}
\nc{\epito}{\twoheadrightarrow}
\nc{\monoTo}[1][]{\xymatrix{\ar@{>->}[r]^-{{#1}}&}}
\nc{\sym}{\mathfrak{S}}
\nc{\inp}[1]{{({#1})_{\mathrm{n}}}}
\nc{\rtl}{\rootl}
\nc{\wtd}{\widetilde}
\nc{\etens}{\boxtimes}
\nc{\ds}[1]{\mathrm{d}(#1)}
\nc{\rmat}[1]{{\mathbf{r}}_%
{\mspace{-2mu}\raisebox{-.6ex}{${\scriptstyle{#1}}$}}}
\nc{\rmats}[1]{{\mathbf{r}}_%
{\mspace{-2mu}\raisebox{-.6ex}{${\scriptscriptstyle{#1}}$}}}
\nc{\shc}{\mathcal{C}}
\nc{\shs}{\mathcal{S}}
\nc{\Fct}{{\on{Fct}}}
\nc{\tC}{\widetilde{\shc}}
\nc{\Zp}{\Z_{\ge0}}
\nc{\tPhi}{\widetilde{\Phi}}
\nc{\tT}{{\widetilde{\T}}}
\nc{\Ob}{\on{Ob}}
\nc{\bwr}{\mbox{\large$\wr$}}
\nc{\Img}{\on{Im}}
\nc{\Ab}{\mathcal{A}^{\mathrm{big}}}
\nc{\Sb}{\mathcal{S}^{\mathrm{big}}}
\nc{\As}{\mathcal{A}}
\nc{\Ss}{\mathcal{S}}
\nc{\ntens}{\widetilde{\otimes}}
\nc{\hR}{\widehat{R}}
\nc{\nconv}{\mathop{\mbox{\large $\odot$}}}
\nc{\snconv}{\mbox{\scriptsize$\odot$}}
\nc{\ts}{\tilde{s}}
\nc{\sho}{\mathcal{O}}
\nc{\bc}{\begin{cases}}
\nc{\ec}{\end{cases}}
\nc{\slnh}{{\widehat{\mathfrak{sl}}_N}}
\nc{\UA}{U_q'(\slnh)}
\nc{\KR}{R_K}
\nc{\cQ}{\mathcal{Q}}
\nc{\Irr}{\mathcal{I}rr}
\nc{\tQ}{\widetilde{\cQ}}
\nc{\bs}{\mathbf{s}}
\nc{\bL}{\mathbb{L}}
\nc{\tg}{\tilde{g}}

\nc{\conv}{\mathbin{\mbox{\large $\circ$}}}
\nc{\shconv}{\mathbin{\large\diamond}}
\nc{\sconv}{\mathbin{\large\Delta}}
\nc{\stens}{\mathbin{\large\Delta}}
\nc{\hconv}{\mathbin{\nabla}}
\nc{\htens}{\mathbin{\nabla}}

\nc{\Rm}{R^{\mathrm{ren}}}

\nc{\bQ}{\ol{Q}}

\nc{\de}{\on{\textfrak{d}\ms{1mu}}}

\nc{\xmono}{\ar@{>->}}
\nc{\xepi}{\ar@{->>}}
\nc{\db}[1]{\raisebox{-.5ex}[2ex][1.8ex]{$#1$}}
\nc{\wb}[1]{\mbox{$\rule[-1.1ex]{0ex}{2ex}#1$}}
\nc{\univ}{\mathrm{univ}}
\nc{\rM}{{}^*\mspace{-2mu}M}
\nc{\lM}{M^*}
\nc{\uqm}{\uq\smod}
\nc{\tR}{\widetilde{R}_{\gamma,\beta}}
\nc{\tx}{\tilde{x}}
\nc{\bi}{\mathbf{i}}
\nc{\ttau}{\widetilde{\tau}}

\nc{\tEnd}{\on{\widetilde{E}nd}}
\nc{\tHom}{\on{\widetilde{H}om}}

\nc{\K}{{J}}
\nc{\Kex}{{\K}_{\mathrm{ex}}}
\nc{\Kfr}{{\K}_{\mathrm{f\mspace{.01mu}r}}}
\nc{\coro}{\cor}
\nc{\tB}{\widetilde{B}}
\nc{\seed}{\mathcal{S}}
\nc{\mseed}{\mathscr{S}}
\nc{\up}{\mathrm{up}}
\nc{\bfa}{\mathbf{a}}
\nc{\bfb}{\mathbf{b}}
\nc{\bfc}{\mathbf{c}}
\nc{\bfm}{\mathbf{m}}
\nc{\hbfm}{ \widehat{\mathbf{m}}}

\newlength{\mylength}
\nc{\ov}[1]{\overline{#1}}
\nc{\Wlmj}[3]{\W_{#2,#3}^{(#1)}}
\nc{\Mkl}[2]{\M_\ttww(#1,#2)}
\nc{\mqs}{(-q^2)}
\nc{\Cquiver}{\upsigma}

\nc{\mut}[1]{{\mu}_{\mspace{-2mu}\raisebox{-.5ex}{${\scriptstyle{#1}}$}}}

\nc{\Kt}{\mathcal K_t}
\nc{\KT}{\mathbb{K}_t}
\nc{\yim}{y_{i,m}}
\nc{\yjm}{y_{j,m}}
\nc{\yjp}{y_{j,p}}
\nc{\yimp}{y_{i,m+1}}
\nc{\yjmp}{y_{j,m+1}}

\nc{\Refl}{\mathscr{S}}
\nc{\Reflinv}{{\Refl}^{-1}}

\nc{\Refn}{\mathsf{S}}

\nc{\catC}{\mathscr C}
\nc{\catA}{\mathcal A}
\nc{\shift}{{\mathrm T}}
\nc{\rE}{ \mathsf{E} }
\nc{\rW}{ \mathcal{W} }
\nc{\rES}{ \mathcal{E} }

\nc{\brd}{\sigma} %generator of the braid group
\nc{\into}{\xymatrix@C=3ex{{}\ar@{^{(}->}[r]&{}}}
\nc{\dual}{\D}
\nc{\cdual}{\mathcal{D}}
\nc{\cat}[1][{\g}]{\catC_{#1}^0}
\nc{\catCO}{{\catC_\g^0}}
\nc{\catCOD}{{\catC_\g^{0, \ddD}}}
\nc{\catCQ}{{\catC_{\qQ}}}
\nc{\catCQd}{{\catC_{\widetilde{\qQ}}}}
\nc{\catCD}{{\catC_{\ddD}}}
\nc{\catCDK}{{\catC_{\ddD, \iK}}}
\nc{\Li}{{\La^\infty}}
\nc{\sig}{{\sigma(\g)}}
\nc{\sigZ}{{\sigma_0(\g)}}
\nc{\sigQ}{{\sigma_\qQ(\g)}}
\nc{\sigQd}{{\sigma_{\widetilde{\qQ}}(\g)}}
\nc{\phiQd}{\phi_{\widetilde{\qQ}}}
\nc{\sigD}{{\sigma_\ddD(\g)}}

\nc{\ZZ}{{\mathbf{Z}}}
\nc{\sP}{{\mathsf{P}}}
\nc{\sV}{{\mathsf{V}}}
\nc{\rxw}{{\underline{w}}}
\nc{\rxwz}{{\underline{w_0}}}
\nc{\boten}[1]{\overrightarrow{\bigotimes_{#1}}}
\nc{\cmA}{{\mathsf{A}}}
\nc{\cmC}{{\mathsf{C}}}
\nc{\ddD}{{\mathcal{D}}}
\nc{\ddDQ}{{\ddD_Q}}
\nc{\ddDQd}{{\ddD_{\widetilde{Q}}}}
\nc{\qQ}{{\mathcal{Q}}}
\nc{\gf}{{\g_{\mathrm{fin}}}}
\nc{\Df}{{\Delta_{\mathrm{fin}}}}
\nc{\If}{{I_{\mathrm{fin}}}}
\nc{\cmAf}{{\cmA_{\mathrm{fin}}}}
\nc{\weyl}{{\mathsf{W}}}
\nc{\weylf}{{\mathsf{W}_{\mathrm{fin}}}}
\nc{\sg}{{\mathfrak{S}}}
\nc{\weylA}{{\mathsf{W}_\cmA}}
\nc{\weylC}{{\mathsf{W}_\cmC}}
\nc{\Deg}{\mathrm{Deg}}
\nc{\Di}{\Deg^\infty}
\nc{\KRc}{{K_{q=1}(R_\cmC\gmod)}}
\nc{\prD}{{\Delta^+}}
\nc{\prDf}{{\Delta^+_{\mathrm{fin}}}}
\nc{\nrD}{{\Delta^-}}
\nc{\prDA}{{\Delta^+_\cmA}}
\nc{\prDC}{{\Delta^+_\cmC}}
\nc{\nrDC}{{\Delta^-_\cmC}}
\nc{\n}{{\mathfrak{n}}}
\nc{\Rt}{\mathsf{L}} %simple root modules in the dyality datum
\nc{\Cp}{\mathsf{V}} %cuspidal modules
\nc{\cuspS}{{\mathsf{S}}}
\nc{\st}[1]{\{{#1}\}}
\nc{\bst}[1]{\bigl\{{#1}\bigr\}}
\nc{\WS}{Quantum affine Schur-Weyl duality\xspace}
\nc{\CWS}{Quantum affine Weyl-Schur duality}
\nc{\zz}{{{\mathbf{z}}}}
\nc{\wlP}{\mathsf{P}}
\nc{\twlP}{\widetilde{\wlP}}
\nc{\wl}{\wlP}
\nc{\clp}{{\mathrm{cl}}}
\nc{\wlPc}{{\wlP_\clp}}
\nc{\awlP}{\widehat{\mathsf{P}}}
\nc{\dM}{\mathsf{M}}
\nc{\dC}{\mathsf{C}}
\nc{\cC}{\mathcal{C}}
\nc{\tcC}{\widetilde{\mathcal{C}}}
\nc{\lR}{\widetilde{{R}}}
\nc{\zero}{\mathrm{zero}}

\nc{\prtl}[1][J]{\rootl_{#1}^+}
\nc{\hL}{\widehat{\Rt}}
\nc{\hF}{\widehat{\F}}
\nc{\Proof}{\begin{proof}}
\nc{\QED}{\end{proof}}
\nc{\e}{\mathrm{e}}

\nc{\Aff}{\mathrm{Aff}}
\nc{\rT}{\mathcal{T}}		% reflction in R-gmod
\nc{\rr}{rationally renormalizable\xspace}
\nc{\RA}{{R_\cmA}}		%  R_A-gmod
\nc{\RC}{{R_\cmC}}		%  R_C-gmod
\nc{\proolim}[1][]{\mathop{``{\varprojlim}{\mbox{''}}}\limits_{#1}}
\nc{\qtq}[1][\text{and}]{\quad\text{#1}\quad}
\newcommand{\gW}{\mathsf{W}}

\newcommand{\Sp}{\mathrm{span}_{\mathbb{R}_{\ge0}}}  	% span
\nc{\corh}{\widehat{\cor}}
\nc{\ang}[1]{\langle{#1}\rangle}
\nc{\rc}{renormalizing coefficient\xspace}
\nc{\cz}{{\cor[z^{\pm1}]}}
\nc{\tp}{\ms{1.5mu}{\widetilde{p}}\ms{2mu}}
\nc{\G}{\mathcal{G}}
\nc{\cc}{\mathfrak{c}}
\nc{\rsP}{{\Phi_\g}}
\nc{\rsX}{{X_\g}}
\nc{\rs}{ \mathsf{s} }
\nc{\Dynkin}{\mathsf{D}}
\nc{\Dat}{\sigma}
\nc{\hf}{\xi}

\nc{\hBi}{\widehat{B}(\infty)}
\nc{\hBvi}{{\widehat{B}_\vt(\infty)}}
\nc{\hBsi}{{\widehat{B}_\sigma(\infty)}}
\nc{\iK}{\mathsf{K}}
\nc{\cBg}[1][\g]{\widehat{B}_{#1}(\infty)}
\nc{\cBsg}[1][\g]{\widehat{B}_{#1}(\infty)^*}
\nc{\cBgk}[1][\g]{\widehat{B}_{#1}^{\iK}(\infty)}
\nc{\cb}{\mathbf{b}}
\nc{\cI}{ \widehat{I} }
\nc{\cIf}{{\widehat{I}_{\mathrm{fin}}}}
\nc{\cIz}{{\widehat{I}_{0}}}
\nc{\cJ}{ \widehat{J} }
\nc{\sB}{\mathcal{B}}
\nc{\sBk}{\sB_{\iK}}
\nc{\sBdk}{\sB_{\ddD, \iK}}
\nc{\sBD}{\sB_{\ddD}}
\nc{\sBkg}{\sBk(\g)}
\nc{\cs}{\star}

\nc{\cd}{\mathrm{D}}
\nc{\cm}{\mathbf{m}}
\nc{\MS}{\mathsf{MS}}
\nc{\hMS}{\widehat{\mathsf{MS}}}
\nc{\rS}{\mathbf{S}}
\nc{\rSs}{\rS^*}
\nc{\brS}{\overline{\rS}}

\nc{\crI}{ \mathsf{I}}
\nc{\crhI}{\mathscr{S} }
\nc{\crD}{\mathsf{D}}
\nc{\crc}{\mathsf{c}}
\nc{\crB}{\mathscr{P}_n}
\nc{\cru}{\mathsf{u}}
\nc{\crsh}{\mathsf{sh}}
\nc{\bfs}{\mathbf{s}}
\nc{\bfp}{\mathbf{p}}
\nc{\crBB}[1]{\mathscr{P}_#1}

\nc{\gc}{{\g_{\cmC}}}
\nc{\cL}{\mathcal{L}}
\nc{\cLD}{{\cL_\ddD}}
\nc{\ccL}{\mathscr{L}}
\nc{\ccLD}{{\ccL_\ddD}}
\nc{\clen}{\mathsf{len}}

\nc{\hv}{\mathsf{1}}
\nc{\qt}[1]{\quad\text{#1}}
\nc{\snoi}{\smallskip\noindent}
\nc{\mnoi}{\medskip\noindent}

\nc{\ul}[1]{\underline{#1}}
\nc{\dul}[1]{\underline{\underline{#1}}}
\nc{\qh}{\qedhere}
\nc{\ca}{\mathsf{v}}
\nc{\sck}[1][k]{\ms{8mu}{\raisebox{-1.3ex}{$\scriptstyle{#1}$}\hs{-1.4ex}{\succ}}\ms{4mu}}
\nc{\scke}[1][k]{\ms{8mu}{\raisebox{-1.3ex}{$\scriptstyle{#1}$}\hs{-1.4ex}%
{\succcurlyeq}}\ms{4mu}}

\nc{\edot}{\emptyset}
\nc{\Nf}{N_{\gf}}
\nc{\fin}{\mathrm{fin}}
\nc{\trg}{\scalebox{.7}{$\triangle$}}
\nc{\ake}[1][1ex]{\rule[-#1]{0ex}{1ex}}
\nc{\akew}[1][1ex]{\rule[-1ex]{#1}{0ex}}
\nc{\akeu}[1][1ex]{\rule[#1]{0ex}{1ex}}

\nc{\bg}{\mathscr{B}}
\newcommand{\B}{B(\infty)}
\nc{\ipi}{{{_i}\pi}}
\nc{\pii}{{\pi_i}}
\nc{\Ld}{\mathcal{P}}
\nc{\Dc}{\Upsilon}
\nc{\DcL}{\mathcal{L}}
\nc{\bR}{\mathrm{K}}
\nc{\bRs}{\bR^*}
\nc{\rR}{\mathrm{R}}
\nc{\bfi}{\mathbf{i}}
\nc{\bfj}{\mathbf{j}}
\nc{\bfk}{\mathbf{k}}
\nc{\vt}{\vartheta}
\nc{\nP}{\Upsilon}
\nc{\hP}{\widehat{\nP}}
\nc{\fF}{ \widetilde{\mathrm{F}}}
\nc{\fE}{ \widetilde{\mathrm{E}}}
\nc{\fR}{\mathscr{R}}
\nc{\fT}{\mathscr{T}}
\nc{\orb}{\mathrm{orb}}
\nc{\fdr}{\mathrm{r}}

\nc{\PBW}{\mathrm{PBW}}
\nc{\STR}{\mathrm{STR}}
\nc{\cP}{\mathsf{P}}
\nc{\cS}{\mathsf{S}}
\nc{\tcP}{\widetilde{\cP}}
\nc{\tcS}{\widetilde{\cS}}
\nc{\gL}{\mathrm{g}^{\mathrm{L}}}
\nc{\gR}{\mathrm{g}^{\mathrm{R}}}
\nc{\GL}{\mathrm{G}^{\mathrm{L}}}
\nc{\GR}{\mathrm{G}^{\mathrm{R}}}
\nc{\sS}{\mathcal{S}}
\nc{\fr}{\mathrm{fr}}
\nc{\uf}{\mathrm{ex}}
\nc{\gN}{\mathcal{N}}
\nc{\gM}{\mathcal{M}}

\nc{\LB}{\widetilde{B}}
\nc{\qA}{\mathrm{A}}
\nc{\tqA}{\widetilde{\qA}}
\nc{\qmD}{\mathrm{D}}
\nc{\qmC}{\mathrm{C}} % frozen
\nc{\fz}{\mathsf{c}}
\nc{\eqq}{{\, \equiv_q \,}}
\nc{\eqfr}{{\, \equiv_{\mathrm{fr}} \,}}
\nc{\strS}{\mathfrak{S}}
\nc{\pbwP}{\mathfrak{P}}
\nc{\cvec}[3]{\begin{pmatrix} #1 \\ #2 \\ #3 \end{pmatrix}}
\nc{\PD}{\varrho}
\nc{\bx}{\mathsf{b}}
\nc{\res}{\mathsf{res}}
\nc{\HW}{\mathscr{W}}
\nc{\YD}{\mathcal{Y}}
\nc{\SYD}{\mathcal{SY}}
\nc{\Zq}{{\Z[q^{\pm1}]}}
\nc{\Fup}{\mathrm{F}^{\rm up}}
\nc{\Gs}{\mathbf{G}}
\nc{\tGs}{\widetilde{\Gs}}
\nc{\IM}{I_{\mathrm{M}}}
\nc{\Specht}{\mathcal{S}}
\nc{\ST}{\mathsf{ST}}
\nc{\qdim}{\mathrm{dim}_q}
\nc{\qch}{\mathrm{ch}_q}
\nc{\lcm}{\mathrm{lcm}}

\title[Crystals and quantum twist automorphisms]{Crystals and quantum twist automorphisms}

\author[W.-S. Jung]{Woo-Seok Jung}
\thanks{The research of W.-S. Jung was supported by Basic Science Research Program through the National Research Foundation of Korea(NRF) funded by the Ministry of Education(2024-00451844).}
\address[W.-S. Jung]{Department of Mathematics, University of Seoul, Seoul 02504, Korea}
\email{jungws@uos.ac.kr}

\author[E. Park]{Euiyong Park}
\thanks{The research of E.\ Park was supported by the National Research Foundation of Korea (NRF) Grant funded by the Korea Government(MSIT)(RS-2023-00273425 and NRF-2020R1A5A1016126).}
\address[E. Park]{Department of Mathematics, University of Seoul, Seoul 02504, Korea}
\email{epark@uos.ac.kr}


\makeatletter
\@namedef{subjclassname@2020}{\textup{2020} Mathematics Subject Classification}
\makeatother

\date{\today}

\begin{document}

\begin{abstract}
Let $\eta_w$ be the quantum twist automorphism for the quantum unipotent coordinate ring $\qA_q(\n(w))$ introduced by Kimura and Oya. 
In this paper, we study the quantum twist automorphism $\eta_w$ in the viewpoint of the crystal bases theory and provide a crystal-theoretic description of $\eta_w$. 
In the case of the $*$-twisted minuscule crystals of classical finite types, we provide a combinatorial description of $\eta_w$ in terms of (shifted) Young diagrams. We further investigate the periodicity of $\eta_w$ up to a multiple of frozen variables in various setting.
\end{abstract}

\maketitle

\setcounter{tocdepth}{1}
\tableofcontents

\section*{Introduction}

The \emph{quantum twist automorphism} (\cite{KiOy21}) constructed by Kimura and Oya is a quantum-analogue of the \emph{twist automorphism} on a unipotent cell introduced by Berenstein, Fomin and Zelevinsky (\cite{BFZ96, BZ97}). The non-quantum twist automorphism plays a crucial role in describing Lusztig's totally positive parts for Schubert varieties, which is known as the \emph{Chamber Ansatz}.  Kimura-Oya's quantum twist automorphism recovers the non-quantum twist automorphism by specializing at $q=1$ and has remarkable compatibilities with the \emph{upper global basis} (or \emph{dual canonical basis}) (see \cite{KashBook02, LusztigBook} and references therein) and the quantum cluster algebra structure (\cite{FZ02, BZ05}).

Let $\cmA$ be a symmetrizable generalized Cartan matrix and $\weyl$ the corresponding Weyl group.
We denote by $\qA_q(\n(w))$ the \emph{quantum unipotent coordinate ring} associated with $\cmA$ and $w\in \weyl$ and by $\tqA_q(\n(w))$ the \emph{localization} of $\qA_q(\n(w))$ by frozen variables (see Section \ref{Sec: qcr} for details). 
Note that the localized algebra $\tqA_q(\n(w))$ is a quantization of the coordinate ring of the unipotent cell and  has a \emph{quantum cluster algebra} structure (\cite{GLS13, GY17}).  
We denote by $\eta_w$ the quantum twist automorphism constructed by Kimura-Oya, which is an automorphism of the localized algebra $\tqA_q(\n(w))$. It was shown in \cite{KiOy21} that $\eta_w$ preserves the upper global basis $\tGup(w)$ of $\tqA_q(\n(w))$ and has compatibility with quantum cluster monomials of $\tqA_q(\n(w))$. 
The quantum twist automorphism $\eta_w$ also has categorical interpretations in the viewpoint of \emph{additive and monoidal categorifications}. 
In the case of the additive categorification (\cite{GLS11, GLS13}), it was shown in \cite{GLS12} that $\eta_w$ can be categorified by the \emph{Auslander-Reiten translations} on certain module categories determined by $w$ over \emph{preprojective algebras}. In monoidal categorification of $\tqA(\n(w))$ (\cite{KKKO18}),  $\eta_w$ can be  categorified by the right dual functor $\dual_w$ of a certain finite-dimensional module category determined by $w$ over \emph{quiver Hecke algebras} (see \cite{KKOP21C} and see also Proposition \ref{Prop: eta_w=dual^-1}).  
In \cite{KQQ23}, the notion of (quantum) twist automorphism was extended in more general cluster algebras setting. We remark that there is another quantum-analogue of twist automorphism by Berenstein and Rupel (\cite{BR15}) and the relation to Kimura-Oya's quantum twist automorphism was studied in \cite{Oya19}. 

The \emph{crystal bases theory} is a powerful combinatorial tool to investigate representations over a quantum group (see \cite{K91, K93, K95, KashBook02}). There are numerous interesting applications to various areas in terms of combinatorial objects including Young diagrams and tableaux (see \cite{BSBook} and references therein). 
The crystal basis is obtained by specializing the (upper) global basis at $q=0$, which tells us that there is a natural bijection between crystal basis and (upper) global basis.
Since the quantum twist automorphism $\eta_w$ preserves the upper global basis $\tGup(w)$, $\eta_w$ also can be understood as permutations in terms of crystal bases. Thus it would be interesting to ask how $\eta_w$ can be described in terms of crystals.  

\smallskip 

In this paper, we study the quantum twist automorphism $\eta_w$ in the viewpoint of crystals and provide a crystal-theoretic description of $\eta_w$ for \emph{arbitrary symmetrizable Kac-Moody} types. This description enables an explicit computation of $\eta_w$ using both combinatorial and computer-aided methods. 
In the case where $w$ is the longest element in $\weyl$, this crystal-theoretic description gives an affirmative answer to the problem of asking an explicit description of the right dual functor $\dual_w$ proposed by Nakashima (see \cite[Problem 4]{Nakashima22}). 
As an application, when 
$b$ is an element of the \emph{$*$-twisted minuscule crystal} (see \eqref{Eq: *-twisted MC}) of finite classical type and $w$ is the longest element of the quotient $\weyl_t \backslash \weyl$ for a minuscule index $t \in I$, 
we provide a combinatorial description of $\eta_w (b)$ in terms of (shifted) Young diagrams. We then investigate the periodicity of $\eta_w$ up to a multiple of frozen variables in type $A_n$ and gives several examples on the periodicity based on computations by \textsc{SageMath}~\cite{sage}.

Let us explain the results in more detail. 
Let $\cmA$ be a generalized Cartan matrix of \emph{symmetrizable Kac-Moody type}. 
We denote by $B(w)$ the crystal basis of the quantum unipotent coordinate ring $\qA_q(\n(w))$ associated with $\cmA$ and $w \in \weyl$, and let $R$ be the quiver Hecke algebra associated with $\cmA$. For each $i$ in the index set $ I$, we set $\fz^w_i$ to be the crystal element in $B(w)$ corresponding the the \emph{quantum unipotent minor} $ \qmD(w\La_i, \La_i)$, which is a frozen variable in $\qA_q(\n(w))$ as a quantum cluster algebra. We invert the elements $\fz^w_i$ in $B(w)$ using the equivalent relation \eqref{Eq: equiv rel} to localize the crystal $B(w)$, which is denoted by $\LB(w)$ (see \eqref{Eq: localized crystal}). By construction, $\LB(w)$ is in 1-1 correspondence with the upper global basis $\tGup(w)$ of $\tqA_q(\n(w))$, i.e., 
$$
\tGup(w) = \{\tGup(x) \mid x \in \LB(w) \}.
$$
Let $\cC_w$ be the finite-dimensional graded $R$-module category which categorifies the quantum unipotent coordinate ring $\qA_q(\n(w))$ (see Section \ref{Sec: QHA} and Proposition \ref{Prop: categorification}). 
We denote by $\Gs(w)$ and  $\tGs(w)$ the set of the equivalence classes of the simple objects in the category $\cC_w$ and its localization $\tcC_w$ (\cite{KKOP21}) respectively. Under the categorification, $\LB(w)$ also is in 1-1 correspondence with $\tGs(w)$, i.e., 
$$
\tGs(w) = \{ [L(x)] \mid x \in \LB(w) \},
$$
where $L(x)$ is the simple object in $\tcC_w$ corresponding to $x \in \LB(w)$ (see Section \ref{Sec: LB}). 
When it is of finite type and $w$ is the longest element $w_\circ$, Nakashima proved that there is a crystal structure on the set $\tGs(w_\circ)$ arising from the nature of  the localized category $\tcC_w$ and showed that $\tGs(w_\circ)$ is isomorphic to the \emph{cellular crystal} $\mathbb{B}_\bfi$ where $\bfi$ is a reduced expression of $w_\circ$ (see \cite{Nakashima22} and see also Remark \ref{Rmk: LB}).
Thus $\LB(w_\circ)$ has the crystal structure induced from $\tGs(w_\circ)$, which can be understood as abstracting of the cellular crystal $\mathbb{B}_\bfi$.
For this reason, we simply call $\LB(w)$ the \emph{localized crystal}.

We investigate the bijections among $\LB(w)$, $\tGup(w)$ and $\tGs(w)$ with the interpretation of $\eta_w$ in terms of the right dual functor $\dual_w$ on $\tcC_w$ by Ishibashi-Oya (see \cite[Theorem B.22]{IO21_arXiv2} and see also Proposition \ref{Prop: eta_w=dual^-1}). This gives the crystal-theoretic counterpart of $\eta_w^{-1}$ and $\dual_w$, i.e., 
$$
\cdual_w: \LB(w) \buildrel \sim \over \longrightarrow  \LB(w)
$$
(see \eqref{Eq: Dw and etaw}). For combinatorial descriptions of $\LB(w)$, we develop \emph{PBW parametrizations} $\tcP_\bfi(w)$ and \emph{string parametrizations} $\tcS_\bfi(w)$ associated with a reduced expression $\bfi$ of $w$, which gives the bijection 
$$
\psi_\bfi : \tcP_\bfi(w) \buildrel \sim \over \longrightarrow \tcS_\bfi(w)
$$ 
through the localized crystal $\LB(w)$ (see \eqref{Eq: PBW STR}). Theorem \ref{Thm: main} provides an explicit description of $\cdual_w$ in terms of $\psi_\bfi$ and the two upper-triangular matrices $M_\bfi$ and $N_\bfi$ defined in \eqref{Eq: N M}. Thus $\cdual_w$ arises essentially from the bijection $\psi_\bfi$ between two different parametrizations of $\LB(w)$.  
The key ingredient in the proof is to use the connections between $\tcP_\bfi(w)$ (resp.\ $\tcS_\bfi(w)$) and the set $\GR_\bfi(w)$ (resp.\ $\GL_\bfi(w)$) of all \emph{$g$-vectors}  $\gR_\bfi(x)$ (resp.\ $\gL_\bfi(x)$) of $x \in \tGup(w)$ coming from quiver Hecke algebras (see Section \ref{Sec: MC}).

As an application, we provide a combinatorial description of $\cdual_{x_t}(b)$ when $b$ is an element of the $*$-twisted minuscule crystal of classical finite type and $x_t$ is the longest element of $\weyl_t \backslash \weyl$ for a minuscule index $t \in I$ (see Section \ref{Sec: MR}). This allows us to compute the PBW and string data of the element $\cdual_{x_t}(b)$ in terms of (shifted) Young diagrams.
In the case of type $A_n$, we investigate further to provide a more detailed description of $\cdual_{x_t}(b)$. When the string data $\STR_\bfi(b)$ of $b $ is given as a \emph{rectangular vector},  its image under $\cdual_{x_t}$ is described in a simple manner (see Proposition \ref{prop: rule for D minuscule type A}). 
This description yields a closed formula (Theorem \ref{thm: period of minuscule type A}) for the \emph{periodicity} of $\cdual_{x_t}$ (and also $\eta_{x_t}$), which recovers the finite periodicity result \cite[Proposition 3.15]{Baur21} on Grassmannian cluster categories (see Remark \ref{rem: periodicity for A}). 
Using \textsc{SageMath}~\cite{sage}, we make further observations. Based on computations by \textsc{SageMath}, Section \ref{Sec: FO} presents several conjectures on (finite) periodicity in (co)minuscule cases, Coxeter element cases, and type $A$ parabolic cases.

This paper is organized as follows. 
In Section \ref{sec:preliminaries}, we review briefly quantum coordinate rings and their monoidal categorifications using quiver Hecke algebras. 
Section \ref{Sec: LB} introduces the notion of localized crystals and explains their PBW and string parametrizations, and  Section \ref{Sec: twist auto and LB} shows a crystal-theoretic description of $\cdual_w$. In Section \ref{Sec: MR}, we study further $\cdual_w$ in the case of  minuscule crystals, and in Section \ref{Sec: periodicity}, we investigate the periodicity of $\cdual_w$.

\medskip

\vskip 1em

{\bf Acknowledgments}
We are very grateful to Masaki Kashiwara, Myungho Kim, Se-jin Oh, Yoshiyuki Kimura, and Hironori Oya for valuable discussions and fruitful comments.

\vskip 2em

\subsection*{Convention}
Throughout this paper, we use the following convention.
\bnum
\item For a statement $P$, we set $\delta(P)$ to be $1$ or $0$ depending on whether $ P$ is true or not. In particular, we set $\delta_{i,j}=\delta(i =j)$. 

\item For $a\in \Z \cup \{ -\infty \} $ and $b\in \Z \cup \{ \infty \} $ with $a\le b$, we set 
\begin{align*}
& [a,b] =\{  k \in \Z \ | \ a \le k \le b\}, &&  [a,b) =\{  k \in \Z \ | \ a \le k < b\}, \allowdisplaybreaks\\
& (a,b] =\{  k \in \Z \ | \ a < k \le b\}, &&  (a,b) =\{  k \in \Z \ | \ a < k < b\},
\end{align*}
and call them \emph{intervals}. 
When $a> b$, we understand them as empty sets. For simplicity, when $a=b$, we write $[a]$ for $[a,b]$. 

\item For sets $A$ and $J$, we write $A^{J}$ for the product of copies of $A$ indexed by $J$.
We simply write $A^m := A^{[1,m]}$ if no confusion arises.
\ee

\vskip 2em

\section{Preliminaries} \label{sec:preliminaries}

\subsection{Quantum unipotent coordinate rings} \label{Sec: qcr}
In this subsection, we briefly review basics of quantum groups and their quantum unipotent coordinate rings. 

\begin{definition}
A quintuple $ (\cmA,\wlP,\Pi,\wlP^\vee,\Pi^\vee) $ is called a  (symmetrizable) {\it Cartan datum} if it
consists of
 a generalized  \emph{Cartan matrix} $\cmA=(a_{ij})_{i,j\in I}$,
 a free abelian group $\wlP$,
its dual $\wlP^{\vee} := \Hom_{\Z}( \wlP, \Z )$, the set $\Pi := \{ \alpha_i \in \wlP \mid i\in I \}$ of \emph{simple roots} and the set $\Pi^{\vee} := \{ h_i \in \wlP^\vee \mid i\in I\}$ of \emph{simple coroots} such that    
 \bnum
\item $\lan h_i, \alpha_j \ran = a_{ij}$ for $i,j \in I$,
\item $\Pi$ is linearly independent over $\Q$,
\item for each $i\in I$, there exists $\Lambda_i \in \wlP$ such that $\lan h_j,\Lambda_i \ran =\delta_{j,i}$ for all $j \in I$.
\item there is a symmetric bilinear 
form $( \cdot \, , \cdot )$ on $\wlP$ satisfying 
$(\al_i,\al_i)\in\Q_{>0}$ and 
$ \lan h_i,  \lambda\ran = {2 (\alpha_i,\lambda)}/{(\alpha_i,\alpha_i)}$. 
\ee
\end{definition}

Let $\prD$  be the set of \emph{positive roots}, and define 
 $ \rlQ^+ \seteq \bigoplus_{i \in I} \Z_{\ge 0} \alpha_i$ and $\rlQ^- := -\rlQ^+ $. We write $\weyl = \langle s_i \mid i\in I \rangle$ for the \emph{Weyl group} associated with $\cmA$, where $s_i$ is the $i$th simple reflection. For $w\in \weyl$, we denote by $\rR(w)$ the set of all reduced expression of $w$, i.e.,  
\begin{align} \label{Eq: reduced ex}
\rR(w) := \{ (i_1, i_2, \ldots, i_l) \in I^t \mid w = s_{i_1} s_{i_2} \cdots s_{i_l} \}
\end{align}
where $l = \ell(w)$. When $\cmA$ is of finite type, we write $w_\circ$ for the longest element of $\weyl$. For any $w = s_{i_1}s_{i_2} \cdots s_{i_m} \in \weyl$, we set 
$$
\supp(w) := \{ i_1, i_2, \ldots, i_m \}. 
$$ 
Note that $\supp(w)$ does not depend on the choice of reduced expressions of $w$.

Let $q$ be an indeterminate and define 
$\bfk \seteq \Q(q)$, $q_i\seteq q^{(\al_i,\al_i)/2}$ and
$$ 
[n]_{q_i}= \dfrac{q_i^n-q_i^{-n}}{q_i-q_i^{-1}}, \quad [n]_{q_i}! = \prod_{k=1}^n [k]_{q_i} \qtq
\begin{bmatrix}
n \\ m
\end{bmatrix}_{q_i} = \dfrac{[n]_{q_i}!}{[m]_{q_i}![n-m]_{q_i}!},
$$
where $i\in I$ and $n \ge m \in \Z_{\ge0}$. 
The \emph{quantum group} $U_q(\g)$ associated with the Cartan datum $(\cmA,\wlP,\Pi,\wlP^\vee,\Pi^\vee)$ is the $\bfk$-algebra
generated by $e_i,f_i$ $(i \in I)$ and $q^h$ $(h \in \wl^\vee)$ subject to certain defining relations (see \cite{HK02, LusztigBook} for details).
Let $U_q^-(\g)$ be the subalgebra of $U_q(\g)$ generated by $f_i$ $(i \in I)$ and $U_{\Zq}^-(\g)$ the $\Zq$-subalgebra of
$U_q(\g)$ generated by $f_i^{(n)} \seteq f_i^n/[n]_{q_i}!$ $(i \in I, n\in \Z_{>0})$. 
For an element $x$ in the \emph{weight space} $U_q^-(\g)_\beta$ of $U_q^-(\g)$, we write $\wt(x):=\beta \in \rlQ^-$.

The \emph{quantum unipotent coordinate ring} $\qA_q(\n)$ of $U_q(\g)$ is defined by 
$$
\qA_q(\n) \seteq \bigoplus_{\beta \in \rlQ^-} \qA_q(\n)_\beta, \quad \text{ where } \qA_q(\n)_\beta \seteq \Hom_{\bfk}(U_q^-(\g)_{\beta},\bfk),
$$
where the multiplication of $\qA_q(\n)$ is defined by using the comultiplication of $U_q(\g)$ (see \cite[Section 8]{KKKO18} for details).
There exists a distinguished basis $\Gup(\infty)$ of $\qA_q(\n)$, which is called the \emph{upper global basis} (or \emph{dual canonical basis}) of $\qA_q(\n)$. We refer the reader to \cite{K91,K93, L90, LusztigBook} and references therein for precise definitions (see also \cite[Section 8]{KKKO18}). 
We set $\qA_\Zq(\n)$ (resp.\ $\qA_{\Z[q]}(\n)$) to be the $\Zq$-lattice (resp.\ $\Z[q]$-lattice) of $\qA_q(\n)$ generated by the upper global basis $\Gup(\infty)$.
For any $\La\in \wlP^+$, $w,v \in \weyl$ with $v \le w$, we denote by 
$\qmD (w \La, v \La) \in \qA_q(\n)$ the \emph{quantum unipotent minor} associated with $w \La $ and $v\La$ (see \cite[Section 9]{KKKO18} for precise definition). Note that $\qmD (w \La, v \La)$ is contained in $\Gup(\infty)$.
For any $x, y \in \qA_q(\n) $, we write 
$$
x \eqq y \qquad \text{ if  $ x = q^t y$ for some integer $t$.}
$$
For any $A, B \subset A_q(\n)$, we write $A \eqq B $ if there is a bijection $f: A \buildrel \sim \over \rightarrow B$ such that $x \eqq f(x)$ for all $x\in A$.

Fix $w\in \weyl$ and take $\bfi=(i_1,\ldots,i_l) \in \rR(w)$. 
We set $w_{\le 0}= 1$, $w_{\le k} := s_{i_1}s_{i_2} \cdots s_{i_k}$ and $w_{<k} := w_{\le k-1}$ for $1\le k \le l$.
Let 
\begin{align} \label{Eq: beta_k}
\beta_{\bfi, k} := w_{<k}(\al_{i_k}),
\end{align} 
and write simply $\beta_{k}= \beta_{\bfi, k}$ if no confusion arises.
The set $\prD \cap w \nrD = \{ \beta_{ k}  \ | \ 1 \le k \le l \}$ has the \emph{convex} order given by
$$
\beta_{ a} \le_{\bfi} \beta_{  b} \quad \text{ for any } a \le b.
$$ 
We denote by 
$$
\{\Fup_{\bfi}(\beta_k)\}_{k\in [1,l]} \subset \qA_q(\n)
$$ 
the set of \emph{dual PBW vectors} with respect to $\bfi$, which is the dual of the set of \emph{PBW vectors} $\{ {\mathrm{F}}_{\bfi}(\beta_k) := T_{i_1} T_{i_2} \cdots T_{i_{k-1}} (f_{i_k}) \}_{k\in [1,l]} $, where $T_i$ are Lusztig's braid symmetries $T'_{i,-1}$ in the notation of \cite[Chapter 37]{LusztigBook}. Note that, when $i\ne j$,   
\begin{equation*}
\begin{aligned}
  T_i(f_j) = \sum_{r+s = -c_{i,j}}  (-1)^{ r} q_i^{r} f_{i}^{(s)} f_{j} f_{i}^{(r)}. 
\end{aligned}
\end{equation*}
Note that the dual PBW vectors $\Fup_{\bfi}(\beta_k)$ are contained in the upper global basis $\Gup(\infty)$.
We define $\qA_q(\n(w))$ to be the $\bfk$-subalgebra of $\qA_q(\n)$ generated by the dual PBW vectors $\{\Fup_{\bfi}(\beta_k)\}_{k\in [1,l]}$.
The algebra $\qA_q(\n(w))$ does not depend on the choice of $\bfi$ and respects the upper global basis, i.e., the intersection
$$
\Gup(w) := \qA_q(\n(w)) \cap \Gup(\infty) 
$$
forms a $\bfk$-linear basis of $A_q(\n(w))$ (\cite{Kimura12}).
For any $\bfa = (a_1, a_2, \ldots, a_l) \in \Z_{\ge0}^{l}$, we define 
$$
\Fup_{\bfi} (\bfa) = 
f_{\bfi, \epsilon}^{\up} (\beta_\ell) ^{\{a_\ell\}} f_{\bfi, \epsilon}^{\up} (\beta_{\ell-1}) ^{\{a_{\ell-1}\}} \cdots 	f_{\bfi, \epsilon}^{\up} (\beta_{1}) ^{\{a_{1}\}}
$$
where $ \Fup_{\bfi}(\beta_k)^{\{a\}} := q_{i_k}^{a(a-1)/2} \Fup_{\bfi}(\beta_k)^a$ for $k\in [1,l]$. Then the set $\{ \Fup_{\bfi} (\bfa)  \}_{\bfa \in \Z_{\ge0}^{l}}$ becomes a $\bfk$-basis of $ \qA_q(\n(w))$, which is called the \emph{dual PBW basis} of $\qA_q(\n(w))$ with respect to $\bfi$. Since there is a uni-triangular property between the dual PBW basis and the upper global basis, for any element $x \in \Gup(w)$, there exists $\bfa \in \Z_{\ge0}^l$ such that 
$$
\Fup_{\bfi} (\bfa) \equiv x \qquad \mod q \qA_{\Z[q]}(\n).
$$
We write $\PBW_\bfi (x) := \bfa$, which is called the \emph{PBW datum} of $x$ with respect to $\bfi$.

We assume that $\supp(w) = I$.
For any $i\in I$, we define the quantum unipotent minor
\begin{align} \label{Eq: frozen}
\qmC^w_{i} := \qmD ({w \La_{i}, \La_{i}} ).
\end{align}
Note that the minor $\qmC^w_{i}$ is contained in $\Gup(w)$ and it commutes with any element of $\Gup(w)$ up to a multiple of $q$. We set 
$\tqA_q(\n(w))$ to be the algebra obtained by localizing $\qA_q(\n(w))$ by the $q$-central elements $\{ \qmC^w_i \mid i\in I \}$. Then the upper global basis $\tGup (w)$ of the localized algebra $\tqA_q(\n(w))$ is determined by 
\begin{align*} 
\tGup (w) = \{ q^{a(x,\bfa)} x \cdot  \qmC^w(\bfa) \mid x \in \Gup(w),\ \bfa \in \Z^{I}  \}, 
\end{align*}
where 
\begin{align} \label{Eq: C(a)}
\qmC^w(\bfa) := \prod_{i\in I} (\qmC^w_i)^{a_i}
\end{align}
and $a(x,\bfa)$ is the integer such that $q^{a(x,\bfa)} x \cdot  \qmC^w(\bfa)$ becomes \emph{dual bar-invariant}.
We simply write $\qmC_i$ and $\qmC(\bfa)$ instead of $\qmC^w_i$ and $\qmC^w(\bfa)$ if no confusion arises.

Let $\eta_w$ be the \emph{quantum twist automorphism} on the localized algebra $\tqA_q(\n(w))$ introduced in \cite{KiOy21} (see \cite[Theorem 6.1]{KiOy21} for precise definition). 
The twist automorphism $\eta_w$ is given by 
\begin{align*}
\eta_w(\qmD (v \La, \La))  \eqq
\qmD (w \La, \La)^{-1} \qmD(w\La, v\La) \qquad 
\text{ for any $v \in \weyl$ with $v \le w$ and $\La \in \wlP^+$.}
\end{align*}
 Moreover, the twist automorphism $\eta_w$ permutes the upper global basis. 
\begin{theorem} [{\cite[Theorem 6.1]{KiOy21}}] \label{Thm: tw tGup}
The twist automorphism $\eta_w$ restricts to a permutation on $\tGup (w)$.
\end{theorem}

\subsection{Quiver Hecke algebras} \label{Sec: QHA}

 For $i,j\in I$, we choose polynomials
$\qQ_{i,j}(u,v) $ in two variables $u,v$ over a field $\bR$ such that
$\qQ_{i,j}(u,v) = \qQ_{j,i}(v,u)$ and 
\begin{align*}
\qQ_{i,j}(u,v) =\bc
                   \sum\limits
_{p(\alpha_i , \alpha_i) + q(\alpha_j , \alpha_j)=-2(\alpha_i , \alpha_j)} t_{i,j;p,q} u^pv^q &
\text{if $i \ne j$,}\\[1.5ex]
0 & \text{if $i=j$,}
\ec
\end{align*}
where $t_{i,j;-a_{ij},0} \in  \bR^{\times}$.
For $\beta\in \rlQ_+$ with $ \Ht(\beta)=n$, we set
$$
I^\beta\seteq  \left\{\nu=(\nu_1, \ldots, \nu_n ) \in I^n \mid \sum_{k=1}^n\alpha_{\nu_k} = \beta \right\}.
$$
The {\em quiver Hecke algebra} $R(\beta)$ associated with the generalized Cartan matrix $\cmA$, $\beta\in\rlQ_+$ and $(\qQ_{i,j}(u,v))_{i,j\in I}$
is the $\bR$-algebra generated by
$$
\{e(\nu) \mid \nu \in I^\beta \}, \; \{x_k \mid 1 \le k \le n \},
\; \{\tau_l \mid 1 \le l \le n-1 \}
$$
satisfying certain defining relations (see \cite{KL1, KL2, R08} and see also \cite[Section 2]{KKKO18} for precise definition).
The algebra $R(\beta)$ has a natural $\Z$-graded algebra structure as follows:
\begin{align*}
\deg(e(\nu))=0, \quad \deg(x_k e(\nu))= ( \alpha_{\nu_k} ,\alpha_{\nu_k}), \quad  \deg(\tau_l e(\nu))= -(\alpha_{\nu_{l}} , \alpha_{\nu_{l+1}}).
\end{align*}
Let $R(\beta)\gmod$ be the category of finite-dimensional graded $R(\beta)$-modules and set $R\gmod := \bigoplus_{\beta\in \rlQ^+} R(\beta)\gmod$. For a module $M \in R(\beta)\gmod$, the dual space  
\begin{align} \label{Eq: star}
M^\star := \mathrm{Hom}_\bR(M, \bR)
\end{align}
admits an $R(\beta)$-module structure (see \cite[Section 2]{KKKO18} for example).

For any $\nu\in I^\beta$ and $\nu'\in I^{\beta'}$, let $e(\nu, \nu')$ be the idempotent generator associated with the concatenation
$\nu\ast\nu'$ of $\nu$ and $\nu'$, and define
$$
e(\beta, \beta') \seteq \sum_{\nu \in I^\beta, \nu' \in I^{\beta'}} e(\nu, \nu')
{\hs{2ex} \in R(\beta+\beta').}
$$
For $M \in R(\beta)\gmod$ and $N \in R(\gamma)\gmod$, we define 
$$
M \conv N := R(\beta+\gamma)e(\beta, \gamma) \tens_{R(\beta)\tens R(\gamma)} (M \tens N).
$$
We denote by $\conv$ the \emph{convolution product} on $R\gmod$, which makes $R\gmod$ a monoidal category (see \cite[Section 2]{KKKO18} for details). 
We denote by $M \hconv N$ the \emph{head} of the convolution $M \circ N$. A simple module $X$ is \emph{real} if $X \conv X$ is simple. If $X$ is not real, $X$ is said to be \emph{imaginary}. If $X \conv Y$ is simple for simple $R$-modules $X$ and $ Y$, we say that $X$ and $Y$ \emph{strongly commute}.

Let $i\in I$ and $n\in \Z_{\ge0}$. We denote by $\ang{i^n}$ the self-dual simple $R(n\al_i)$-module. We simply write $\ang{i} = \ang{i^1}$. The \emph{character} of $M \in R(\beta)\gmod$ is defined by 
$$
\qch(M) := \sum_{\nu \in I^{\Ht(\beta)}} \qdim (e(\nu)M) \nu,
$$
where $\qdim(V)$ is the graded dimension of a $\Z$-graded vector space $V$.
For $M \in R(\beta)\gmod$, we define
\begin{align} \label{Ei divEi}
E_i(M) := e(\al_i, \beta-\al_i) M \qtq 
E_i^{(n)}(M) :=  \mathrm{Hom}_{R(n\alpha_i)}(P(i^{n}), e(n\al_i, \beta-n\al_i)M ),
\end{align}
where $P(i^{n})$ is the projective cover of $(i^n)$. Note that $E_i$ categorifies the $q$-derivation $e_i'$ on $\qA_q(\n)$ (\cite[(3.4.4)]{K91}).
For a simple module $X$ in $R(\beta)\gmod$, we define 
\begin{align} \label{Eq: ep in Rgmod}
\ep_i(X) := \max\{ k \in \Z_{\ge0} \mid e(k\al_i, \beta-k\al_i)X \ne 0 \}.
\end{align}

For any $w\in \weyl$, we define $\cC_{w}$ to be the  full subcategory of $R\gmod$ consisting of objects $M$ such that
\begin{align*} 
\gW(M) \subset \Sp( \prD \cap w \nrD ),
\end{align*}
where $\gW(M) = \{  \gamma \in  \rlQ_+ \cap (\beta - \rlQ_+)  \mid  e(\gamma, \beta-\gamma) M \ne 0  \}$. 

For a monoidal subcategory $C$ of $R\gmod$, we denote by $K_q(C)$ the Grothendieck ring of $C$. The \emph{grading shift functor} $q$ gives a $\Zq$-module structure to $K_q(C)$. 
We also denote by $K(C)$ the Grothendieck ring of the ungraded category $C_0$ obtained from $C$ by forgetting the grading.
Note that $K(C)$ is equal to the specialization of $K_q(C)$ at $q=1$. 
We write $[X]$ for the element in $K(C)$ (resp.\ $K_q(C)$) corresponding to an object $X \in C$.
If no confusion arises, then we write $X$ instead of $[X]$.
Then we have the categorification of $\qA_q(\n(w))$ (see \cite{KL1, KL2, R08} and see also \cite{KKKO18, KKOP18}).

\begin{prop} \label{Prop: categorification}
There exists an isomorphism between 
\begin{align} \label{Eq: categorification Aqnw}
K_q(\cC_w) \simeq \qA_\Zq(\n(w)).
\end{align}
\end{prop}

For $\La \in \wlP^+$ and $w,v \in \weyl$ with $v \le w$, we denote by $ \dM(w\La, v\La)$ the \emph{determinantial module}, which corresponds to the quantum minor $\qmD (w \La, v\La)$ via the categorification \eqref{Eq: categorification Aqnw} (See \cite{KKKO18} and see also \cite[Section 4]{KKOP18} for details). 
We assume that $\supp(w) = I$. Let $\tcC_w$ be the \emph{localization} of the category $\cC_w$ by the determinantial modules $\dM(w\La_i, \La_i)$ ($i\in I$) (see \cite[Section 5]{KKOP21C} for precise definition). By construction, we have the isomorphism 
\begin{align} \label{Eq: categorification localized}
K_q(\tcC_w) \simeq \tqA_\Zq(\n(w)).
\end{align}
We denote by $\Gs(\infty)$ (resp.\ $\Gs(w)$, $\tGs(w)$) the set of the isomorphic classes of self-dual simple objects in $R\gmod$ (resp.\ $\cC_w$, $\tcC_w$) with respect to $\star$ (\eqref{Eq: star}). Note that $\tGs(w)$ and $\tGup(w)$ coincide via the categorification when $R$ is of symmetric type, but they are different in general.

It is shown in \cite{KKOP21C, KKOP23} that the localized category $\tcC_w$ is rigid. Thus there exists the right dual functor of $\tcC_w$, which is denoted by $\dual_w$. Note that $\dual_w^{-1}$ is the left dual on $\tcC_w$. We simply write $\dual$ for $\dual_w$ if no confusion arises. Note that $\dual$ is a permutation of $\tGs(w)$ (if the grading is ignored). The following proposition was proved in Appendix B of \cite{IO21_arXiv2}. Note that Appendix B of \cite{IO21_arXiv2} is not included in the published version \cite{IO21}.  
The anti-automorphism $\sigma : R\gmod \buildrel \sim \over \rightarrow R\gmod$ is defined by $\sigma[M] = q^{-(\beta,\beta)/2}[M^\star]$ for $M \in R(\beta)\gmod$.

\begin{prop}[{\cite[Theorem B.22]{IO21_arXiv2}}] \label{Prop: eta_w=dual^-1}
We identify  $K_q(\tcC_w) $ with $ \tqA_\Zq(\n(w))$ via the categorification \eqref{Eq: categorification localized}. 
\bni
\item 
As algebra automorphisms, we have
$$
\eta_w = [\dual^{-1} \circ \sigma].
$$
\item For any $x \in \tGup(w)$ (resp.\ $x \in \tGs(w)$), we have
$$
\eta_w(x) \eqq [\dual^{-1}](x).
$$
This means that $\eta_w$ coincides with $[\dual^{-1}]$ as permutations of $\tGup(w)$ and $\tGs(w)$ 
\ee
\end{prop}

\subsection{Monoidal categorification} \label{Sec: MC}

In this subsection, we briefly review the monoidal categorification of the algebra $\qA_q(\n(w))$ using quiver Hecke algebras (\cite{KKKO18}). We assume that $\supp(w) = I$.

It is known that the algebra $\qA_q(\n(w))$ has a (quantum) cluster algebra structure  (\cite{GLS11, GLS13}). 
Let us fix $\bfi=(i_1, i_2,\ldots,i_l) \in \rR(w)$ and define 
\begin{equation} \label{Eq: k+ k-} 
\begin{aligned}
    k^+ &:= \min (  \{ l+1 \} \cup \{  k+1 \le j \le l \mid i_j = i_k \}  ), \\
    k^- &:= \max (  \{ 0 \} \cup \{  1 \le j \le k-1 \mid i_j = i_k \}  )
\end{aligned}
\end{equation}
for $1 \le k \le l$.
We define $J := \{ 1,2, \ldots, l \}$,  $J_\fr := \{j\in J \mid j^+ = l+1 \}$ and $J_\uf := J \setminus J_\fr$, and set the quantum unipotent minor
$\qmD_k^{\bfi} := \qmD  ({w_{\le k} \varpi_{i_k}, \varpi_{i_k}} ) $ for $1 \le k \le l$.
The \emph{exchange matrix} $\tB = (b_{s,t})_{s \in J, t\in J_\uf}$ associated with $\bfi$ is given by 
\begin{align*} 
    b_{s,t} = 
    \begin{cases}
        1 & \text{ if } s = t^+, \\
         a_{i_s, i_t}  & \text{ if } t < s < t^+ < s^+, \\
        -1 & \text{ if } s^+ = t, \\
         -a_{i_s, i_t}  & \text{ if } s < t < s^+ < t^+, \\
        0 & \text{ otherwise,}
    \end{cases}
\end{align*}
where $a_{i,j}$ is the $(i,j)$-entry of the generalized Cartan matrix $\cmA$. Then the pair 
$$
\seed_\bfi := ( \{ \qmD_k^{\bfi} \}_{k\in J}, \tB) 
$$
is the GLS seed of the algebra $\qA_q(\n(w))$ (\cite{GLS11, GLS13}).
Note that the frozen variables $\{ \qmD_k^{\bfi} \}_{k\in J_\fr}$ coincide with the $q$-central elements $\{\qmC^w_{i}\}_{i\in I}$ (see \eqref{Eq: frozen}).
Let $\dM_k^{\bfi} := \dM  ({w_{\le k} \varpi_{i_k}, \varpi_{i_k}} )$  for $1 \le k \le l$.
We now consider the corresponding \emph{monomial seed} of $\cC_w$ 
\begin{align} \label{Eq: m seed}
 ( \{ \dM_k^{\bfi} \}_{k\in J}, \tB),
\end{align}
which is also denoted by the same notation $\seed_\bfi$.
When $R$ is of symmetric type, it was proved in \cite{KKKO18} that the initial seed \eqref{Eq: m seed} \emph{admits} successive mutations in all directions, which tells that $\cC_w$ is a monoidal categorification of $\qA_q(\n(w))$. However, it is still open when $R$ is of non-symmetric case. 

Since the bases $\tGup(w)$ and $\tGs(w)$ of $\tqA_q(\n(w))$ are \emph{pointed} and $\emph{copointed}$ with respect to $\seed_\bfi$, the two kind of $g$-vectors $\gL_\bfi(x)$ and $\gR_\bfi(x)$, which are \emph{degree} and \emph{codegree} with respect to the \emph{dominance order} arising from $\seed_\bfi$, are well-defined for $x\in \tGup(w)$ (resp.\ $x\in \tGs(w)$) (see \cite{Qin17, Qin20} for $\tGup(w)$, \cite{KK19, KKOP24} for $\tGs(w)$, and see also \cite[Section 5.3]{KKOP24} for details).  

\begin{prop} [{\cite[ Lemma 3.9, Proposition 4.4 and Corollary 5.14]{KKOP24}}] \label{Prop: GrGl for Rgmod}
Let $X \in \Gs(w)$.
\bni
\item There exist $\bfa, \bfb \in \Z_{\ge0}^{J}$ such that 
$$
\dM^{\bfi}(\bfa) \hconv  X  \simeq \dM^{\bfi}(\bfb).
$$
Moreover we have $\gL_\bfi(X) = \bfb-\bfa$.
\item There exist $\bfa, \bfb \in \Z_{\ge0}^{J}$ such that 
$$
X \hconv \dM^{\bfi}(\bfa) \simeq \dM^{\bfi}(\bfb).
$$
Moreover we have $\gR_\bfi(X) = \bfb-\bfa$.
\ee
Here we set $\dM^{\bfi}(\bfc) = \conv_{k\in J} (\dM^\bfi_k)^{\conv c_k}$ for $\bfc = (c_k)_{k\in J } \in \Z_{\ge0 }^{J}$.
\end{prop}

The following lemma will be used Section \ref{Sec: twist auto and LB} for comparing g-vectors for $\tGup(w)$ and $\tGs(w)$.

\begin{lemma} \label{Lem: ei and seed}
Let $i\in I$.
\bni
\item Let $M, N$ be simple $R$-modules.
\bna
\item $\ep_i(M\conv N) = \ep_i(M) + \ep_i(N)$.
\item If $M$ is real and commutes strongly with the simple $R(\al_i)$-module $\ang{i}$, then 
$\ep_i(M\hconv N) = \ep_i(M) + \ep_i(N)$.
\ee
\item Let $\bfi=(i_1, i_2,\ldots,i_l) \in \rR(w)$ and suppose that $i=i_1$. We set $\bfi' := (i_2,\ldots,i_l) \in \rR(s_iw)$.
Then we have $\dM_1^\bfi = \ang{i}$ and 
$$
\dM_k^{\bfi'} = E_i^{(n_k)} ( \dM_{k+1}^{\bfi} )
\qquad \text{for $k\in [1,l-1]$,}
$$
where $n_k = \langle h_i, s_{i_2} \cdots s_{i_k} \La_{i_k} \rangle$.
\ee
\end{lemma}
\begin{proof}
(i) 
The first one follows from the shuffle lemma \cite[Lemma 2.20]{KL1} and the second follows from \cite[Lemma 3.2]{KKOP18}.

(ii) Since $[\dM_k^\bfi] = \qmD_k^\bfi$, it follows from the properties of quantum unipotent minors (see \cite[Lemma 1.7]{KKOP18} for example).
\end{proof}

\vskip 2em 

\section{Localized crystals} \label{Sec: LB}
In this section, we introduce localized crystals and show PBW and string parametrizations. 
We refer the reader to \cite{K91, K93, K95} for details on crystals. 

\begin{definition}
	A \emph{crystal} is a set $B$ endowed with  maps $\wt\col  B \rightarrow \wlP$,  $\varphi_i$,  $\ep_i \col  B \rightarrow \Z \,\sqcup \{ -\infty \}$
	and $\te_i$, $\tf_i \col  B \rightarrow B\,\sqcup\{0\}$ 
 for all $i\in I$ which satisfy the following axioms: 
	\bna
	\item $\varphi_i(b) = \ep_i(b) + \langle h_i, \wt(b) \rangle$,
	\item $\wt(\te_i b) = \wt(b) + \alpha_i$ if $\te_i b\in B$, and 
$\wt(\tf_i b) = \wt(b) - \alpha_i$ if $\tf_i b\in B$, 
	\item for $b,b' \in B$ and $i\in I$, $b' = \te_i b$ if and only if $b = \tf_i b'$,
	\item for $b \in B$, if $\varphi_i(b) = -\infty$, then $\te_i b = \tf_i b = 0$,
	\item if $b\in B$ and $\te_i b \in B$, then $\ep_i(\te_i b) = \ep_i(b) - 1$ and $ \varphi_i(\te_i b) = \varphi_i(b) + 1$,
	\item if $b\in B$ and $\tf_i b \in B$, then $\ep_i(\tf_i b) = \ep_i(b) + 1$ and $ \varphi_i(\tf_i b) = \varphi_i(b) - 1$.
\ee
\end{definition}
We set $B(\infty)$ to be the \emph{infinite crystal} of the negative half $U_{q}^-(\g)$. 
Since the crystal can be understood as the specialization of the upper global basis $\Gup(\infty)$ at $q=0$, the basis $\Gup(\infty)$ can be written as 
$$
\Gup(\infty) = \{ \Gup(b) \mid b\in B(\infty) \}.
$$
It is shown in \cite{LV11} that the basis $\Gs(\infty)$ has a crystal structure which is isomorphic to $B(\infty)$. Thus we have 
$$
\Gs(\infty) = \{ L(b) \mid b\in B(\infty) \},
$$
where $L(b)$ is the self-dual simple $R$-module corresponding to $b$.

The $\bfk$-antiautomorphism 
$$
*: U_q(\g) \rightarrow U_q(\g) \quad \text{defined by }  e_i \mapsto e_i,\ f_i \mapsto f_i, \ \  q^h \mapsto q^{-h}
$$
gives the $*$-twisted crystal operators to $B(\infty)$ as follows:
$$
\tes_i := * \circ \te_i \circ * \qtq \tfs_i := * \circ \tf_i \circ *.
$$
For $b\in B(\infty)$, we write 
$$
\tem_i(b) = \te_i^{m}(b)\qtq \tesm_i(b) = {\te_i}^{* n}(b),
$$
where $m = \ep_i(b)$ and $n = \eps_i(b)$. We denote by $\one$ the highest weight vector in $B(\infty)$.
Let $(e_i')^{(n)} := (e_i')^n / [n]_{q_i} ! $, where $e_i'$ is the $q$-derivation (\cite[(3.4.4)]{K91}), and let $E_i^{(n)}$ be the counterpart of $(e_i')^{(n)}$ given in \eqref{Ei divEi}.

\begin{lemma} [{\cite[Lemma 5.1.1]{K93} and \cite[Section 2.5.1]{LV11}}] \label{Lem: ei' max}
Let $b \in B(\infty)$ and $n = \ep_i(b)$. Then we have 
$$
(e_i')^{(n)} (\Gup(b)) = \Gup(\te_i^n b) \qtq 
E_i^{(n)} (L(b)) = L(\te_i^n b).
$$  
\end{lemma}
Note that $\ep_i(b) = \ep_i(L(b))$ (see \eqref{Eq: ep in Rgmod}) for $b\in B(\infty)$.

Let $w$ be a fixed element in the Weyl group $\weyl$ and assume that $\supp(w) = I$. We define 
$$
B(w) := \{ b\in B(\infty) \mid \Gup(b) \in \qA_q(\n(w)) \}.
$$
Note that $B(\infty)$ coincides with $B(w_\circ)$ when $U_q(\g)$ is of finite type. 
For any $i\in I$, we denote by $\fz^w_i $ the element in the crystal $B(\infty)$ determined by
\begin{align} \label{Eq: fzwi}
\Gup(\fz^w_i) = \qmC_{i} = \qmD(w\La_i, \La_i). 
\end{align}
We simply write $\fz_i = \fz^w_i$ if no confusion arises.
Let $(b, \bfa), (b', \bfa')  \in B(w) \times \Z^{\oplus I}$ and write 
$\bfa = (a_i)_{i\in I}$ and $\bfa' = (a_i')_{i\in I}$. 
We define 
\begin{align} \label{Eq: equiv rel}
(b, \bfa) \sim (b', \bfa') 
\end{align}
if there exists a positive integer $N$ such that $N + \min_{i\in I}\{ a_i, a_i' \}  \ge 0$ and
\begin{align} \label{Eq: LB and Gup}
\Gup(b) \prod_{i\in I}\Gup( \fz_i)^{N+a_i}  \eqq  \Gup(b') \prod_{i\in I}\Gup( \fz_i)^{N+a_i'}.
\end{align}
It is easy to see that $\sim$ is an equivalence relation on $B(w) \times \Z^{\oplus I}$. We define $\LB(w)$ to be the set of all equivalence classes by $\sim$ in $B(w) \times \Z^{\oplus I}$, i.e,  
\begin{align} \label{Eq: localized crystal}
\LB(w) := B(w) \times \Z^{\oplus I} / \sim. 
\end{align}
For any $b\in B(w)$ and $\bfa = (a_i)_{i\in I} \in \Z^{\oplus I}$, we simply write 
$$
b \cdot \fz(\bfa) = b \cdot \prod_{i\in I}\fz_i^{a_i} := [(b, \bfa)] \in \tB(w)
$$
if no confusion arises. 
When $w = w_\circ$, we simply write $\LB(\infty)$ instead of $\LB(w_\circ)$.
\begin{definition} \label{Def: localized crystal}
We call $\LB(w)$ the \emph{localized crystal} of $B(w)$ by the elements $\{ \fz_i\}_{i\in I}$.
\end{definition}

\begin{remark} \label{Rmk: LB}
Nakashima introduced the notion of localized crystals in studying the categorical crystal structure on $\tqA_q(\n)$ with cellular crystals when $\cmA$ is of finite type (\cite{Nakashima22}).  
Although categorical crystals and cellular crystals are dealt mainly in \cite{Nakashima22}, 
it is not difficult to obtain Definition \ref{Def: localized crystal} by abstracting the result in \cite{Nakashima22}. 
%The localized crystals over a general setting is studied by Kashiwara and Nakashima (see \cite{KN25_arXiv1}).
\end{remark}

For any $x,y \in \LB(w)$, we write 
\begin{align} \label{Eq: equiv for fr}
x \eqfr y \quad \text{ if and only if \quad $x = y \cdot \fz(\bfa)$ for some $\bfa \in \Z^{\oplus I}$,}  
\end{align}
which means that $x$ and $y$ are same up to a multiple of frozens.
 For any element $b \cdot \fz(\bfa) \in \LB(w)$, we define the element $\tGup(b \cdot \fz(\bfa)) \in \tGup (w) $ by
$$
\tGup(b \cdot \fz(\bfa)) \eqq   \Gup(b)\cdot \qmC(\bfa),
$$
where $\qmC(\bfa)$ is given in \eqref{Eq: C(a)}.
Thus the map  $\tGup: \LB(w) \longrightarrow \tGup(w)$ is bijective and
$$
\tGup(w) = \{\tGup(x) \mid x \in \LB(w) \}.
$$
In the same manner, we define the element $L(b \cdot \fz(\bfa)) \in \tGs (w) $ by
\begin{align} \label{Eq: L(x)}
[L(b \cdot \fz(\bfa))] \eqq   [L(b)]\cdot \qmC(\bfa). 
\end{align}
Note that $[\dM_i] = \qmC_i$ for $i\in I$.
Thus we have
$$
\tGs(w) = \{ [L(x)] \mid x \in \LB(w) \}.
$$

By Theorem \ref{Thm: tw tGup} and Proposition \ref{Prop: eta_w=dual^-1}, we can define the permutation $\cdual_w$ on $\LB(w)$ such that the following diagram commutes:
\begin{equation} \label{Eq: Dw and etaw}
\begin{aligned} 
\xymatrix{
 \tGs(w) \ar[d]_{[ \dual_w^{-1}] }  &&  \ar[ll]_\sim  \LB(w)  \ar[rr]^\sim \ar[d]_{\cdual_w^{-1}} && \tGup(w) \ar[d]_{\eta_w} \\
\tGs(w) && \ar[ll]_\sim \LB(w) \ar[rr]^\sim && \tGup(w)
}.
\end{aligned}
\end{equation}
Note that $\cdual_w$ corresponds to the right dual functor $\dual_w$ of $\tcC_w$.

From now until the end of this section, we fix a reduced expression 
$$
\bfi = (i_1, i_2, \ldots, i_m) \in R(w).
$$
In the following subsections, we introduce two parametrizations of $\LB(w)$. 

\subsection{PBW parametrization}

For any $b\in B(w)$, we define $\PBW_\bfi (b) := \PBW_\bfi( \Gup(b) ) $ and set 
$$
\cP_\bfi(w) := \{ \PBW_\bfi(b) \mid b \in B(w) \} \subset \Z_{\ge0}^{[1,m]}
$$
so that the map $\PBW_\bfi: B(w) \rightarrow \cP_\bfi(w)$ is bijective (\cite[Chapter 40]{LusztigBook}). 

For each $j\in I $, we define $P_j := \PBW_\bfi(\fz_j)$ for any $j\in I$.
It is known that 
$$
P_j = (a_t)_{t\in [1,m] } \qquad \text{ by $a_t = \delta_{i_{t}, j} $ for $ t\in [1,m]$} 
$$ 
(see \cite[Section 6.1.3]{Kimura12} and \cite[Corollary 4.7]{KKOP18} for example).
We now extend $\PBW_\bfi$ to the localized crystal $\LB(w)$ as follows: 
for any $x = b \cdot \prod_{i\in I}\fz_i^{a_i} \in \LB(w)$,
$$
\PBW_\bfi(x) = \PBW_\bfi(b) + \sum_{i\in I} a_i P_i \in \Z^{[1,m]}.
$$
Note that, by \eqref{Eq: LB and Gup}, $\PBW_\bfi(x)$ does not depend on the choice of expression of $x$.
If we define 
$$
\tcP_\bfi(w) := \cP_\bfi(w) + \sum_{k\in I} \Z P_k \subset \Z^{[1,m]},
$$
we have the bijection 
\begin{align} \label{Eq: LB PBW}
\pbwP_\bfi :  \LB(w) \buildrel \sim \over \longrightarrow \tcP_\bfi(w), 
\qquad x \mapsto \PBW_\bfi(x).
\end{align}
The set $\tcP_\bfi(w)$ is called the \emph{PBW parametrization} of the localized crystal $\LB(w)$ with respect to the reduced expression $\bfi$ of $w$.

\subsection{String parametrization}
 For any $b \in B(w)$, we define 
$$ 
\STR_{\bfi} (b) := (t_k)_{k\in [1,m]},
$$ 
where $t_k = \ep_{i_k}( \te_{i_{k-1}}^{t_{k-1}} \cdots \te_{i_2}^{t_2}\te_{i_1}^{t_1}(b))$
for $k\in [1,m]$, and set 
$$
\cS_\bfi(w) := \{ \STR_\bfi(b) \mid b\in B(w) \}.
$$

\begin{lemma} \label{Lem: STR wrt w}
Let $\bfi = (i_1, i_2, \ldots, i_m) \in R(w)$.  
\bni
\item For any $b\in B(w)$, we have 
$$
\te_{i_{m}}^{t_{m}} \cdots \te_{i_2}^{t_2}\te_{i_1}^{t_1}(b) = \one,
$$
where $\STR_{\bfi} (b) = (t_k)_{k\in [1,m]}$.
\item The map 
$
\STR_{\bfi}: B(w) \longrightarrow \cS_\bfi(w)
$
is bijective.
\ee
\end{lemma}
\begin{proof}
(i) 
Let $  F_{\bfi, k} (k\in [1,m])$ be  the dual PBW root vector associated with $\bfi$, and let $f_{\bfi, k}$ be the corresponding element in $B(w)$, i.e., $\Gup(f_{\bfi, k}) = F_{\bfi, k}$. Since the dual PBW vector $F_{\bfi, k} $ is the quantum unipotent minor $ \qmD ({w_{\le k} \La_{i_k}, w_{< k} \La_{i_k}} ) $, we have 
$$
  \tem_{i_{m}} \cdots \tem_{i_2}\tem_{i_1}( f_{\bfi, k} ) = \one
  \qquad \text{for all $k\in [1,m]$}
$$
by \cite[Lemma 1.7]{KKOP18}. 
Lemma \ref{Lem: ei' max} says that 
$$
(e'_{i_{m}})^{\max} \cdots (e'_{i_2})^{\max} (e'_{i_1})^{\max}( F_{\bfi, k} ) \in \bfk,
$$
where $(e_i')^{\max} ( x) := (e_i')^{(\ep_i(x))} (x) $ for $x \in \Gup(\infty)$.

Since $e_i'$ is the $q$-derivation (\cite[(3.4.4)]{K91}) and the dual PBW vectors $F_{\bfi, k}$ ($k\in [1,m]$) generate the algebra $A_q(\n(w))$,  we conclude that 
$$
  \tem_{i_{m}} \cdots \tem_{i_2}\tem_{i_1}( b ) = \one
  \qquad \text{for any $b\in B(w)$.}
$$

(ii) It follows from (i) directly.
\end{proof}

\begin{remark}
Lemma \ref{Lem: STR wrt w}  can be proved by using \cite[Proposition 3.3]{KKOP23} in terms of quiver Hecke algebras.\end{remark}

For each $j\in I $, we define $S_j := \STR_\bfi(\fz_j)$.
It follows from \cite[Lemma 1.7]{KKOP18} that $ S_j = (t_k)_{k\in [1,m] } $ is given as
$$
t_k =   \langle h_{i_k}, s_{i_{k+1}} \cdots s_{i_{m-1}} s_{i_{m}} \La_{j} \rangle \qquad \text{ for } k\in [1,m].
$$
We now extend $\STR_\bfi$ to the localized crystal $\LB(w)$ as 
$$
\STR_\bfi(x) = \STR_\bfi(b) + \sum_{i\in I} a_i S_i \in \Z^{[1,m]}
\qquad \text{ for any $x = b \cdot \prod_{i\in I}\fz_i^{a_i} \in \LB(w)$}.
$$
By \eqref{Eq: LB and Gup}, $\STR_\bfi(x)$ does not depend on the choice of expression of $x$.
Define 
$$
\tcS_\bfi(w) := \cS_\bfi(w) + \sum_{k\in I} \Z S_k \subset \Z^{[1,m]},
$$
which gives a bijection 
\begin{align} \label{Eq: LB STR}
\strS_\bfi :  \LB(w) \buildrel \sim \over \longrightarrow \tcS_\bfi(w), 
\qquad x \mapsto \STR_\bfi(x).
\end{align}
The set $\tcS_\bfi(w)$ is called the \emph{string parametrization} of the localized crystal $\LB(w)$ with respect to the reduced expression $\bfi$ of $w$.

\vskip 2em

\section{Twist automorphisms and localized crystals} \label{Sec: twist auto and LB}

In this section, we give a description of $ \cdual_w$ in terms of localized crystals. 

Let $w\in \weyl$ and assume that $\supp(w) = I$. We fix 
$$
\bfi = (i_1, i_2, \ldots, i_m) \in \rR(w).
$$
Recall the bijections $\pbwP_\bfi: \LB(w) \buildrel \sim \over \longrightarrow \tcP_\bfi(w)$ and $\strS_\bfi: \LB(w) \buildrel \sim \over \longrightarrow \tcS_\bfi(w)$ given in \eqref{Eq: LB PBW} and \eqref{Eq: LB STR} respectively. We define the bijection $\psi_\bfi :=  \strS_\bfi \circ \pbwP_\bfi^{-1}  $, i.e., 
\begin{align} \label{Eq: PBW STR}
\psi_\bfi : \tcP_\bfi(w) \buildrel \sim \over \longrightarrow \tcS_\bfi(w), \qquad  P + \sum_{i\in I} a_i P_i \mapsto \psi_\bfi(P) +  \sum_{i\in I} a_i S_i.
\end{align}
Note that $\psi_\bfi$ is not $\Z$-linear.

We now recall two kind of $g$-vectors $\gR_\bfi $ and $\gL_\bfi$ given in Section \ref{Sec: MC}. For any $x \in \LB(w)$, we set 
$$
\gL_\bfi(x) := \gL_\bfi(\Gup(x)) \qtq \gR_\bfi(x) := \gR_\bfi(\Gup(x))
$$ and define 
$$
\GL_\bfi(w) := \{ \gL_\bfi(x) \mid b\in \LB(w)\} \qtq \GR_\bfi(w) := \{ \gR_\bfi(x) \mid b\in \LB(w)\}.
$$
Note that 
\begin{align} \label{Eq: GrGlZm}
\GL_\bfi(w) = \Z^{[1,m]} = \GR_\bfi(w)
\end{align}
and the maps 
$$
\gL_\bfi: \LB(w) \buildrel \sim \over \longrightarrow \GL_\bfi(w) \qtq
\gR_\bfi: \LB(w) \buildrel \sim \over \longrightarrow \GR_\bfi(w)
$$
are bijective (\cite{Qin17} and \cite[Corollary 3.17 and Corollary 3.18]{KK19}).

Define $m\times m$ matrices  
\begin{align}\label{Eq: N M}
N_\bfi = (n_{p,q})_{p,q \in [1,m]} \qtq 
M_\bfi = (m_{p,q})_{p,q \in [1,m]}
\end{align}
by
\begin{align*}
n_{p,q} = 
\begin{cases}
1 & \text{ if } p=q,  \\
-1 & \text{ if } p^+ = q, \\
0  & \text{ otherwise,}
\end{cases}
\qquad
m_{p,q} = 
\begin{cases}
 \langle h_{i_{ p}}, s_{i_{ p+1}} s_{i_{p+2}} \cdots s_{i_{q}} \La_{i_q}   \rangle  & \text{ if } p\le q,   \\
0  & \text{ otherwise,}
\end{cases}
\end{align*}
where $k^+$ is defined in \eqref{Eq: k+ k-}.
Then it is shown in \cite{KK19, KKOP24, FuOy20} that the bijections  $\gR_\bfi \circ \pbwP_\bfi^{-1}$ and $ \strS_\bfi \circ (\gL_\bfi)^{-1}$ are described explicitly in terms of the matrices $N_\bfi$ and $M_\bfi$.

\begin{prop} \label{Prop: gM gN}   \  
\bni
\item For any $x \in \LB(w)$, we have 
$$
\gL_\bfi ( x) =  \gL_\bfi (L(x)) \qtq
\gR_\bfi ( x) =  \gR_\bfi (L(x)),
$$
where $L(x) \in \tGs (w)$ is the simple object defined in \eqref{Eq: L(x)}.

\item We define $\gM_\bfi :=  \strS_\bfi \circ (\gL_\bfi)^{-1}$, i.e., 
$$ 
\gM_\bfi: \GL_\bfi(w) \buildrel \sim \over \longrightarrow \tcS_\bfi(w).
$$
For any $v = (v_k)_{k\in [1,m]}\in \GL_\bfi(w)$,  we have $\gM_\bfi(v) = (v_k')_{k\in [1,m]}$, where
$$
v_k' = \sum_{t=1}^m m_{k,t} v_t \qquad \text{ for any $k\in [1,m]$},
$$
i.e., we have 
$
\gM_\bfi(v) =   M_\bfi \cdot v.
$
\item We define $\gN_\bfi :=  \gR_\bfi \circ \pbwP_\bfi^{-1}$, i.e., 
$$ 
\gN_\bfi: \tcP_\bfi(w) \buildrel \sim \over \longrightarrow \GR_\bfi(w).
$$
For any $v = (v_k)_{k\in [1,m]}\in \tcP_\bfi(w)$,  we have $\gN_\bfi(\bfa) = (v_k')_{k\in [1,m]}$, where
$$
v_k' = \sum_{t=1}^m n_{k,t} v_t \qquad \text{ for any $k\in [1,m]$},
$$
i.e., we have 
$
\gN_\bfi(v) =   N_\bfi \cdot v .
$

\ee	
\end{prop}	
\begin{proof}
By \cite[Theorem 6.5]{FuOy20}, we have (ii). 

We shall now prove (i). Let $x \in \LB(w)$. Since the seed $\seed_\bfi$ is contained in both of $\Gup(w)$ and $\tGs(w)$, there is no harm to assume that $x\in B(w)$. We set $X := L(x) \in \Gs(w)$ and define 
$\STR_\bfi (X) = (\xi_1, \xi_2, \ldots, \xi_m) $, where $\xi_{k} = \ep_{i_k}( E_{i_{k-1}}^{(\xi_{k-1})} \cdots E_{i_2}^{(\xi_2)}E_{i_1}^{(\xi_1)}(X))$ for $k\in [1,m]$.
Lemma \ref{Lem: ei' max} tells us that 
\begin{align} \label{Eq: same str}
\STR_\bfi (x) = \STR_\bfi (X).
\end{align}
By Proposition \ref{Prop: GrGl for Rgmod}, there exist $\bfa=(a_k)_{k\in [1,m]}, \bfb=(b_k)_{k\in [1,m]} \in \Z_{\ge0}^{[1,m]}$ such that 
$$
\dM^{\bfi}(\bfa) \hconv  X  \simeq \dM^{\bfi}(\bfb).
$$
It follows from Lemma \ref{Lem: ei and seed} (i) and \cite[Lemma 1.7]{KKOP18} that 
\bna
\item $ \ep_{i_1}(X) + \sum_{k\in [1,m]} a_k \ep_{i_1} (\dM^{\bfi}_k) = \ep_{i_1} (\dM^{\bfi}(\bfa) \hconv  X) =  \ep_{i_1} (\dM^{\bfi}(\bfb) ) = \sum_{k\in [1,m]} b_k \ep_{i_1} (\dM^{\bfi}_k)$, 
\item $\ep_{i_1}(\dM^{\bfi}_k) = \langle h_{i_1}, s_{i_{ 2}} s_{i_{3}} \cdots s_{i_{k}} \La_{i_k} \rangle = m_{1,k}$.
\ee
This implies that 
$$
\ep_{i_1} (X) = \sum_{k\in [1,m]} m_{1,k} (b_k - a_k) = \sum_{k\in [1,m]} m_{1,k} g_k,
$$
where $ \gL_\bfi (X) = (g_k)_{k\in [1,m]}$ by Proposition \ref{Prop: GrGl for Rgmod}. Thus
Lemma \ref{Lem: ei and seed} (ii) says that $  \STR_\bfi (X)=  M_\bfi \cdot \gL_\bfi (X) $ by the standard induction argument. By (ii) and \eqref{Eq: same str}, we have  
$$
\gL_\bfi (x) = \gL_\bfi (X).
$$
Since $ \gR_\bfi(b) + \gL_\bfi(\cdual_w(b)) =0$ by \cite[Lemma 3.13]{KK19} and \cite[Corollary 4.6]{KKOP24}, we also obtain
$$\gR_\bfi (x) = \gR_\bfi (X),
$$ which completes the proof for (i).

By (i), we have (iii) by  \cite[Proposition 3.14]{KK19} and \cite[Proposition 4.5]{KKOP24}.
\end{proof}

\begin{remark} \
\bni
\item 
By construction, the matrices $N_\bfi$ and $M_\bfi$ are upper-triangular matrices whose diagonal entries are all $1$.
Proposition \ref{Prop: gM gN} says that $\gN_\bfi$ and $\gM_\bfi$ are $\Z$-linear.

\item Although two bases $\Gup(w)$ and $\Gs(w)$ are different in general, their $g$-vectors coincide by Proposition \ref{Prop: gM gN} (i). Thus when dealing with $g$-vectors in this paper, we will sometimes not distinguish between two bases $\Gup(w)$ and $\Gs(w)$.
\item By Proposition \ref{Prop: gM gN} and \eqref{Eq: GrGlZm}, we have 
$$
\tcS_\bfi(w) = \Z^{[1,m]} = \tcP_\bfi(w).
$$
\ee
\end{remark}

Recall the crystal isomorphism $\psi_\bfi$ given in \eqref{Eq: PBW STR},  the twist automorphism $\eta_w$ on the localized quantum coordinate ring $\tqA_q(\n(w))$ given in Section \ref{Sec: qcr} and the permutation $\cdual_w$ on the localized crystal $\LB(w)$ given in \eqref{Eq: Dw and etaw}. 
Note that we have the following bijections:
\begin{equation} \label{Eq: isom}
\begin{aligned} 
\xymatrix{
&&& \LB(w) \ar[ld]_{\pbwP_\bfi} \ar[rd]^{\strS_\bfi}
\\
\GR_\bfi(w) && \ar[ll]_ {\gN_\bfi} \tcP_\bfi(w) \ar[rr]^ {\psi_\bfi} && \tcS_\bfi(w) && \ar[ll]_{\gM_\bfi} \GL_\bfi(w).
}
\end{aligned}
\end{equation}

\begin{theorem} \label{Thm: main}
Let $w\in \weyl$, $\bfi \in \rR(w)$ and $b \in \LB(w)$.
Under the identification 
$$
\GL_\bfi(w) = \Z^{[1,m]} = \GR_\bfi(w),
$$ 
we have the following: for any $b \in \LB(w)$,
\bni 
\item $ \gL_\bfi( b ) =  \gM_\bfi^{-1} \circ \psi_\bfi \circ \gN_\bfi^{-1} (\gR_\bfi( b )) $,
\item $\gL_\bfi( \cdual_w( b) ) = - \gN_\bfi \circ \psi_\bfi^{-1} \circ \gM_\bfi (\gL_\bfi( b))$,
\item $\gR_\bfi( \cdual_w( b) ) = \gN_\bfi \circ \psi_\bfi^{-1} \circ \gM_\bfi ( -  \gR_\bfi( b))$,
\item $ \PBW_{\bfi}( \cdual_w(b) )  =  \psi_\bfi^{-1} \circ \gM_\bfi \circ \gN_\bfi (  -   \PBW_{\bfi}(b )) $,
\item $ \STR_{\bfi}( \cdual_w(b) )  = -   \gM_\bfi \circ \gN_\bfi \circ \psi_\bfi^{-1} (\STR_{\bfi}(b )) $.
\ee
\end{theorem}	
\begin{proof}
Since $\cdual_w$ is the crystal-theoretic counter part of the right dual functor $\dual_w$ in the localized category $\tcC_w$ (see \eqref{Eq: Dw and etaw}), by \cite[Lemma 3.13]{KK19} and \cite[Corollary 4.6]{KKOP24}, we have 
\begin{align} \label{Eq: gR gL}
 \gR_\bfi(b) + \gL_\bfi(\cdual_w(b)) =0. 
\end{align}
 
Then (i) follows from \eqref{Eq: isom}, and the others follow from (i) and \eqref{Eq: gR gL}.	
\end{proof}

\begin{remark} \
\bni
\item In \cite[Theorem 5.22]{KiOy21}, the quantum twist automorphism is
described in an abstract manner from the crystal-theoretic viewpoint. 
It is not clear how \cite[Theorem 5.22]{KiOy21} and the description in Theorem \ref{Thm: main} are connected. 
\item
Since the map $\psi_\bfi$ is not $\Z$-linear, Theorem~\ref{Thm: main}~(iii) and (iv) 
\emph{do not} imply
\[
\gR_\bfi( \cdual_w( b) ) = -\gN_\bfi \circ \psi_\bfi^{-1} \circ \gM_\bfi ( \gR_\bfi( b))
\quad \text{or} \quad
\PBW_{\bfi}(\cdual_w(b)) = - \psi_\bfi^{-1} \circ \gM_\bfi \circ \gN_\bfi (\PBW_{\bfi}(b)).
\]
%hold in general.
%
For instance, let $\cmA$ be of type $A_3$, $w = s_2 s_1 s_3 s_2 \in \weyl$ and $\bfi = (2,1,3,2) \in \rR(w)$. 
If we take $b = \tf_2 \one \in \LB(\infty)$, then  $\PBW_{\bfi}(b) = (1,0,0,0)$ and 
\begin{align*}
\psi_\bfi^{-1} \circ \gM_\bfi \circ \gN_\bfi (-\PBW_{\bfi}(b)) &= (0,-1,-1,1), \\ 
- \psi_\bfi^{-1} \circ \gM_\bfi \circ \gN_\bfi (\PBW_{\bfi}(b)) &= (-1,0,0,0).
\end{align*}
\ee
\end{remark}

\begin{example} 
Let $\cmA$ is of type $A_2$, and let $\bfi = (1,2,1) \in R(w_\circ)$.
For any $b\in \LB(\infty)$, we write 
\begin{align*}
\PBW_\bfi( b ) = (a_1, a_2, a_3), \quad 
\STR_\bfi( b ) = (t_1,t_2,t_3), \quad
\gR_\bfi( b ) = (r_1, r_2, r_3), \quad
\gL_\bfi( b ) = (l_1,l_2,l_3).
\end{align*}
\bni
\item
We simply write $ \qmD_{ k} = \qmD^\bfi_{k}$ and $\dM_{ k} = \dM^\bfi_{k}$ for $k\in [1,3]$.
Then we have
$$
\qmD_1 = [\dM_1] = \Gup(\tf_1 \one), \qquad \qmD_2 = [\dM_2] = \Gup(\tf_1\tf_2 \one), \qquad \qmD_3 = [\dM_3] = \Gup(\tf_2\tf_1 \one).
$$
 Note that $\Gup(\fz_1) = \qmD_3 $ and $\Gup(\fz_2)  = \qmD_2 $ are frozen.
We have
 \begin{align*}
 P_1 &= \PBW_\bfi(\fz_1) = (1,0,1), \quad  P_2 = \PBW_\bfi(\fz_2) = (0,1,0), \\ 
 S_1 &= \STR_\bfi(\fz_1) = (0,1,1), \quad \  S_2 = \STR_\bfi(\fz_2) = (1,1,0),
 \end{align*}
and
 $$
 \tcP_\bfi(\infty) =  \cP_\bfi(\infty) + \Z P_1 + \Z P_2, \qquad 
 \tcS_\bfi(\infty) =  \cS_\bfi(\infty) + \Z S_1 + \Z S_2. 
 $$
 Note that $\cP_\bfi(\infty) = \Z_{\ge0}^{[1,3]}$, $   \cS_\bfi(\infty) =  \{ (x,y,z) \mid  x \ge0, y \ge z\ge 0 \}$ and $\tcP_\bfi(\infty) = \tcS_\bfi(\infty) = \Z^3$.

Recall the map $\psi_\bfi : \tcP_\bfi(\infty) \buildrel \sim \over \longrightarrow \tcS_\bfi(\infty)$ defined in \eqref{Eq: PBW STR}. The map $\psi_\bfi$ can be computed explicitly as follows: 
$$
\psi_\bfi ( a_1, a_2, a_3) = (a_1 + a_2 - \min \{ a_1, a_3\}, a_2+a_3, \min\{a_1,a_3\}).
$$
Note that $\psi_\bfi$ is not $\Z$-linear but 
$$ 
\psi_\bfi ( x + aP_1 + bP_2 ) = \psi_\bfi ( x ) + aS_1 + bS_2 
$$ for any $x\in \tcP_\bfi(\infty)$ and $a,b \in \Z$. 	
\item It follows from the definition \eqref{Eq: N M} that
$$
N_\bfi = \begin{pmatrix} 1&0&-1\\ 0&1&0 \\ 0&0&1 \end{pmatrix} \qtq 
M_\bfi = \begin{pmatrix} 1&1&0\\ 0&1&1 \\ 0&0&1 \end{pmatrix},
$$ 
which yields that 
$$
\gN_\bfi(a_1, a_2, a_3) = (a_1-a_3,a_2,a_3) \qtq 
\gM_\bfi(l_1, l_2, l_3) = (l_1+l_2,l_2+l_3,l_3).
$$

On the other hand, by a direct computation, we have
$$
N_\bfi^{-1} = \begin{pmatrix} 1&0&1\\ 0&1&0 \\ 0&0&1 \end{pmatrix} \qtq
M_\bfi^{-1} = \begin{pmatrix} 1&-1&1\\ 0&1&-1 \\ 0&0&1 \end{pmatrix}. 
$$
Let $b\in \tB(\infty)$ and write $\cdual = \cdual_{w_\circ}$. If we write $\gR_\bfi(b) = (r_1,r_2,r_3)$, then we have 
$$
\text{$\gL_\bfi(b) =  (r_1, r_2-\min \{ r_1,0\}, r_3+\min \{ r_1,0 \} ).$}
$$
\ee
\end{example}

\vskip 2em

\section{Minuscule representations} \label{Sec: MR}

In this section, when $\g$ is of finite classical type and $b$ is an element of the \emph{$*$-twisted minuscule crystal}, we shall describe $\cdual(b)$ explicitly in terms of combinatorics of Young diagrams. Throughout this section, we assume that $\g$ is of finite classical type, i.e., of type $A_n$, $B_n$, $C_n$ or $D_n$. We shall also use the standard convention for Dynkin diagrams (see \cite[Chapter 4]{Kac90} for example).

Let $B(\La)$ be the crystal of the irreducible module $V(\La)$ with highest weight $\La \in \wlP^+$. We denote by $b_\La$ the highest weight element of $B(\La)$ and consider the embedding
$$
\iota: B(\La) \hookrightarrow B(\infty)
$$ 
defined by $\iota(b_\La)= \one $ and $\iota(\te_i(b)) = \te_i(\iota(b))$ for any $b\in B(\La)$ and $i\in I$. 
Note that the embedding $\iota$ is the crystal counterpart of the dual map of the canonical projection $U_q^-(\g) \twoheadrightarrow V(\Lambda)$ sending $x$ to $x \cdot v_\La$, where $v_\La$ is the highest weight vector of $V(\La)$.
It is known that the image $\iota(B(\La))$ is equal to $\{ b \in B(\infty) \mid \eps_i(b) \le \La(h_i) \text{ for all } i\in I   \}$. We identify $B(\La)$ with the image $\iota(B(\La))$ if no confusion arises. Similarly, we define 
\begin{align} \label{Eq: *-twisted MC}
B(\Lambda)^* := \{ b\in B(\infty) \mid \ep_i(b) \le \La(h_i) \text{ for all } i\in I  \},
\end{align}
which is called the $*$-twisted crystal. It is easy to see that the involution $*$ takes $B(\La)$ to $B(\La)^*$ as subsets of $B(\infty)$.

We define 
\begin{align*}
\IM = 
\begin{cases}
I & \text{ if $\weyl$ is of type $A_n$}, \\
\{ n \} & \text{ if $\weyl$ is of type $B_n$}, \\
\{ 1 \} & \text{ if $\weyl$ is of type $C_n$}, \\
\{1, n-1, n \} & \text{ if $\weyl$ is of type $D_n$}, 
\end{cases}
\end{align*}
and an index $i \in \IM $ is called a \emph{minuscule} index.

For a minuscule index $t \in I_{\mathrm{M}}$, we define $\weyl_t$ to be the subgroup of $\weyl$ generated by the generators $s_i$ for $i\in I \setminus \{t \}$.
It is known that the weight map
\begin{align} \label{Eq: bij bw B and WLa}
 B(\La_t)^* \buildrel \wt \over \longrightarrow \weyl \cdot \La_t
\end{align}
is bijective (see \cite[Section 1]{P20} for example).
Let $x_t$ be the longest element of the set of minimal-length coset representatives for $\weyl_t \backslash \weyl$, and we simply write $ \cdual $ for $\cdual_{x_t}$ if no confusion arises.

\subsection{Type $A_n$}
In this case, we have $I = \IM$. We take and fix $t \in I$. 

For a \emph{partition} $\la = (\la_1 \ge \la_2 \ge \cdots \ge \la_l >0)$ with $N = \sum_{k=1}^l \la_k$, we set $|\la|:=N$. We identify partitions with \emph{Young diagrams} using English convention. 
We write a box $\bx = (i,j) \in \la$ if there is a box in the $i$th row and the $j$th column of $\la$.
A \emph{standard tableau} $T$ of shape $\la$ is a bijection from the set of boxes in $\la$ to the set $[1,N]$ such that the entries in rows and columns increase from left to right and top to bottom respectively.
Define $\ST(\la)$ to be the set of all standard tableaux of shape $\la$.
We denote by $T^\la$ the \emph{initial tableau}, which is the distinguished standard tableau where we fill the boxes with $[1,N]$ along successive rows from top to bottom.

For a box $\bx = (i,j) \in \la$,
we define 
$$
\res_t(\bx) := t-i+j,
$$ which is called the \emph{residue} of $\bx$. 
We simply write $\res(\bx) $ for $\res_t(\bx)$ if no confusion arises.
The weight of $\la$ is defined by 
\begin{align} \label{Eq: wt for la}
\wt(\la) := \sum_{\bx\in \la} \al_{\res(\bx)}.
\end{align}
For a standard tableau $T$ of shape $\la$, we define 
\begin{align} \label{Eq: res of T}
\res(T) = ( \res(T^{-1}(1)), \res(T^{-1}(2)), \ldots, \res(T^{-1}(N))).
\end{align}

We define 
\begin{align} \label{Eq: la circ and bfi An}
\la^\circ_t := ( \underbrace{n-t+1, n-t+1, \ldots, n-t+1}_{t \text{ times}}), 
\qquad  \bfi := \res(T^{\la_t^\circ}),
\end{align}
and set $N := |\la_t^\circ|= t(n-t+1)$.
Note that $\bfi$ is a reduced expression of the longest element $x_t$ of the set of minimal-length coset representatives for $\weyl_t \backslash \weyl$.

\begin{example} \label{Ex: An 1}
Let $n=5$ and $t=3$.  
Then we have $\lambda^\circ_3 = (3,3,3)$ and $N=9$. 
The residues for $\lambda^\circ_3$ are given as
$$
\la_3^\circ = 
\begin{tabular}{ |c|c|c| } 
 \hline
 3 & 4 & 5 \\ 
 \cline{1-3}
 2 & 3 & 4 \\
  \cline{1-3}
 1 & 2 & 3 \\ 
  \cline{1-3}
\end{tabular}\ ,
$$
where the number in a box $\bx \in \la_3^\circ$ is the residue $\res(\bx)$.
The residues tell us that $\wt(\la_3^\circ) = \al_1 + 2\al_2 + 3\al_3 + 2\al_4 + \al_5$.
The initial tableau $T^{\la_3^\circ}$ is give as
$$
T^{\la_3^\circ} = 
\begin{tabular}{ |c|c|c| } 
 \hline
 1 & 2 & 3 \\ 
 \cline{1-3}
 4 & 5 & 6 \\
  \cline{1-3}
 7 & 8 & 9 \\ 
  \cline{1-3}
\end{tabular},
$$ 
which gives  $\bfi = \res(T^{\la_3^\circ})= ( 3,4,5, 2,3,4, 1,2,3)$.
\end{example}

The following lemma can be obtained from the bijection between $B(\La_t)$ and its weight set 
%(see \cite[Section 3.5]{KR08} and \cite[Section 2.2]{P20} for example).
(see \cite{KR08} and \cite{P20} for example). 

\begin{lemma} \label{Lem: YD and WLat}
Let $\YD_t = \{ \la \vdash N \mid \la \subset \la^\circ_t \}$.
Then the correspondence $ \la \mapsto \La_t - \wt(\la) $ gives a bijection from $\YD_t$ to the set 
$\weyl \cdot \La_t$.
\end{lemma}

Thanks to Lemma \ref{Lem: YD and WLat}, we denote by $\la_{ \beta}$ the Young diagram in $\YD_t$ corresponding to the weight $\La_t-\beta \in \weyl \cdot \La_t$. 
If $b\in B(\La_t)^*$ is the corresponding element of $\la \in \YD_t$ (see \eqref{Eq: bij bw B and WLa}), we set $\la_b := \la_{\La_t - \wt(b)}$ and $b_\la := b$.

We briefly recall the homogeneous representations of the quiver Hecke algebra associated with minuscule representations (see \cite{KR08} and \cite{P20}), which will take a key role in proving Proposition \ref{Prop: homogenous YD A}. 
Let $R$ be the quiver Hecke algebra over a field $\bR$ associated with the Cartan matrix of type $A_n$ and polynomials $ \qQ_{i,j}$ satisfying certain conditions \cite[(2.1)]{P20}.
For $\la \subset \la_t^\circ$, we set $\Specht^\la := \bigoplus_{T \in \ST(\la)} \bR T $ and define an $R(\wt(\la))$-module structure on $\Specht^\la$ by 
$$
e(\nu) T = \delta_{\nu, \res(T)}T, \qquad x_i e(\nu) T = 0, \qquad 
\tau_j e(\nu) T = 
\begin{cases}
s_j T  & \text{ if $\nu = \res(T)$ and } s_j T \in \ST(\la),  \\
0 & \text{ otherwise,}
\end{cases}
$$
for $T \in \ST(\la)$ and admissible $i,j, \nu$. Here $s_j T$ is the tableau obtained from $T$ by swapping $j$ and $j+1$. Then $\Specht^\la$ becomes an irreducible $R$-module, which is called the \emph{Specht module}. 
By construction, we have $\ep_i(\Specht^\la) = \delta_{i,t}$ and  
\begin{align} \label{Eq: qch of S}
\qch(\Specht^\la) = \sum_{T \in \ST(\la)} \res(T).
\end{align}
Thus one can easily show $L(b_\la) \simeq \Specht^\la $.

We now take $\la \subset \la_t^\circ$. Given $ 1 \le k  \le N$, 
we take $\bx = (r,c) \in \la_t^\circ $ such that $ T^{\la_t^\circ}(\bx)=k $.
 We then define $a_k := 1$ if $(r,c) \in \la$ and $a_k := 0$ otherwise, and set 
\begin{align} \label{Eq: def si}
\bfs_\bfi (\la) := (a_1, a_2, \ldots, a_{N}) \in \{0,1 \}^N.
\end{align}
On the other hand, the Young diagram $\la$ can be decomposed into \emph{diagonal hooks}, i.e., 
\begin{align}
H(\la) = (H_1, H_2, \ldots, H_r),
\end{align}
where $H_1$ is the largest hook in $\la$ consisting of the first row and the first column; $H_2$ is the largest hook in the Young diagram obtained from $\la$ by removing $H_1$, and so on.   
Note that $\wt(H_k) \in \prD$ for all $k$. Let $ \HW_t := \{ \wt(H_k) \mid k\in [1,r] \} \subset \prD$. 
For a given $ 1 \le k  \le N$, we define $p_k := 1$ if $ \beta_k \in \HW_t $ and $p_k := 0$ otherwise, and set 
\begin{align}
\bfp_\bfi (\la) := (p_t)_{t\in [1,N]} \in \{0,1 \}^N,
\end{align}
where $\beta_k$ is the positive root given in \eqref{Eq: beta_k}.
We then have the following proposition.

\begin{prop} \label{Prop: homogenous YD A}
For any $b \in B(\La_t)^*$,  we have
\bni
\item $\PBW_\bfi(b) = \bfp_\bfi (\la_b)$,
\item $\STR_\bfi(b) = \bfs_\bfi (\la_b)$,
\item $\PBW_\bfi( \cdual^{-1}(b) ) = -\gN_\bfi^{-1} \circ \gM_\bfi^{-1}(\bfs_\bfi (\la_b)) $,
\item $\STR_\bfi( \cdual(b) ) = -\gM_\bfi \circ \gN_\bfi(\bfp_\bfi (\la_b)) $,
\ee
where $\gM_\bfi$ and $\gN_\bfi$ are the linear maps defined in Proposition \ref{Prop: gM gN}.
\end{prop}
\begin{proof}
We consider the Specht module $\Specht^{\la_b}$. Since $L(b) \simeq \Specht^{\la_b}$, we have the assertion (ii) by \eqref{Eq: ep in Rgmod} and \eqref{Eq: qch of S}.

On the other hand, we consider the diagonal hook decomposition $H(\la) = (H_1, H_2, \ldots, H_r)$.
By the same argument as in the proof of \cite[Theorem 3.4 (2)]{P20}, we have an isomorphism 
$$
\hd (\Specht^{H_1} \circ \Specht^{H_2} \circ \cdots \circ \Specht^{H_r} ) \simeq \Specht^{\la_b}.
$$
Since $\Specht^{H_k}$ ($k\in [1,r]$) are \emph{cuspidal modules} associated with $\bfi$, i.e., $\Specht^{H_k}$ belongs to the set of dual PBW vectors associated with $\bfi$ under the categorification, we have the assertion (i).
The other assertions (iii) and (iv) follows from Theorem \ref{Thm: main}.
\end{proof}

\begin{example}\label{ex: Type A (3,2,2) diagonal hook decomposition}
We keep all notations appearing in Example \ref{Ex: An 1}.
Let $\la = (3,2,2) \subset \la_3^\circ$. 
The residues for $\la$ are given as
$$
\begin{tabular}{ |c|c|c| } 
 \hline
 3 & 4 & 5 \\ 
 \cline{1-3}
 2 & 3   \\
  \cline{1-2}
 1 & 2  \\ 
  \cline{1-2}
\end{tabular},
$$
where the number in the box $\bx \in \la$ is the residue $\res(\bx)$. The residues say that 
$$
\wt(\la) = \al_1+2\al_2 + 2\al_3+\al_4+\al_5 = \La_3 - s_2 s_1s_3s_2 s_5s_4 s_3 (\La_3),
$$
which means that $\La_3 - \wt(\la) \in \weyl \cdot \La_3$, and 
$
\bfs_\bfi(\la) = (1,1,1,1,1,0,1,1,0).
$

The diagonal hook decomposition of $\la$ is $H(\la) = (H_1, H_2)$, where
$$
H_1 =\begin{tabular}{ |c|c|c| } 
 \hline
 3 & 4 & 5 \\ 
 \cline{1-3}
 2    \\
\cline{1-1}
 1   \\ 
  \cline{1-1}
\end{tabular}
 \qtq  
H_2 =\begin{tabular}{ |c| } 
 \hline
 3    \\
\hline
 2   \\ 
 \hline
\end{tabular}.
$$
We write 
\begin{align} \label{Eq: al_ab}
\al_{a,b} := \sum_{k=a}^b \al_k \qquad \text{ for } 1 \le a\le b\le n.
\end{align}
Since $\wt(H_1) = \al_{1,5}$, $\wt(H_2) = \al_{2,3}$ and 
\begin{align*}
&\beta_1= \al_3, \quad \beta_2= \al_{3,4}, \quad \beta_3= \al_{3,5}, \quad \beta_4= \al_{2,3}, \quad \beta_5= \al_{2,4}, \\
&\beta_6= \al_{2,5}, \quad \beta_7= \al_{1,3}, \quad \beta_8= \al_{1,4}, \quad \beta_9= \al_{1,5},
\end{align*}
we have $\bfp_\bfi(\la) = (0,0,0,1,0,0,0,0,1).$ 
By Proposition \ref{Prop: gM gN} (ii), we have
\begin{align*}
\STR_\bfi(\fz_1) &= (1,0,0,1,0,0,1,0,0), \\ 
\STR_\bfi(\fz_2) &= (1,1,0,1,1,0,1,1,0),  \\
\STR_\bfi(\fz_3) &= (1,1,1,1,1,1,1,1,1), \\
\STR_\bfi(\fz_4) &= (1,1,1,1,1,1,0,0,0), \\
\STR_\bfi(\fz_5) &= (1,1,1,0,0,0,0,0,0), 
\end{align*}
where $\fz_k$ are elements in $B(\infty)$ corresponding to frozen variables (see \eqref{Eq: fzwi}).
Hence we have 
\begin{align*} 
\STR_\bfi( \cdual(b_\la) ) & = - \gM_\bfi \circ \gN_\bfi(\bfp_\bfi (\la))
= -(1,0,1,1,0,1,1,1,1) \\ 
&= (1,1,0,1,1,0,1,0,0) - \STR_\bfi(\fz_3) - \STR_\bfi(\fz_1) \eqfr (1,1,0,1,1,0,1,0,0),
%& \eqfr (0,1,0,0,1,0,0,0,0) ,
\end{align*}
where $\eqfr$ is given in \eqref{Eq: equiv for fr}.
Letting $\mu = (2,1,1) $, we have $b_\mu \in B(\La_t)^*$ and 
$$\STR_\bfi( b_\mu ) = (1,1,0,1,1,0,1,0,0),$$ which says that 
$
\cdual(b_\la) \eqfr  b_\mu.
$
\end{example}

\subsection{Type $B_n$}
In this case, we have $\IM= \{ n \}$.

A partition $\la = (\la_1, \la_2, \cdots, \la_l)$ is \emph{strict} if  $\la_1 > \la_2 > \cdots > \la_l$.
We identify strict partitions with \emph{shifted Young diagrams} using English convention.
The standard tableaux and the initial tableau $T^\la$ of shape $\la$ are defined in the same manner as in the Young diagram case. 
For a box $\bx = (i,j)$ in a shifted Young diagram $ \la$, the residue of $\bx$ is defined as 
$$
\res(\bx) := n+i-j,
$$  
and the weight $\wt(\la)$ and the residue sequence $\res(T)$ are defined by the same manner as before (see \eqref{Eq: wt for la} and \eqref{Eq: res of T}).
We define 
$$
\la^\circ := ( n,n-1, \ldots, 2,1), 
\qquad  \bfi := \res(T^{\la^\circ}),
$$
and set $N := |\la^\circ|= n(n+1)/2$.
Note that $\bfi$ is a reduced expression of the longest element $x_n$ of the set of minimal-length coset representatives for $\weyl_n \backslash \weyl$.

\begin{example} \label{Ex: Bn 1}
Let $n=4$. We then have $\lambda^\circ = (4,3,2,1)$, $N=10$ and 
the residues for $\lambda^\circ$ are given as
$$
\la^\circ = 
\begin{tabular}{ cccc| } 
 \hline
 \multicolumn{1}{|c}{4} & \multicolumn{1}{|c}{3} & \multicolumn{1}{|c}{2} & \multicolumn{1}{|c|}{1} \\ 
 \cline{1-4}
 & \multicolumn{1}{|c}{4} & \multicolumn{1}{|c}{3} & \multicolumn{1}{|c|}{2} \\
  \cline{2-4}
 &   & \multicolumn{1}{|c}{4}  & \multicolumn{1}{|c|}{3}  \\ 
  \cline{3-4}
  &&& \multicolumn{1}{|c|}{4}  \\ 
  \cline{4-4}
\end{tabular},
$$
where the number in the box $\bx \in \la^\circ$ is the residue $\res(\bx)$.
This implies that $\wt(\la^\circ) = \al_1 + 2\al_2 + 3\al_3 + 4\al_4$ and
the initial tableau $T^{\la^\circ}$ is give as
$$
T^{\la^\circ} = 
\begin{tabular}{ cccc| } 
 \hline
 \multicolumn{1}{|c}{1} & \multicolumn{1}{|c}{2} & \multicolumn{1}{|c}{3} & \multicolumn{1}{|c|}{4} \\ 
 \cline{1-4}
 & \multicolumn{1}{|c}{5} & \multicolumn{1}{|c}{6} & \multicolumn{1}{|c|}{7} \\
  \cline{2-4}
 &   & \multicolumn{1}{|c}{8}  & \multicolumn{1}{|c|}{9}  \\ 
  \cline{3-4}
  &&& \multicolumn{1}{|c|}{10}  \\ 
  \cline{4-4}
\end{tabular},
$$
which says   $\bfi = \res(T^{\la^\circ})= ( 4,3,2,1, 4,3,2, 4,3, 4)$.
\end{example}

\begin{lemma} [{\cite[Lemma 3.2]{P20}}] \label{Lem: B SYD and WLat}
Let $\SYD_n = \{ \la \vdash N \mid \la \subset \la^\circ \}$.
Then the correspondence $ \la \mapsto \La_n - \wt(\la) $ gives a bijection from $\SYD_n$ to the set 
$\weyl \cdot \La_n$.
\end{lemma}

We denote by $\la_{ \beta}$ the shifted Young diagram in $\SYD_n$ corresponding to the weight $\La_n-\beta \in \weyl \cdot \La_n$.
If $b\in B(\La_n)^*$ is the corresponding element of $\la \in \SYD_n$ (see \eqref{Eq: bij bw B and WLa}), define $\la_b := \la_{\La_t - \wt(b)}$ and $b_\la := b$.

We briefly recall the homogeneous representations.
Let $R$ be the quiver Hecke algebra over a field $\bR$ associated with the Cartan matrix of type $B_n$ and polynomials $ \qQ_{i,j}$ satisfying certain conditions \cite[(2.1)]{P20}.
For $\la \subset \la_t^\circ$, we set $\Specht^\la := \bigoplus_{T \in \ST(\la)} \bR T $ and define an $R(\wt(\la))$-module structure on $\Specht^\la$ by 
$$
e(\nu) T = \delta_{\nu, \res(T)}T, \qquad x_i e(\nu) T = 0, \qquad 
\tau_j e(\nu) T = 
\begin{cases}
s_j T  & \text{ if $\nu = \res(T)$ and } s_j T \in \ST(\la),  \\
0 & \text{ otherwise,}
\end{cases}
$$
for $T \in \ST(\la)$ and admissible $i,j, \nu$. Then $\Specht^\la$ is irreducible (\cite[Theorem 3.4]{P20}). 
By construction, we have $\ep_i(\Specht^\la) = \delta_{i,n}$ and  $\qch(\Specht^\la) = \sum_{T \in \ST(\la)} \res(T).$
Thus we have $L(b_\la) \simeq \Specht^\la $.

 For a given $ 1 \le k  \le N$,
we take $\bx = (r,c) \in \la^\circ $ such that $ T^{\la^\circ}(\bx)=k $,
and define $a_k := 1$ if $(r,c) \in \la$ and $a_k := 0$ otherwise. We set 
\begin{align*}
\bfs_\bfi (\la) := (a_1, a_2, \ldots, a_{N}) \in \{0,1 \}^N.
\end{align*}
The shifted Young diagram $\la$ can be decomposed into \emph{rows}, i.e., 
\begin{align*}
R(\la) = (R_1, R_2, \ldots, R_r),
\end{align*}
where $R_1$ is the first row of $\la$, $R_2$ is the second row, and so on. Note that $\wt(R_k) \in \prD$ for all $k$. Let $ \HW := \{ \wt(R_k) \mid k\in [1,r] \} \subset \prD$. 
For any $ 1 \le k  \le N$, we define $p_k := 1$ if $ \beta_k \in \HW $ and $p_k := 0$ otherwise, and set 
\begin{align*}
\bfp_\bfi (\la) := (p_t)_{t\in [1,N]} \in \{0,1 \}^N,
\end{align*}
where $\beta_k$ is the positive root given in \eqref{Eq: beta_k}.
By \cite[Theorem 3.4]{P20} and the same argument of the proof of Proposition \ref{Prop: homogenous YD A}, we have the following. 

\begin{prop} \label{Prop: homogenous SYD B}
For any $b \in B(\La_n)^*$, we have
\bni
\item $\PBW_\bfi(b) = \bfp_\bfi (\la_b)$,
\item $\STR_\bfi(b) = \bfs_\bfi (\la_b)$,
\item $\PBW_\bfi( \cdual^{-1}(b) ) = -\gN_\bfi^{-1} \circ \gM_\bfi^{-1}(\bfs_\bfi (\la_b)) $,
\item $\STR_\bfi( \cdual(b) ) = -\gM_\bfi \circ \gN_\bfi(\bfp_\bfi (\la_b)) $,
\ee
where $\gM_\bfi$ and $\gN_\bfi$ are the linear maps defined in Proposition \ref{Prop: gM gN}.
\end{prop}

\begin{example}
We keep all notations appearing in Example \ref{Ex: Bn 1}.
Let $\la = (4,2,1) \subset \la^\circ$, i.e.,
$$
\la  = 
\begin{tabular}{ cccc| } 
 \hline
 \multicolumn{1}{|c}{4} & \multicolumn{1}{|c}{3} & \multicolumn{1}{|c}{2} & \multicolumn{1}{|c|}{1} \\ 
 \cline{1-4}
 & \multicolumn{1}{|c}{4} & \multicolumn{1}{|c|}{3}  \\
  \cline{2-3}
 &   & \multicolumn{1}{|c|}{4}    \\ 
  \cline{3-3}
\end{tabular},
$$
where the number in the box $\bx \in \la$ is the residue $\res(\bx)$. 
We then have  
$$
\wt(\la) = \al_1+ \al_2 + 2\al_3+ 3\al_4 = \La_4 - s_4 s_3s_4 s_1s_2s_3 s_4 (\La_4),
$$
which says that $\La_4 - \wt(\la) \in \weyl \cdot \La_4$, and 
$\bfs_\bfi(\la) = (1,1,1,1, 1,1,0, 1,0,0).$

The row decomposition of $\la$ is $R(\la) = (R_1, R_2, R_3)$, where
$$
R_1 =\begin{tabular}{ |c|c|c|c| } 
 \hline
 4 & 3 &2&1 \\ 
\hline
\end{tabular}
, \qquad 
R_2 =\begin{tabular}{ |c|c| } 
 \hline
 4 & 3  \\ 
\hline
\end{tabular}
, \qquad 
R_3 =\begin{tabular}{ |c| } 
 \hline
 4   \\ 
\hline
\end{tabular}.
$$
We then have  $\wt(R_1) = \al_{1,4}$, $\wt(R_2) = \al_{3,4}$, $\wt(R_3) = \al_{4}$ and 
\begin{align*}
&\beta_1= \al_4, \quad \beta_2= \al_{3,4}+\al_4, \quad \beta_3= \al_{2,4}+\al_4, \quad \beta_4= \al_{1,4}+\al_4, \quad \beta_5= \al_{3,4}, \\
&\beta_6= \al_{2,4}+\al_{3,4}, \quad \beta_7= \al_{1,4}+\al_{3,4}, \quad \beta_8= \al_{2,4}, \quad \beta_9= \al_{1,4}+\al_{2,4},\quad \beta_{10}= \al_{1,4},
\end{align*}
where $\al_{a,b}$ is given in \eqref{Eq: al_ab}.
Thus we obtain $\bfp_\bfi(\la) = (1,0,0,0,1,0,0,0,0,1)$, and 
\begin{align*} 
\STR_\bfi( \cdual(b_\la) ) & = - \gM_\bfi \circ \gN_\bfi(\bfp_\bfi (\la))
= -(1,1,0,1,1,0,1,0,1,1) \\ 
& = (0,0,1,0,0,1,0,1,0,0) - \STR_\bfi(\fz_4)  \eqfr (0,0,1,0,0,1,0,1,0,0),
\end{align*}
where $\eqfr$ is given in \eqref{Eq: equiv for fr} and  
$\STR_\bfi(\fz_4) = (1,1,1,1,1,1,1,1,1,1)$.
\end{example}

\subsection{Type $C_n$}
In this case, we have $\IM= \{ 1 \}$. 

We set $ \la^\circ := ( 2n-1)$, $N := |\la^\circ|= 2n-1$ and, for $\bx = (1,j) \in \la^\circ$, define
$$
\res(\bx) = \res(1,j) := 
\begin{cases}
j & \text{ if } 1\le j \le n,\\
2n-j & \text{ if } n < j \le 2n-1.
\end{cases}
$$
For $\la \subset \la^\circ$, the weight $\wt(\la)$ and the residue sequence $\res(T)$ are defined in \eqref{Eq: wt for la} and \eqref{Eq: res of T}.
We set 
$$
\bfi := \res(T^{\la^\circ}) = (1,2,\ldots n-1,n, n-1, \ldots, 2,1),
$$
where $T^{\la^\circ}$ is the initial tableau.
Note that $\bfi$ is a reduced expression of the longest element $x_1$ of the set of minimal-length coset representatives for $\weyl_1 \backslash \weyl$. One can prove the following lemma by using the crystal $B(\La_1)$ of type $C_n$ (see \cite[Section 8.3]{HK02} and \cite[Section 2.3]{P20}).

\begin{lemma} \label{Lem: C YD and WLat}
Let $\YD_1 = \{ \la \vdash N \mid \la \subset \la^\circ \}$.
Then the correspondence $ \la \mapsto \La_1 - \wt(\la) $ gives a bijection from $\YD_1$ to the set 
$\weyl \cdot \La_1$.
\end{lemma}

We denote by $\la_{ \beta}$ the one row Young diagram corresponding to the weight $\La_1-\beta \in \weyl \cdot \La_1$, and write $\la_b := \la_{\La_1 - \wt(b)}$ for any $b\in B(\La_1)^*$. For $\la \in \YD_1$, we denote by $b_\la$ the corresponding element in $ B(\La_1)^*$.
The homogeneous representations $\Specht^\la$ of the quiver Heck algebra of type $C_n$ is defined in a similar manner (see \cite[Section 4]{BKOP14} for definition and \cite[(4.4))]{BKOP14} for the character $\qch (\Specht^\la)$). 
The module $\Specht^\la$ is a 1-dimensional $R$-module and $L(b_\la) \simeq \Specht^\la $.

Let $\la \subset \la^\circ$, $ k \in [1,  N]$ and $\bx = (r,c) \in \la^\circ $ such that $ T^{\la^\circ}(\bx)=k $.
We define $a_k := 1$ if $(r,c) \in \la$ and $a_k := 0$ otherwise, and set 
\begin{align}
\bfs_\bfi (\la) := (a_1, a_2, \ldots, a_{N}) \in \{0,1 \}^N.
\end{align}
Let $ \HW := \{ \wt(\la) \} $ if $\la \ne \la^\circ$ and $ \HW := \{ \al_1, \al_1+\al_n +2\sum_{k=2}^{n-1} \al_k \} $ if $\la = \la^\circ$. Note that $\HW \subset \prD$.
For any $ 1 \le k  \le N$, we define $p_k := 1$ if $ \beta_k \in \HW $ and $p_k := 0$ otherwise, and set 
\begin{align}
\bfp_\bfi (\la) := (p_t)_{t\in [1,N]} \in \{0,1 \}^N,
\end{align}
where $\beta_k$ is the positive root given in \eqref{Eq: beta_k}.

By the same argument of the proof of Proposition \ref{Prop: homogenous YD A}, we have the following. 

\begin{prop} \label{Prop: homogenous YD C}
For any $b \in B(\La_1)^*$, we have
\bni
\item $\PBW_\bfi(b) = \bfp_\bfi (\la_b)$,
\item $\STR_\bfi(b) = \bfs_\bfi (\la_b)$,
\item $\PBW_\bfi( \cdual^{-1}(b) ) = -\gN_\bfi^{-1} \circ \gM_\bfi^{-1}(\bfs_\bfi (\la_b)) $,
\item $\STR_\bfi( \cdual(b) ) = -\gM_\bfi \circ \gN_\bfi(\bfp_\bfi (\la_b)) $,
\ee
where $\gM_\bfi$ and $\gN_\bfi$ are the linear maps defined in Proposition \ref{Prop: gM gN}.
\end{prop}

\begin{example} \label{Ex: Cn 1}
We take $n=4$.  Then $\lambda^\circ = (7)$, $N=7$ and 
the residues for $\lambda^\circ$ are given as
$$
\la^\circ = 
\begin{tabular}{ |c|c|c|c|c|c|c|c|c| } 
 \hline
1&2&3&4&3&2&1\\
\hline
\end{tabular},
$$
where the residue $\res(\bx)$ of a box $\bx \in \la^\circ$ is the number in the box $\bx$.
This implies that $\wt(\la^\circ) = 2\al_1 + 2\al_2 + 2\al_3 + \al_4$ and
the initial tableau $T^{\la^\circ}$ is given as
$$
T^{\la^\circ} = 
\begin{tabular}{ |c|c|c|c|c|c|c|c|c| } 
 \hline
1&2&3&4&5&6&7\\
\hline
\end{tabular},
$$
which says   $\bfi = \res(T^{\la^\circ})= ( 1,2,3,4,3,2,1)$.

\bni
\item Let $\la = (5) \subset \la^\circ$. We have 
$$
\wt(\la) = \al_1+ \al_2 + 2\al_3+ \al_4 = \La_1 - s_3 s_4s_3s_2 s_1(\La_1),
$$
and $\bfs_\bfi(\la) = (1,1,1,1,1,0,0)$. It follows from
\begin{align*}
&\beta_1= \al_1, \quad \beta_2= \al_{1,2}, \quad \beta_3= \al_{1,3}, \quad \beta_4= \al_{1,4}+\al_{1,3},  \\
& \beta_5= \al_{1,4},  \quad \beta_6= \al_{1,4}+\al_{3}, \quad \beta_7= \al_{1,4}+\al_{2,3}
\end{align*}
that $\bfp_\bfi(\la) = (0,0,0,0,0,1,0)$, where $\al_{a,b}$ is given in \eqref{Eq: al_ab}.
Thus we obtain 
\begin{align*} 
\STR_\bfi( \cdual(b_\la) ) & = - \gM_\bfi \circ \gN_\bfi(\bfp_\bfi (\la))
= -(1,0,1,1,1,1,0) \\ 
& = (0,1,0,0,0,0,1) - \STR_\bfi(\fz_1)  \eqfr (0,1,0,0,0,0,1),
\end{align*}
where $\eqfr$ is given in \eqref{Eq: equiv for fr} and $\STR_\bfi(\fz_1) = (1,1,1,1,1,1,1,1,1)$.

\item Let $\la =  \la^\circ$. We have 
$$
\wt(\la) = 2\al_1+ 2\al_2 + 2\al_3+ \al_4 = \La_1 - s_1s_2s_3 s_4s_3s_2 s_1(\La_1).
$$
Then we have  $\bfs_\bfi(\la) = (1,1,1,1,1,1,1)$ and  $\bfp_\bfi(\la) = (1,0,0,0,0,0,1)$.
Since $b_\la$ coincides with the frozen $\fz_1$, we know that $ \cdual(b_\la) = \fz_1^{-1} \eqfr \one$. 
\ee
\end{example}

\subsection{Type $D_n$}
In this case, we have $\IM = \{ 1,n-1,n \} $. 
We shall consider two cases where $t=1$ and $t\in \{n-1, n \}$.

\subsubsection{Case where $t=1$}
Let $ \la^\circ := ( 2n-2)$, $N := |\la^\circ|=2n-2$ and define
$$
\res(\bx)= \res(1,j) := 
\begin{cases}
j & \text{ if } 1\le j \le n,\\
2n-j-1 & \text{ if } n < j \le 2n-2,
\end{cases}
$$
for $\bx = (1,j) \in \la^\circ$.
For $\la \subset \la^\circ$, the weight $\wt(\la)$ and the residue sequence $\res(T)$ are given in \eqref{Eq: wt for la} and \eqref{Eq: res of T}.
We set 
$$
\bfi := \res(T^{\la^\circ}) = (1,2,\ldots n-1,n, n-2, n-3 \ldots, 2,1),
$$
where $T^{\la^\circ}$ is the initial tableau.
Note that $\bfi$ is a reduced expression of the longest element $x_1$ of the set of minimal-length coset representatives for $\weyl_1 \backslash \weyl$.  
We set $\mu := (1^{n-1})$ and $\wt(\mu) = \al_{1,n-2} + \al_n$.

Like the case of type $C_n$, it is easy to prove the following (see \cite[Section 8.3]{HK02} and \cite[Section 2.2]{P20}).

\begin{lemma} \label{Lem: D YD and WLat}
Let $\YD_1 = \{ \la \vdash N \mid \la \subset \la^\circ \}  \cup \{ \mu \} $.
Then the correspondence $ \la \mapsto \La_1 - \wt(\la) $ gives a bijection from $\YD_1$ to the set 
$\weyl \cdot \La_1$.
\end{lemma}

  Let $\la_{ \beta} \in \YD_1 $ be the   Young diagram corresponding to the weight $\La_1-\beta \in \weyl \cdot \La_1$, and write $\la_b := \la_{\La_1 - \wt(b)}$ for any $b\in B(\La_1)^*$. For $\la \in \YD_1$, we denote by $b_\la$ the corresponding element in $ B(\La_1)^*$.
The homogeneous representations $\Specht^\la$ of the quiver Heck algebra of type $D_n$ is defined in a similar manner (see \cite[Section 4]{BKOP14} for definition and see \cite[(4.4))]{BKOP14} for the character $\qch (\Specht^\la)$). The module $\Specht^\la$ is a 1 or 2-dimensional $R$-module and $L(b_\la) \simeq \Specht^\la $.

Let $\la\in \YD_1 $. If $\la=  \mu$, then we set 
$\bfs_\bfi (\mu) := (\underbrace{1,\ldots, 1}_{n-2 \text{ times} }, 0,1, 0, 0, \ldots, 0  ) \in \{0,1 \}^N$ 
and $\bfp_\bfi (\mu) = e_n$.
Suppose that $\la \subset \la^\circ$.
Let  $ k \in [1,  N]$ and $\bx = (r,c) \in \la^\circ $ such that $ T^{\la^\circ}(\bx)=k $.
We define $a_k := 1$ if $(r,c) \in \la$ and $a_k := 0$ otherwise, and set 
\begin{align} \label{Eq: si for D1}
\bfs_\bfi (\la) := (a_1, a_2, \ldots, a_{N}) \in \{0,1 \}^N.
\end{align}
Let $ \HW := \{ \wt(\la) \} $ if $\la \ne \la^\circ$ and $ \HW := \{ \al_1, \al_1+ \al_{n-1}+ \al_n +2\sum_{k=2}^{n-2} \al_k \} $ if $\la = \la^\circ$. 
For any $ 1 \le k  \le N$, we define $p_k := 1$ if $ \beta_k \in \HW $ and $p_k := 0$ otherwise, and set 
\begin{align} \label{Eq: pi for D1}
\bfp_\bfi (\la) := (p_t)_{t\in [1,N]} \in \{0,1 \}^N,
\end{align}
where $\beta_k$ is the positive root given in \eqref{Eq: beta_k}.

\subsubsection{Case where $t\in \{ n-1, n \}$}

We choose  $t, t' \in \{n-1, n \}$ so that $\{ t,t' \} = \{ n-1, n \}$. We define the shifted Young diagram 
$\la^\circ := ( n-1,n-2, \ldots, 2,1)$, $N := |\la^\circ|= n(n-1)/2$ and, 
set 
$$
\res(\bx) := 
\begin{cases}
t & \text{ if $i=j$ and $i$ is odd},  \\
t' & \text{ if $i=j$ and $i$ is even},\\
n+i-j-1 & \text{ otherwise,}
\end{cases}
$$
for $\bx = (i,j) \in \la^\circ $.
For $\la \subset \la^\circ$, the weight $\wt(\la)$ and the residue sequence $\res(T)$ are given in \eqref{Eq: wt for la} and \eqref{Eq: res of T}.
Let $  \bfi := \res(T^{\la^\circ})$, which is a reduced expression of the longest element $x_t$ of the set of minimal-length coset representatives for $\weyl_t \backslash \weyl$.
Like the case of type $B_n$, we have the following (see \cite[Section 3.5]{KR08} and \cite[Section 2.2]{P20} for example).

\begin{lemma}  \label{Lem: D SYD and WLat2}
Let $\SYD_t = \{ \la \vdash N \mid \la \subset \la^\circ \}$.
Then the correspondence $ \la \mapsto \La_t - \wt(\la) $ gives a bijection from $\SYD_t$ to the set 
$\weyl \cdot \La_t$.
\end{lemma}
We denote by $\la_{ \beta}$ the shifted Young diagram in $\SYD_t$ corresponding to the weight $\La_t-\beta$.
For $b\in B(\La_t)^*$ and $\la \in \SYD_t$,
We set $\la_b := \la_{\La_t - \wt(b)}$ and denote by $b_\la$ the corresponding element in $ B(\La_t)^*$.
The homogeneous representations $\Specht^\la$ can be realized by using $\ST(\la)$ like the case of type $B_n$ (see \cite[the proof of Corollary 3.5)]{P20}). 
The module $\Specht^\la$ is a homogeneous $R$-module with dimension $|\ST(\la)|$, $\qch(\Specht^\la) = \sum_{T \in \ST(\la)} \res(T)$ and $L(b_\la) \simeq \Specht^\la $.

For a given $ 1 \le k  \le N$,
we take $\bx = (r,c) \in \la^\circ $ such that $ T^{\la^\circ}(\bx)=k $,
and define $a_k := 1$ if $(r,c) \in \la$ and $a_k := 0$ otherwise. We set 
\begin{align} \label{Eq: si for D2}
\bfs_\bfi (\la) := (a_1, a_2, \ldots, a_{N}) \in \{0,1 \}^N.
\end{align}
On the other hand, the shifted Young diagram $\la$ can be decomposed into \emph{two rows}, i.e., 
\begin{align*}
R(\la) = (R_1, R_2, \ldots, R_r),
\end{align*}
where $R_1$ is the first two rows of $\la$; $R_2$ is the first two rows of the shifted Young diagram obtained from $\la$ by removing $R_1$, and so on. 
Note that $\wt(R_k) \in \prD$ for all $k$. Let $ \HW = \{ \wt(R_k) \mid k\in [1,r] \} \subset \prD$. 
For a given $ 1 \le k  \le N$, we define $p_k := 1$ if $ \beta_k \in \HW $ and $p_k := 0$ otherwise, and set 
\begin{align} \label{Eq: pi for D2}
\bfp_\bfi (\la) := (p_t)_{t\in [1,N]} \in \{0,1 \}^N,
\end{align}
where $\beta_k$ is the positive root given in \eqref{Eq: beta_k}.

We are now ready to write the main proposition for type $D_n$.
For any $b\in B(\La_t)^*$ with $t\in \{1,n-1,n\}$, we obtain $\bfs_\bfi(\la_b)$ and $\bfp_\bfi(\la_b)$ in \eqref{Eq: si for D1}, \eqref{Eq: pi for D1}, \eqref{Eq: si for D2} and \eqref{Eq: pi for D2} depending on the index $t$. 
By the same argument of the proof of Proposition \ref{Prop: homogenous YD A} and \cite[Corollary 3.5]{P20}, we have the following.

\begin{prop} \label{Prop: homogenous YD D}
For any $b \in B(\La_t)^*$, we have
\bni
\item $\PBW_\bfi(b) = \bfp_\bfi (\la_b)$,
\item $\STR_\bfi(b) = \bfs_\bfi (\la_b)$,
\item $\PBW_\bfi( \cdual^{-1}(b) ) = -\gN_\bfi^{-1} \circ \gM_\bfi^{-1}(\bfs_\bfi (\la_b)) $,
\item $\STR_\bfi( \cdual(b) ) = -\gM_\bfi \circ \gN_\bfi(\bfp_\bfi (\la_b)) $,
\ee
where $\gM_\bfi$ and $\gN_\bfi$ are the linear maps defined in Proposition \ref{Prop: gM gN}.
\end{prop}

\begin{example} Let $n=5$ and $t \in \IM = \{1,4,5 \}$.
\bni
\item  Suppose that $t=1$. Then $\lambda^\circ = (8)$, $N=8$ and 
the residues for $\lambda^\circ$ are given as
$$
\la^\circ = 
\begin{tabular}{ |c|c|c|c|c|c|c|c|c|c| } 
 \hline
1&2&3&4&5&3&2&1\\
\hline
\end{tabular},
$$
where the residue $\res(\bx)$ of a box $\bx \in \la^\circ$ is the number in the box $\bx$.
This implies that $\wt(\la^\circ) = 2 \al_1 + 2\al_2 + 2\al_3 + \al_4 + \al_5$ and
the initial tableau $T^{\la^\circ}$ is given as
$$
T^{\la^\circ} = 
\begin{tabular}{ |c|c|c|c|c|c|c|c|c|c| } 
 \hline
1&2&3&4&5&6&7&8\\
\hline
\end{tabular},
$$
which says   $\bfi = \res(T^{\la^\circ})= ( 1,2,3,4,5,3,2,1)$.
Note that $\mu = (1,1,1,1)$ and $\wt(\mu) = \al_{1,3} + \al_5$.

\bna
\item Let $\la = ( 5 ) \subset \la^\circ$. We have 
$\wt(\la) = \al_1+ \al_2 + \al_3+ \al_4 + \al_5 = \La_1 - s_3 s_5 s_4s_3s_2 s_1(\La_1),$
and $\bfs_\bfi(\la) = (1,1,1,1,1,0,0,0)$. Since
\begin{align*}
&\beta_1= \al_1, \quad \beta_2= \al_{1,2}, \quad \beta_3= \al_{1,3}, \quad \beta_4= \al_{1,4}, \quad \beta_5= \al_{1,3}+\al_5, \\
&  \beta_6= \al_{1,5}, \quad \beta_7= \al_{1,5}+\al_{3}, \quad 
\beta_8= \al_{1,5}+\al_{2,3}
\end{align*}
that $\bfp_\bfi(\la) = (0,0,0,0,0,1,0,0)$. 
Thus we have
\begin{align*} 
\STR_\bfi( \cdual(b_\la) ) & = - \gM_\bfi \circ \gN_\bfi(\bfp_\bfi (\la))
= -(1,1,0,1,1,1,0,0) \\ 
& \eqfr (0,0,1,0,0,0,1,1),
\end{align*}
where $\eqfr$ is given in \eqref{Eq: equiv for fr}.

\item Let $\la =  \la^\circ$. We have 
$\wt(\la) = 2\al_1+ 2\al_2 + 2\al_3+ \al_4$, which says that 
  $\bfs_\bfi(\la) = (1,1,1,1,1,1,1,1)$ and  $\bfp_\bfi(\la) = (1,0,0,0,0,0,0,1)$.
 Since $b_\la$ coincides with the frozen $\fz_1$, we know that $ \cdual(b_\la) = \fz_1^{-1} \eqfr \one$. 
\ee
\item Suppose that $t=5$.
Then $\lambda^\circ = (4,3,2,1)$, $N=10$ and 
the residues for $\lambda^\circ$ are given as
$$
\la^\circ = 
\begin{tabular}{ cccc| } 
 \hline
 \multicolumn{1}{|c}{5} & \multicolumn{1}{|c}{3} & \multicolumn{1}{|c}{2} & \multicolumn{1}{|c|}{1} \\ 
 \cline{1-4}
 & \multicolumn{1}{|c}{4} & \multicolumn{1}{|c}{3} & \multicolumn{1}{|c|}{2} \\
  \cline{2-4}
 &   & \multicolumn{1}{|c}{5}  & \multicolumn{1}{|c|}{3}  \\ 
  \cline{3-4}
  &&& \multicolumn{1}{|c|}{4}  \\ 
  \cline{4-4}
\end{tabular},
$$
where the number in a box $\bx \in \la^\circ$ is the residue $\res(\bx)$.
We thus have $\wt(\la^\circ) = \al_1 + 2\al_2 + 3\al_3 + 2\al_4+ 2\al_5$ and
the initial tableau $T^{\la^\circ}$ is give as
$$
T^{\la^\circ} = 
\begin{tabular}{ cccc| } 
 \hline
 \multicolumn{1}{|c}{1} & \multicolumn{1}{|c}{2} & \multicolumn{1}{|c}{3} & \multicolumn{1}{|c|}{4} \\ 
 \cline{1-4}
 & \multicolumn{1}{|c}{5} & \multicolumn{1}{|c}{6} & \multicolumn{1}{|c|}{7} \\
  \cline{2-4}
 &   & \multicolumn{1}{|c}{8}  & \multicolumn{1}{|c|}{9}  \\ 
  \cline{3-4}
  &&& \multicolumn{1}{|c|}{10}  \\ 
  \cline{4-4}
\end{tabular},
$$
which says   $\bfi = \res(T^{\la^\circ})= ( 5,3,2,1, 4,3,2, 5,3, 4)$.

Let $\la = (4,2,1) \subset \la^\circ$, i.e.,
$$
\la  = 
\begin{tabular}{ cccc| } 
 \hline
 \multicolumn{1}{|c}{5} & \multicolumn{1}{|c}{3} & \multicolumn{1}{|c}{2} & \multicolumn{1}{|c|}{1} \\ 
 \cline{1-4}
 & \multicolumn{1}{|c}{4} & \multicolumn{1}{|c|}{3}  \\
  \cline{2-3}
 &   & \multicolumn{1}{|c|}{5}    \\ 
  \cline{3-3}
\end{tabular},
$$
where the numbers in boxes are residues.
We then have  $\wt(\la) = \al_1+ \al_2 + 2\al_3+ \al_4 + 2\al_5$, and 
$$
\bfs_\bfi(\la) = (1,1,1,1, 1,1,0, 1,0,0).
$$
The 2-rows decomposition of $\la$ is $R(\la) = (R_1, R_2)$, where
$$
R_1 =
\begin{tabular}{ cccc| } 
 \hline
 \multicolumn{1}{|c}{5} & \multicolumn{1}{|c}{3} & \multicolumn{1}{|c}{2} & \multicolumn{1}{|c|}{1} \\ 
 \cline{1-4}
 & \multicolumn{1}{|c}{4} & \multicolumn{1}{|c|}{3}  \\
  \cline{2-3}
\end{tabular}
, \qquad 
R_2 =\begin{tabular}{ |c| } 
 \hline
 5   \\ 
\hline
\end{tabular}.
$$
Then  $\wt(R_1) = \al_{1,5} + \al_3$,  $\wt(R_2) = \al_{5}$ and 
\begin{align*}
&\beta_1= \al_5, \quad \beta_2= \al_{3}+\al_5, \quad \beta_3= \al_{2,3}+\al_5, \quad \beta_4= \al_{1,3}+\al_5, \quad \beta_5= \al_{3,5}, \\
&\beta_6= \al_{2,5}, \quad \beta_7= \al_{1,5}, \quad \beta_8= \al_{2,5}+\al_3, \quad \beta_9= \al_{1,5}+\al_{3},\quad \beta_{10}= \al_{1,5}+\al_{2,3}.
\end{align*}
Hence We obtain $\bfp_\bfi(\la) = (1,0,0,0,0,0,0,0,1,0)$ and
\begin{align*} 
\STR_\bfi( \cdual(b_\la) ) & = - \gM_\bfi \circ \gN_\bfi(\bfp_\bfi (\la))
= -(1,1,0,1,1,0,1,1,1,0) \\ 
& \eqfr (0,0,1,0,0,1,0,0,0,1),
\end{align*}
where $\eqfr$ is given in \eqref{Eq: equiv for fr}.
\ee
\end{example}

\vskip 2em

\section{Periodicity for twist automorphisms} \label{Sec: periodicity}

In this section, we investigate the periodicity of the twist automorphism $\eta_w$ using the operator $\cdual_w$ (see  \eqref{Eq: Dw and etaw}).

Let $w \in \weyl$ with $\supp(w) = I$. We define $\PD(w)$ to be the smallest positive integer $k$ such that 
$$
\cdual_w^{k} (x) \eqfr x \qquad \text{ for all }x\in \LB(w), 
$$
where $\eqfr$ is given in \eqref{Eq: equiv for fr}. If such a number exists, then we call $\PD(w)$ the \emph{periodicity} of $\cdual_w$. Otherwise, we write $\PD(w) = \infty $.
Since $\cdual_w$ is the crystal-theoretic counterpart of the twist automorphism $\eta_w$ (see \eqref{Eq: Dw and etaw}), the number $\PD(w)$ is also the periodicity of $\eta_w$.

Recall the initial seed  $\seed_\bfi = ( \{ \qmD_k^{\bfi} \}_{k\in J}, \tB) $ of $\qA_q(\n(w))$ given in Section \ref{Sec: MC}.
\begin{prop}\label{prop: period of minors determines total period}
Let $\bfi \in R(w)$. For any $k\in [1, \ell(w)]$, define $\xi_k$ to be the smallest positive integer $p$ such that $ \cdual^p ( \qmD_{k}^\bfi )  \eqfr \qmD_{k}^\bfi$ if it exists, and set $\xi_k :=\infty$ otherwise. Then we have 
$$
\PD(w) = \lcm\{ \xi_k \mid k\in [1, \ell(w)]\},
$$
where $ \lcm (A)$ is the least common multiple of a set $A \in \Z_{>0} \cup \{ \infty\}$. Here we understand $ \lcm (A) = \infty$ if $\infty$ belongs to $A$.
\end{prop}
\begin{proof}
To prove the assertion, we shall use the basis $\tGs(w)$ of $\tqA_q(\n(w))$ arising from simple $R$-modules. 
Let $J := [1, \ell(w)]$ and divide into $J = J_\uf \sqcup J_\fr$ as in Section \ref{Sec: MC}. 
 Let $\dM_k^\bfi$ be the determinantial module such that $[\dM_k^\bfi] = \qmD_k^\bfi$ for $k\in J$. 
For simplicity, we write $\qmD_{k}$ and $\dM_k$ instead of $\qmD_{k}^\bfi$ and $\dM_k^\bfi$.

Let $t$ be a common multiple of $A:=\{\xi_k \mid k\in J \}$. 
As it is obvious that $ t = \infty = \PD(w) $ if $\infty \in A$, we assume that $\infty \notin A$, i.e, $t < \infty$.
By the assumption and \eqref{Eq: Dw and etaw}, we have 
\begin{align} \label{Eq: Da Mk}
\dual^t (\dM_k) \simeq \dM_k \circ \dM^\bfi( \bfm_k) \qquad
\text{for some } \bfm_k \in \sum_{j\in J_\fr} \Z e_{j}
\end{align}
where $\{ e_{j} \}_{j\in J}$ is the standard basis of $\Z^J$ and $\dM^\bfi( \bfc)$ is defined in Proposition \ref{Prop: GrGl for Rgmod}.

Let $X \in \tGs(w)$. Since $\xi_j = 1$ for any $j\in J_\fr$ by construction, we may assume that $X \in \Gs(w)$. By Proposition \ref{Prop: GrGl for Rgmod}, 
there exist $\bfa, \bfb \in \Z_{\ge0}^{J}$ such that 
$$
 X \hconv \dM (\bfa)  \simeq \dM (\bfb).
$$
Applying $\dual^t$ to the above isomorphism, we have
$$
\dual^t ( X\hconv \dM(\bfa) ) \simeq \dual^t ( X )\hconv  \dual^t( \dM(\bfa) ))  \simeq   \dual^t (\dM)(\bfb),
$$
which tells us that, by \eqref{Eq: Da Mk}, 
$$
 \dual^t(X) \hconv \dM (\bfa')  \simeq \dM (\bfb'),
$$
for some $\bfa', \bfb' \in \Z^{J} $ such that 
$ \bfa- \bfa', \bfb- \bfb' \in \sum_{j\in J_\fr} \Z e_{j} $.
Proposition \ref{Prop: GrGl for Rgmod} says that 
$$
\gR_\bfi (\dual^t(X)) - \gR_\bfi ( X) \in \sum_{j\in J_\fr} \Z e_{j}, 
$$
which yields that $ \dual^t(X) \eqfr X $. Thus we have the assertion by taking $t=\lcm A$.
\end{proof}

\subsection{Periodicity for minuscule cases of type $A_n$}
In this subsection, we assume that $\g$ is of finite type $A_n$. 

From now on to the end of the section,  we choose and fix $t \in I$ and set
\begin{align} \label{Eq: xt}
\text{$x_t$ := the longest element of the set of minimal-length coset representatives for $\weyl_t \backslash \weyl$.}
\end{align}
We define $l := \ell(x_t)$.
Note that the sequence $\bfi$ given in \eqref{Eq: la circ and bfi An} is a reduced expression of $x_t$. We simply write $\cdual$ for $\cdual_{x_t}$.

\subsubsection{String parametrization and rectangular vectors}
In this subsubsection, we further develop the expression for $\STR_\bfi(\cdual(b))$ given in Proposition \ref{Prop: homogenous YD A}~(iv). We also introduce a family of rectangular vectors in $\tcS_\bfi(x_t)$ that is preserved under the twist automorphism $\mathcal{D}$ and describe an explicit rule for its action.
When $\bfa = \STR_\bfi(b)$, we often write $\cdual(\bfa)$ for $\STR_\bfi ( \cdual(b) )$ if no confusion arises.

We extend the definition $\res(i,j) = t - i + j$ to all $i \ge 0$ and $j \ge 0$, and set $\Lambda_0 = \Lambda_{n+1} = 0$. Then we obtain the following relation.
\begin{lemma}\label{lem: simple reflection acts on fundamental weights}
Let $a,b \in \Z_{\ge 1}$. Then the following statements hold.
\bni
\item For any $i,j\in \Z_{\ge1}$, we have
\[
s_{\res(i, j)} \Lambda_{\res(a, b)} =
\begin{cases}
\Lambda_{\res(a, b) - 1} - \Lambda_{\res(a, b)} + \Lambda_{\res(a, b) + 1} & \text{if } i - j = a - b, \\
\Lambda_{\res(a, b)} & \text{otherwise}.
\end{cases}
\]

\item 
$
s_{\res(a, b)} \left( \Lambda_{\res(a, b)} - \Lambda_{\res(a-1, b)} \right)
= \Lambda_{\res(a, b-1)} - \Lambda_{\res(a-1, b-1)}.
$
\ee
\end{lemma}

\begin{proof}
(i) follows directly from the action of simple reflections on fundamental weights:
\[
s_i \Lambda_j =
\begin{cases}
\Lambda_{i-1} - \Lambda_i + \Lambda_{i+1} & \text{if } i = j, \\
\Lambda_j & \text{otherwise}.
\end{cases}
\]

(ii) follows from (i) together with the identities
\[
    \res(a,b) - 1 = \res(a, b - 1),  \
    \res(a,b) = \res(a - 1, b - 1), \ \text{ and } \
    \res(a,b) + 1 = \res(a - 1, b).
\]
\end{proof}

For any $p \in [1, l]$, we define 
\begin{align} \label{Eq: ipjp}
(i^p, j^p) := (T^{\la_t^\circ})^{-1} (p).
\end{align}
Note that $(i^p, j^p) \in \la_t^\circ$ such that $T^{\la_t^\circ}(i^p, j^p) = p$.
    Now fix $1 \le k \le l$, and define the function
    \[
    \mathfrak{N}_k:[1,k] \to \wlP, \quad p \mapsto s_{i_{p+1}} s_{i_{p+2}} \cdots s_{i_k} \Lambda_{i_k}.
    \]
    Then $\mathfrak{N}_k$ satisfies the recursion
    \[
    s_{i_p} \mathfrak{N}_k(p) = \mathfrak{N}_k(p-1).
    \]
    For any $p\in [1,l]$, we extend the definition of $\mathfrak{N}_k$ to such boxes by setting $\mathfrak{N}_k(i^p, j^p) := \mathfrak{N}_k(p)$.

    \begin{lemma}\label{lem: weight associated to box}
    Let $1 \le p \le k$. Then
    \[
    \mathfrak{N}_k(i^p, j^p) =
    \begin{cases}
    \Lambda_{\res(i^k,0)} - \Lambda_{\res(i^p,0)} + \Lambda_{\res(i^p,j^p)} - \Lambda_{\res(i^p - 1, j^p)} + \Lambda_{\res(i^p - 1, j^k)} & \text{if } j^p < j^k, \\
    \Lambda_{\res(i^k,0)} - \Lambda_{\res(i^p,0)} + \Lambda_{\res(i^p,j^k)} & \text{if } j^p \ge j^k,
    \end{cases}
    \]
    where $(i^p, j^p)$ and $(i^k, j^k)$ are given in \eqref{Eq: ipjp}.
    \end{lemma}
\begin{proof}
We proceed by induction on $d := i^k - i^p$.

If $d = 0$, then $i^p = i^k$. Applying Lemma~\ref{lem: simple reflection acts on fundamental weights}~(ii) along the row, we obtain
\begin{align*}
\mathfrak{N}_k(i^k, j^p) 
&= s_{\res(i^k, j^p+1)} \cdots s_{\res(i^k, j^k)} \Lambda_{\res(i^k, j^k)} \\
&= \Lambda_{\res(i^k, j^p)} - \Lambda_{\res(i^k - 1, j^p)} + \Lambda_{\res(i^k - 1, j^k)}.
\end{align*}

Now assume the formula holds for all $d \le s$. We prove the claim for $d = s+1$, that is, $i^p = i^k - s - 1$. We consider the following cases based on the value of $j^p$:

\smallskip
\noindent
\textbf{Case 1.} Suppose $j^p = n - t + 1$. By the induction hypothesis and Lemma~\ref{lem: simple reflection acts on fundamental weights}~(ii), we have
\begin{align*}
\mathfrak{N}_k & (i^k - s - 1, n - t + 1) 
 = s_{\res(i^{p+1}, j^{p+1})} \, \mathfrak{N}_k(i^{p+1}, j^{p+1}) = s_{\res(i^k - s, 1)} \, \mathfrak{N}_k(i^k - s, 1) \\
&  = s_{\res(i^k - s, 1)} \Big( \Lambda_{\res(i^k, 0)} - \Lambda_{\res(i^k - s, 0)} 
+ \Lambda_{\res(i^k - s, 1)} - \Lambda_{\res(i^k - s - 1, 1)} 
+ \Lambda_{\res(i^k - s - 1, j^k)} \Big) \\
&= \Lambda_{\res(i^k, 0)} - \Lambda_{\res(i^k - s - 1, 0)} + \Lambda_{\res(i^k - s - 1, j^k)}.
\end{align*}

\smallskip
\noindent
\textbf{Case 2.} Suppose $j^k \le j^p < n - t + 1$. Let $q := j^p - j^k$. Then by Lemma~\ref{lem: simple reflection acts on fundamental weights}~(i), we obtain
\begin{align}\label{eq: case2}
\mathfrak{N}_k(i^k - s - 1, j^k + q) 
&= s_{\res(i^k - s - 1, j^k + q + 1)} \cdots s_{\res(i^k - s - 1, n - t + 1)} \, \mathfrak{N}_k(i^k - s - 1, n - t + 1) \notag \\
&= \Lambda_{\res(i^k, 0)} - \Lambda_{\res(i^k - s - 1, 0)} + \Lambda_{\res(i^k - s - 1, j^k)}.
\end{align}

\smallskip
\noindent
\textbf{Case 3.} Suppose $1 \le j^p < j^k$, and let $q=j^k - j^p $. Applying $s_{\res(i^k - s - 1, j^k)}$ to~\eqref{eq: case2}, we have
\begin{align*}
\mathfrak{N}_k & (i^k - s - 1, j^k - 1) 
= s_{\res(i^k - s - 1, j^k)} \mathfrak{N}_k(i^k - s - 1, j^k) \\
&= \Lambda_{\res(i^k, 0)} - \Lambda_{\res(i^k - s - 1, 0)} 
+ \Lambda_{\res(i^k - s - 1, j^k - 1)} - \Lambda_{\res(i^k - s - 2, j^k - 1)} 
+ \Lambda_{\res(i^k - s - 2, j^k)},
\end{align*}
by Lemma~\ref{lem: simple reflection acts on fundamental weights}~(ii).

Finally, by applying the sequence of reflections 
$s_{\res(i^k - s - 1, j^k - q + 1)} \cdots s_{\res(i^k - s - 1, j^k - 1)}$ 
and using Lemma~\ref{lem: simple reflection acts on fundamental weights}~(ii) at each step, we obtain
\begin{align*}
\mathfrak{N}_k(i^k - s - 1, j^k - q) 
= & \Lambda_{\res(i^k, 0)} - \Lambda_{\res(i^k - s - 1, 0)} 
+ \left( \Lambda_{\res(i^k - s - 1, j^k - q)} - \Lambda_{\res(i^k - s - 2, j^k - q)} \right) \\
& + \Lambda_{\res(i^k - s - 2, j^k)},
\end{align*}
which completes the proof. 
\end{proof}

\begin{remark}
In \cite[Corollary 5.7]{T24}, a formula for $w \cdot \Lambda_i$ for $w \in \weyl$ is given using the combinatorics of a poset associated with Young diagrams. Lemma \ref{lem: weight associated to box} also can be proved by using this result. 
\end{remark}

Let $\{e_k\}_{k \in [1,l]}$ denote the standard basis of $\mathbb{Z}^{[1,l]}$. For each $(i,j) \in \la_t^\circ$, define $e_{i,j} = e_{T^{\la_t^\circ}(i,j)}$. 
For each $k \in [1,l]$, define the vectors
\begin{align} \label{Eq: bkrk}
b_k = \sum_{\substack{i \leq i^k \\ j \leq j^k}} e_{i,j},  
\quad  
r_k = \sum_{\substack{i \leq i^k \\ j \leq j^k \\ i = i^k \text{ or } j = j^k}} e_{i,j},
\end{align}
where $(i^k, j^k)$ is given in \eqref{Eq: ipjp}.
We also set $b_{i^k, j^k} := b_k$ and $r_{i^k, j^k} := r_k$.

\begin{lemma}\label{lem: col vect of M_i and M_i N_i}  
Let $k \in [1, l]$. Then
\bni
    \item The $k$-th column of $M_\bfi$ is equal to $b_k$.  
    \item The $k$-th column of $M_\bfi \cdot N_\bfi$ is equal to $r_k$.  
\ee
\end{lemma}
\begin{proof}
(i) Let $M_k$ denote the $k$-th column of $M_\bfi$. By the definition of $M_\bfi$, the $p$-th entry of $M_k$ vanishes for all $p > k$, and for $1 \le p \le k$, it is given by
$
\langle h_{\res(i^p, j^p)},\, \mathfrak{N}_k(i^p, j^p) \rangle.
$

By Lemma \ref{lem: weight associated to box}, we compute
\begin{equation*}
\left\langle h_{\res(i^p, j^p)},\ \mathfrak{N}_k(i^p, j^p) \right\rangle =
\begin{cases}
    1 & \text{if } j^p \le j^k, \\
    0 & \text{if } j^p > j^k.
\end{cases}
\end{equation*}        
Therefore, we conclude that $M_k = b_k$, as desired.

(ii)  
By the definition of $N_\bfi$, the $k$-th column of $N_\bfi$ is given by
$
\begin{cases}
e_k - e_{k^-} & \text{if } k^- \neq 0, \\
e_k & \text{if } k^- = 0.
\end{cases}
$
By (i), the $k$-th column of $M_\bfi \cdot N_\bfi$ is
$
\begin{cases}
b_k - b_{k^-} & \text{if } k^- \neq 0, \\
b_k & \text{if } k^- = 0.
\end{cases}
$
If $k^- \ne 0$, then $T^{\lambda_t^\circ}(i^k - 1,\ j^k - 1) = k^-$, and hence $b_k - b_{k^-} = r_k$.  
If $k^- = 0$, then either $i^k = 1$ or $j^k = 1$, so by definition $b_k = r_k$. In both cases, the $k$-th column of $M_\bfi \cdot N_\bfi$ is equal to $r_k$. This completes the proof. 
\end{proof}

Let $\lambda \subset \lambda_t^{\circ}$. Recall that 
$
\HW_t(\lambda) = \{ \wt(H_k) \mid 1 \le k \le r \} \subset \prD,
$
where $H(\lambda) = (H_1, H_2, \dots, H_r)$ denotes the diagonal hook decomposition of $\lambda$.
Combining Proposition~\ref{Prop: homogenous YD A}~(iv) with Lemma~\ref{lem: col vect of M_i and M_i N_i}, we obtain the following.

\begin{theorem}\label{thm: minuscule, D of string parametrization}
    Let $b \in B(\La_t)^*$ with $\wt(b) = \La_t - \beta$. Then, we have  
    \[
    \STR_\bfi( \cdual(b) ) = -\sum_{\{k \mid \beta_k \in \HW_t(\lambda_{\beta})\}} r_k.
    \] 
\end{theorem}  

\begin{example}
We keep all notations appearing in Example~\ref{ex: Type A (3,2,2) diagonal hook decomposition}.
\bna
\item 
By the definition \eqref{Eq: bkrk}, we have 
\[
b_4 = (1,0,0,1,0,0,0,0,0), \quad  
b_9 = (1,1,1,1,1,1,1,1,1),
\]
\[
r_4 = (1,0,0,1,0,0,0,0,0), \quad  
r_9 = (0,0,1,0,0,1,1,1,1).
\]
\item 
Let $\lambda = (3,2,2)$ and set $b := b_\lambda \in B(\La_t)^*$.
As described in Example~\ref{ex: Type A (3,2,2) diagonal hook decomposition}, the diagonal hook decomposition of $\lambda$ is given by
\[
H(\lambda) = (H_1, H_2), \quad
H_1 =
\begin{array}{|c|c|c|}
\hline
3 & 4 & 5 \\
\cline{1-3}
2 \\
\cline{1-1}
1 \\
\cline{1-1}
\end{array}
\quad
H_2 =
\begin{array}{|c|}
\hline
3 \\
\hline
2 \\
\hline
\end{array}.
\]

Thus, $\HW_t(\lambda) = \{ \beta_4, \beta_9 \}$. Applying Theorem~\ref{thm: minuscule, D of string parametrization}, we obtain
\[
\STR_\bfi( \cdual(b) ) = -r_4 - r_9 = -(1,0,1,1,0,1,1,1,1),
\]
which agrees with the computation in Example~\ref{ex: Type A (3,2,2) diagonal hook decomposition}.
\ee
\end{example}

The following lemmas will be used in the proof of Proposition \ref{prop: rule for D minuscule type A} and again later.
\begin{lemma}\label{lem: string of qmD_bfi, k}
For each $k \in [1, l]$, let $d_k$ be the crystal element such that $\Gup(d_k) = \qmD_{k}^\bfi$. Then we have
$
\STR_\bfi(d_k) = b_k.
$
\end{lemma}
\begin{proof}
By the definition of $\gL_\bfi$, we have $\gL_\bfi(\qmD_k^{\bfi}) = e_k$. Thus the assertion follows from Lemma  \ref{lem: col vect of M_i and M_i N_i} (i).
\end{proof}

\begin{lemma}\label{lem: type A minuscule, STR_i(d_k) frozen} For each $k \in [1, l]$, the quantum minor $\qmD_{k}^\bfi$ is frozen if and only if the corresponding box $(i^k, j^k) \in \la_t^\circ$ satisfies $i^k = t$ or $j^k = n - t + 1$.
\end{lemma}
\begin{proof}
    The assertion follows directly from the definition of $\bfi$ given in \eqref{Eq: la circ and bfi An}.
\end{proof}

In general, the set $B(\Lambda_t)^*$ is not closed under the action of $\cdual$. For instance, let $n = 7$, $t=4$ and $\lambda = (4,3,2,1) \in \YD_4$, and take $b \in B(\Lambda_t)^*$ such that $\wt(b) = \Lambda_t - \wt(\lambda)$. Then, by Theorem \ref{thm: minuscule, D of string parametrization}, we have
\[
\STR_\bfi( \cdual(b) ) = -r_6 - r_{16} = (0, -1, 0, -1, -1, -1, 0, -1, 0, 0, 0, -1, -1, -1, -1, -1).
\]
Note that there is no $b' \in B(\Lambda_t)^*$ such that $\STR_\bfi(b') \eqfr \STR_\bfi( \cdual(b) )$.

We now consider the family consisting of vectors in $\tcS_\bfi(w)$ which are indexed by rectangular regions.
This family is preserved under the action of $\cdual$ and provides a convenient framework for computing $\PD(x_t)$.
For $a,b,c,d \ge 0$ such that $0 \le a +c \le t$ and $0 \le b+d \le n-t+1$,
define
\[
R_{(a,b), (c,d)} := \sum_{\substack{a < i \le a + c \\ b < j \le b + d}} e_{i,j} \in \tcS_\bfi(x_t),
\]
and set
\[
\Gamma := \{ R_{(a,b), (c,d)} \mid a = 0 \text{ or } b = 0 \}.
\]

\begin{prop}\label{prop: rule for D minuscule type A}
The set $\Gamma$ is preserved under the action of $\mathcal{D}$. More precisely, for any $R_{(a,b), (c,d)} \in \Gamma$, we have
\bni
\item If $a = 0$, then
$
\mathcal{D}(R_{(0,b), (c,d)}) \eqfr R_{(c - \min(b,c), b - \min(b,c)), (t - c, d)}.
$

\item If $b = 0$, then
$
\mathcal{D}(R_{(a,0), (c,d)}) \eqfr R_{(a - \min(a,d), d - \min(a,d)), (c, n - t + 1 - d)}.
$
\ee
\end{prop}
\begin{proof}
(i) Suppose $b \ge c$. By Lemma \ref{lem: type A minuscule, STR_i(d_k) frozen}, we have
\[
R_{(0,b), (c,d)} \eqfr R_{(0,b),(c,d)} + b_{t,b} = \bfs_\bfi(\lambda),
\]
where $\lambda = ( 
    \underbrace{b + d, \ldots, b + d}_{c\ \text{times}}, 
    \underbrace{b, \ldots, b}_{(t - c)\ \text{times}})$. Here $b_{i,j}$ is given below \eqref{Eq: bkrk}, and $\bfs_\bfi(\lambda)$ is given in \eqref{Eq: def si}.

Let $H(\lambda) = (H_1, \ldots, H_r)$ be the diagonal hook decomposition of $\lambda$. Set
\[
\mathcal{H}_1 := \{H_1, \ldots, H_c\}, \qquad \mathcal{H}_2 := \{H_{c+1}, \ldots, H_r\}.
\]
Then $\mathcal{H}_1$ corresponds to the Young diagram of shape $(b + d, \ldots, b + d, c, \ldots, c)$, and $\mathcal{H}_2$ corresponds to a $(t-c) \times (b-c)$ rectangle whose upper-left corner lies at $(c,c)$.

    By Theorem \ref{thm: minuscule, D of string parametrization}, we have
    \[
    \cdual(R_{(0,b), (c,d)}) \eqfr -\sum_{\{k \mid \beta_k \in \wt(\mathcal{H}_1)\}} r_k - \sum_{\{k \mid \beta_k \in \wt(\mathcal{H}_2)\}} r_k.
    \]
    For a hook $H$ of shape $(a, 1^b)$, we have $\wt(H) = \alpha_{t - b,\, t + a - 1}$. Since $\beta_k = \alpha_{t+1 - i^k,\, t - 1 + j^k}$, it follows that $\wt(H) = \beta_k$ implies $(i^k, j^k) = (1 + b, a)$.
    
    Hence we obtain
    \begin{align*}
    \sum_{\{k \mid \beta_k \in \wt(\mathcal{H}_1)\}} r_k &= R_{(0,0), (t, b + d)} - R_{(0,0), (t - c, b + d - c)}, \\
    \sum_{\{k \mid \beta_k \in \wt(\mathcal{H}_2)\}} r_k &= R_{(0,0), (t - c, b - c)}.
    \end{align*}
    
    Therefore
    \[
    \cdual(R_{(0,b), (c,d)}) \eqfr -R_{(0,0), (t, b + d)} + R_{(0,0), (t - c, b + d - c)} - R_{(0,0), (t - c, b - c)}.
    \]
    Adding the frozen variable $R_{(0,0), (t, b + d)}$, we find
    \[
    \cdual(R_{(0,b), (c,d)}) \eqfr R_{(0, b - c), (t - c, d)},
    \]
    which proves the claim when $b \ge c$.
    
    If $b < c$, a similar argument yields
    \[
    \mathcal{D}(R_{(0,b), (c,d)}) \eqfr R_{(c - b, 0), (t - c, d)}.
    \]
    The assertion follows from the two cases above.
    
    (ii) The proof is analogous and thus omitted.
    \end{proof}

    \begin{remark}
        \begin{enumerate}
            \item Proposition \ref{prop: rule for D minuscule type A} describes the action of the operator $\mathcal{D}$ on the set $\Gamma$ in terms of rectangular configurations.
            \begin{itemize}
                \item In case~(i), where $a = 0$, the rectangle corresponding to $R_{(0,b), (c,d)}$ is the $c \times d$ rectangle with upper-left corner at $(0,b)$, attached along the top edge of $\la_t^{\circ}$. Consider its \emph{column-complementary rectangle}, which has shape $(t - c) \times d$ and upper-left corner at $(c,b)$. The action of $\mathcal{D}$ can be interpreted as translating this complementary rectangle diagonally in the north-west direction until it meets either the top edge or the left edge. The number of such diagonal steps is equal to $\min(b, c)$.
                \item In case~(ii), where $b = 0$, the rectangle corresponding to $R_{(a,0), (c,d)}$ is the $c \times d$ rectangle with upper-left corner at $(a,0)$, attached along the left edge of $\la_t^{\circ}$. Consider its \emph{row-complementary rectangle}, which has shape $c \times (n - t + 1 - d)$ and upper-left corner at $(a,d)$. The action of $\mathcal{D}$ can be interpreted as translating this complementary rectangle diagonally in the north-west direction until it meets either the top edge or the left edge. The number of such diagonal steps is equal to $\min(a, d)$.
            \end{itemize}
            \item If $a = b = 0$, then either (i) or (ii) rule of Proposition \ref{prop: rule for D minuscule type A} can be applied to describe $\mathcal{D}(R_{(0,0),(c,d)})$. Both rule gives the same element up to frozen variables.
        \end{enumerate}
            \end{remark}

    \begin{example}\label{ex: rectangle rule for D}
        Let $n = 5$ and $t = 3$. By Proposition \ref{prop: rule for D minuscule type A}, we compute $\mathcal{D}^k(R_{(0,0),(1,2)})$ for $k = 1, 2, 3, 4$ as follows:
        \begin{align*}
            \mathcal{D}^1(R_{(0,0),(1,2)}) &\eqfr R_{(0,2),(1,1)}, &
            \mathcal{D}^2(R_{(0,0),(1,2)}) &\eqfr R_{(0,1),(2,1)}, \\
            \mathcal{D}^3(R_{(0,0),(1,2)}) &\eqfr R_{(1,0),(1,1)}, &
            \mathcal{D}^4(R_{(0,0),(1,2)}) &\eqfr R_{(0,0),(1,2)}.
        \end{align*}
        
        The diagrams below illustrate $\mathcal{D}^k(R_{(0,0),(1,2)})$ for $k = 0, 1, 2, 3, 4$. Each diagram represents an element of $\tcS_\bfi(x_t)$, visualized as a vector supported on a rectangle. If the shaded rectangle has size $c \times d$ and its upper-left corner is located at position $(a,b)$, then the diagram corresponds to the vector $R_{(a,b),(c,d)}$.
        
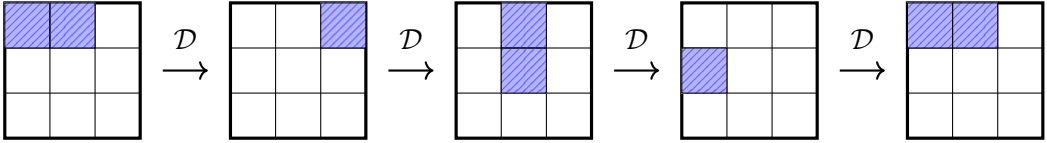
        \begin{center}
            \begin{tikzpicture}[scale=0.6]
                % Box 0
                \begin{scope}[shift={(0,0)}]
                    \draw[very thick] (0,0) rectangle (3,3);
                    \foreach \i in {1,2} {
                        \draw[thin] (\i,0) -- (\i,3);
                        \draw[thin] (0,\i) -- (3,\i);
                    }
                    \fill[blue!30] (0,2) rectangle (1,3);
                    \fill[blue!30] (1,2) rectangle (2,3);
                    \draw[pattern=north east lines, pattern color=blue!60] (0,2) rectangle (1,3);
                    \draw[pattern=north east lines, pattern color=blue!60] (1,2) rectangle (2,3);
                \end{scope}
                
                \draw[->, thick] (3.5,1.5) -- ++(1,0);
                \node at (4,2.2) {$\cdual$};
            
                % Box 1
                \begin{scope}[shift={(5,0)}]
                    \draw[very thick] (0,0) rectangle (3,3);
                    \foreach \i in {1,2} {
                        \draw[thin] (\i,0) -- (\i,3);
                        \draw[thin] (0,\i) -- (3,\i);
                    }
                    \fill[blue!30] (2,2) rectangle (3,3);
                    \draw[pattern=north east lines, pattern color=blue!60] (2,2) rectangle (3,3);
                \end{scope}
                
                \draw[->, thick] (8.5,1.5) -- ++(1,0);
                \node at (9,2.2) {$\cdual$};
            
                % Box 2
                \begin{scope}[shift={(10,0)}]
                    \draw[very thick] (0,0) rectangle (3,3);
                    \foreach \i in {1,2} {
                        \draw[thin] (\i,0) -- (\i,3);
                        \draw[thin] (0,\i) -- (3,\i);
                    }
                    \fill[blue!30] (1,2) rectangle (2,3);
                    \fill[blue!30] (1,1) rectangle (2,2);
                    \draw[pattern=north east lines, pattern color=blue!60] (1,2) rectangle (2,3);
                    \draw[pattern=north east lines, pattern color=blue!60] (1,1) rectangle (2,2);
                \end{scope}
                
                \draw[->, thick] (13.5,1.5) -- ++(1,0);
                \node at (14,2.2) {$\cdual$};
            
                % Box 3
                \begin{scope}[shift={(15,0)}]
                    \draw[very thick] (0,0) rectangle (3,3);
                    \foreach \i in {1,2} {
                        \draw[thin] (\i,0) -- (\i,3);
                        \draw[thin] (0,\i) -- (3,\i);
                    }
                    \fill[blue!30] (0,1) rectangle (1,2);
                    \draw[pattern=north east lines, pattern color=blue!60] (0,1) rectangle (1,2);
                \end{scope}
                
                \draw[->, thick] (18.5,1.5) -- ++(1,0);
                \node at (19,2.2) {$\cdual$};
            
                % Box 4
                \begin{scope}[shift={(20,0)}]
                    \draw[very thick] (0,0) rectangle (3,3);
                    \foreach \i in {1,2} {
                        \draw[thin] (\i,0) -- (\i,3);
                        \draw[thin] (0,\i) -- (3,\i);
                    }
                    \fill[blue!30] (0,2) rectangle (1,3);
                    \fill[blue!30] (1,2) rectangle (2,3);
                    \draw[pattern=north east lines, pattern color=blue!60] (0,2) rectangle (1,3);
                    \draw[pattern=north east lines, pattern color=blue!60] (1,2) rectangle (2,3);
                \end{scope}
            
            \end{tikzpicture}
        \end{center}
    \end{example}

\subsubsection{Lattice path interpretation and periodicity}
We introduce a lattice path model that describes the inverse action of the twist automorphism $\mathcal{D}$ on certain elements of $\tcS_\bfi(x_t)$. This model gives a combinatorial method to track the orbit of $\qmD_k^{\bfi}$ in the GLS seed under the action of $\mathcal{D}^{-1}$, and thereby determines $\PD(x_t)$ for all $t \in [1,n]$.

For \( 0 < i < t \) and \( 0 < j < n - t + 1 \), we set
$$
i^c := t - i \qtq j^c := n - t + 1 - j.
$$
Note that 
$(i^c)^c = i$ and $(j^c)^c = j.$

Consider the set of lattice points $\mathcal{Z} = \mathbb{Z}_{> 0}^2 \subset \mathbb{R}^2$. For convenience in describing the correspondence between string parametrizations and lattice points later, we set the positive $x$-axis to point downward and the positive $y$-axis to point rightward.
\begin{definition}
    In $\mathbb{R}^2$, we define the following families of lines.
    \bni
        \item A \emph{horizontal line} is a line of the form $x = t \cdot m$ for some $m \in \mathbb{Z}_{> 0}$. It is called an \emph{odd horizontal line} if $m$ is odd, and an \emph{even horizontal line} if $m$ is even.
        
        \item A \emph{vertical line} is a line of the form $y = (n - t + 1) \cdot m$ for some $m \in \mathbb{Z}_{> 0}$. It is called an \emph{odd vertical line} if $m$ is odd, and an \emph{even vertical line} if $m$ is even.
    \ee
    \end{definition}

For $(x,y) \in \mathcal{Z}$, there are $m, u \in \mathbb{Z}_{\ge 0}$ satisfying 
\[
t \cdot m < x \le t \cdot (m+1)\quad\text{and}\quad (n - t + 1) \cdot  u  < y \le (n - t + 1) \cdot ( u  + 1).
\] 
We define $\kappa(x,y) :=(p,q)$ by
\[
p = \begin{cases}
i, & \text{if $m$ is even,}\\
i^c, & \text{if $m$ is odd,}
\end{cases}
\quad\text{and}\quad
q = \begin{cases}
j, & \text{if $ u $ is even,}\\
j^c, & \text{if $ u $ is odd.}
\end{cases}
\]

Now, we set 
$$p_0 := (i,j)$$ 
and construct a lattice path $L$ starting from $p_0 = (i,j)$.  
Given $p_n = (x,y)$ and a path $L_n$ from $p_0$ to $p_n$, we define $p_{n+1} = (x', y')$ and extend $L_n$ inductively as follows.

\textbf{Case 1.} Suppose $(x,y)$ lies on neither a horizontal nor a vertical line.  
Extend $L_n$ diagonally in the $(1,1)$ direction until it first intersects a horizontal or vertical line, say at $(u,v)$.  
Let $\widehat{L}_n$ denote the diagonally extended path. Then:

\begin{itemize}
\item If the intersection is with a horizontal line only, set
\[
(x', y') = (u + \iota, v), \quad \text{where } \iota = 
\begin{cases}
i^c & \text{if the line is odd,} \\
i & \text{if the line is even.}
\end{cases}
\]

\item If the intersection is with a vertical line only, set
\[
(x', y') = (u, v + \jmath), \quad \text{where } \jmath = 
\begin{cases}
j^c & \text{if the line is odd,} \\
j & \text{if the line is even.}
\end{cases}
\]

\item If it intersects both, set $(x', y') = (u, v)$.
\end{itemize}

Define $L_{n+1}$ by extending $\widehat{L}_n$ with the segment from $(u,v)$ to $(x', y')$.

\textbf{Case 2.} Suppose $(x,y)$ lies on both a horizontal and vertical line with associated indices $\kappa(x,y) = (p,q)$.  
Then set
\[
(x', y') = (x + p^c, y + q^c),
\]
and define $L_{n+1}$ by attaching the segments from $(x,y)$ to $(x + p^c, y)$ and then to $(x + p^c, y + q^c)$.

\begin{remark}\label{rem: the lattice path}
The path proceeds in three possible directions: $(1,0)$, $(0,1)$, and $(1,1)$. Among these, the steps in the $(1,0)$ and $(0,1)$ directions occur only immediately after the path crosses a horizontal or vertical line, as follows:
\begin{itemize}
\item If the path crosses an even horizontal line, it contributes $i$ steps in the $(1,0)$ direction.
\item If the path crosses an odd horizontal line, it contributes $i^c$ steps in the $(1,0)$ direction.
\item If the path crosses an even vertical line, it contributes $j$ steps in the $(0,1)$ direction.
\item If the path crosses an odd vertical line, it contributes $j^c$ steps in the $(0,1)$ direction.
\end{itemize}
\end{remark}

Let $\mathcal{P} = \{p_0, p_1, \ldots\}$ be the resulting sequence of points.  
Note that no $p_n$ lies solely on a horizontal or solely on a vertical line.
For each point $p_k = (x,y)$ with $\kappa(x,y) = (p,q)$, where $p \in \{i, i^c\}$ and $q \in \{j, j^c\}$, we define $\cS(p_k) \in \tcS_\bfi(w)$ as follows:
\bni
\item If $p_k$ lies on neither a horizontal nor a vertical line, set
\[
\cS(p_k) := R_{(\bar{x} - p, \bar{y} - q), (p, q)}
\]
where $\bar{x}$ and $\bar{y}$ denote the residues of $x$ modulo $t$ and $y$ modulo $n - t + 1$, respectively.
\item If $p_k$ lies on both a horizontal and a vertical line, set
\[
\cS(p_k) := R_{(0, q^c), (p^c, q)}.
\]
\ee

\begin{lemma}\label{lem: D of path model}
For each point $p_k$ in the sequence,
\[
\mathcal{D}^{-1}(\cS(p_k)) = \cS(p_{k+1}).
\]
\end{lemma}

\begin{proof}
The assertion follows by a straightforward case-by-case verification using Proposition \ref{prop: rule for D minuscule type A}. 
\end{proof}

\begin{example}\label{ex: lattice path}
Set $n = 5$, $t = 3$, and $(i,j) = (1,2)$. The points $p_0, \ldots, p_4$ are marked as dots, and the thick path from $p_0$ to $p_4$ represents the lattice path $L_4$ in Figure \ref{Figure: P and L}.
The thick grid lines indicate the horizontal and vertical lines,
and the values $\kappa(p_k) = (p,q)$ are determined by the indices shown on the left and top, according to the horizontal and vertical regions containing $p_k$.

Let us compute $\cS(p_0), \ldots, \cS(p_4)$. We have $\cS(p_0) = R_{(0,0),(1,2)}$, and the remaining values are given by
\begin{align*}
\cS(p_1) &= R_{(1,0),(1,1)}, &
\cS(p_2) &= R_{(0,1),(2,1)}, \\
\cS(p_3) &= R_{(0,2),(1,1)}, &
\cS(p_4) &= R_{(0,0),(1,2)}.
\end{align*}
These elements also appear in Example \ref{ex: rectangle rule for D}.

\begin{center}
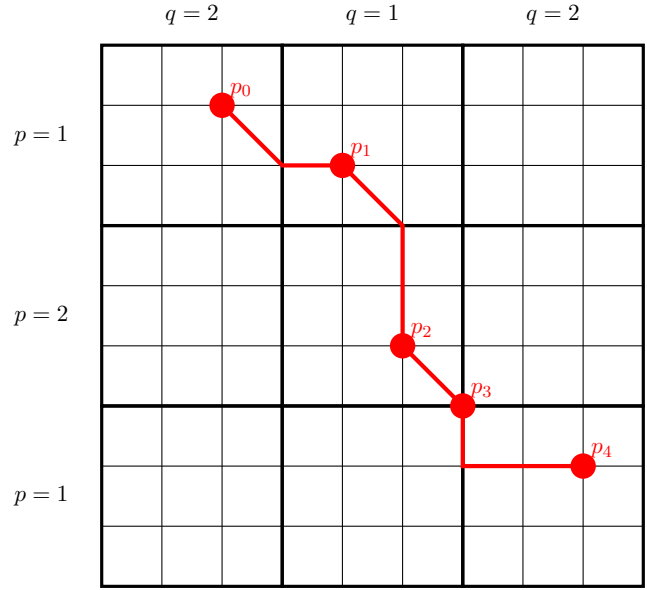
\begin{figure}
\begin{tikzpicture}[scale=0.8, every node/.style={scale=0.8}]
% Draw outer 3x3 grid blocks (each of size 3x3 small squares)
\foreach \i in {0,1,2} {
    \foreach \j in {0,1,2} {
        \draw[very thick] (\i*3,\j*3) rectangle ++(3,3);
    }
}

% Draw inner 9x9 thin grid
\foreach \i in {0,...,9} {
    \draw[thin] (\i,0) -- (\i,9);
}
\foreach \j in {0,...,9} {
    \draw[thin] (0,\j) -- (9,\j);
}

% Define coordinates of points
\coordinate (P0) at (2, 8);
\coordinate (Q0) at (3, 7);
\coordinate (P1) at (4, 7);
\coordinate (Q1) at (5, 6);
\coordinate (P2) at (5, 4);
\coordinate (P3) at (6, 3);
\coordinate (Q3) at (6, 2);
\coordinate (P4) at (8, 2);

% Draw red path
\draw[ultra thick, red] 
    (P0) -- (Q0) -- (P1) -- (Q1) -- (P2) -- (P3) -- (Q3) -- (P4);

% Draw red points
\foreach \pt in {P0,P1,P2,P3,P4} {
    \fill[red] (\pt) circle (6pt);
}

% Label points
\node[red] at (P0) [above right] {$p_0$};
\node[red] at (P1) [above right] {$p_1$};
\node[red] at (P2) [above right] {$p_2$};
\node[red] at (P3) [above right] {$p_3$};
\node[red] at (P4) [above right] {$p_4$};

% Add row labels (left)
\node at (-1,1.5) {$p=1$};
\node at (-1,4.5) {$p=2$};
\node at (-1,7.5) {$p=1$};

% Add column labels (top)
\node at (1.5,9.5) {$q=2$};
\node at (4.5,9.5) {$q=1$};
\node at (7.5,9.5) {$q=2$};
\end{tikzpicture}
\caption{The points $p_0,\ldots,p_4$ and path $L_4$ associated to $n = 5$, $t = 3$, and $(i,j) = (1,2)$}
\label{Figure: P and L}
\end{figure}
\end{center}

\end{example}

We analyze specific points in $\mathcal{P}$ which helps us to prove the main theorem. For each $u \in \mathbb{Z}_{>0}$, define
\begin{align}\label{eq: y(u)}
    y(u) := (n - t + 1) \cdot u +
    \begin{cases}
        j & \text{if } u \text{ is even},\\
        j^c & \text{if } u \text{ is odd}.
    \end{cases}
\end{align}
By the construction of $\mathcal{P}$, there exists a unique $x(u) \in \mathbb{Z}_{>0}$ such that $(x(u), y(u)) \in \mathcal{P}$. 
Observe that $(x(u), y(u))$ is the point in $\mathcal{P}$ reached after crossing $u$ vertical lines. 
Let $\mathcal{L}(u)$ be the associated lattice path from $(i,j)$ to $(x(u), y(u))$.
For instance, in Figure \ref{Figure: P and L}, $(x(1),y(1)) = p_1, \mathcal{L}(1) = L_1$ and $(x(2),y(2)) = p_4, \mathcal{L}(2) = L_4$.

\begin{lemma}\label{lem: x(u) and e}
    Let $u \ge 1$, and denote by $D(u)$ the number of $(1,1)$-moves in $\mathcal{L}(u)$.

    \bni
    \item The value $x(u)$ is determined by $D(u)$, and is explicitly given by
    \begin{align*}
        x(u) = 2 \left\lfloor \frac{D(u)}{t} \right\rfloor t +
        \begin{cases}
            i + \overline{D(u)} & \text{if } \overline{D(u)} < i^c, \\
            t + \overline{D(u)} & \text{if } \overline{D(u)} \ge i^c,
        \end{cases}
    \end{align*}
    where $\overline{D(u)} \in [0, t-1]$ such that $\overline{D(u)} = D(u) \bmod t$.

    \item  Let $e \in \mathbb{Z}_{>0}$ such that $p_e = (x(u), y(u))$. 
    Then
    \begin{align*}
        e = 
        u + 2  \left\lfloor \frac{D(u)}{t} \right\rfloor +
        \begin{cases}
            0 & \text{if } \overline{D(u)} < i^c,\\
            1 & \text{if } \overline{D(u)} \ge i^c.
        \end{cases}
    \end{align*}
    \item The value $D(u)$ is given by 
    \[
    D(u) = (n - t + 1) \cdot \left\lfloor \frac{u}{2} \right\rfloor +
\begin{cases}
0 & \text{if } u \text{ is even}, \\[4pt]
j^c & \text{if } u \text{ is odd}.
\end{cases}
    \]
    
    \ee
\end{lemma}

\begin{proof}
(i) By Remark~\ref{rem: the lattice path}, the $(1,0)$-steps in $\mathcal{L}(u)$ arise solely from horizontal lines crossed by the path. Since the contribution pattern of horizontal lines repeats every $t$ diagonal steps, the path crosses $\left\lfloor \frac{D(u)}{t} \right\rfloor$ even horizontal lines and the same number of odd horizontal lines, and crosses one additional odd horizontal line if $\overline{D(u)} \ge i^c$. Therefore, the formula follows.

(ii) By construction of the path, $e$ is equal to the total number of horizontal and vertical lines crossed by the path $\mathcal{L}(u)$. The expression for $e$ then follows directly from (i) and \eqref{eq: y(u)}.

(iii) 
Let $E(u)$ denote the number of $(0,1)$-steps in the path $\mathcal{L}(u)$. 
Since $\mathcal{L}(u)$ crosses $u$ vertical lines, it follows from Remark~\ref{rem: the lattice path} that
\[
E(u) = (n - t + 1) \cdot \left\lfloor \frac{u}{2} \right\rfloor +
\begin{cases}
0 & \text{if } u \text{ is even}, \\[4pt]
j^c & \text{if } u \text{ is odd}.
\end{cases}
\]
Since $y(u) = D(u) + E(u) + j$, the desired expression for $D(u)$ follows.
\end{proof}

We now state the main result of this subsection.
Recall the element $x_t$ defined in \eqref{Eq: xt}.
\begin{theorem}\label{thm: period of minuscule type A}
    When the Weyl group is of type $A_n$ and $t \in [1,n]$, we have
    \[
    \PD(x_t) = 
    \begin{cases}
        1 & \text{if } t = 1, n, \\[6pt]
        2 & \text{if } n =3 \text{ and } t = 2, \\[6pt]
        n + 1 & \text{if } n \ge 4 \text{ and } t = 2, n - 1,\\[6pt]
        \dfrac{2 (n + 1)}{\gcd(n + 1, t)} & \text{otherwise.}
    \end{cases}
    \]
\end{theorem}

\begin{proof}
Let $1 \le k \le l$, and let $(i^k, j^k) \in \lambda_t^\circ$ be the box defined in \eqref{Eq: ipjp}. Define $\xi_{i^k,j^k} := \xi_k$ and $d_{i^k,j^k}:=d_k$, where $\xi_k$ is given in Proposition~\ref{prop: period of minors determines total period} and $d_k$ is given in Lemma \ref{lem: string of qmD_bfi, k}.
Since, by Lemma~\ref{lem: type A minuscule, STR_i(d_k) frozen}, we have $\xi_{a,b} = 1$ whenever $a = t$ or $b = n - t + 1$, Proposition~\ref{prop: period of minors determines total period} says 
\[
\PD(x_t) = \lcm\{ \xi_{a,b} \mid 0 < a < t \text{ and } 0 < b < n - t + 1 \}.
\]

Let $0 < a < t$ and $0 < b < n - t + 1$.
By Lemma~\ref{lem: string of qmD_bfi, k}, we have
\begin{align}\label{eq: STR d_k}
    \STR_\bfi(d_{a,b}) = b_{a,b} = R_{(0,0),(a,b)}.
\end{align}
Let $\mathcal{P} = \{p_0, p_1, \ldots\}$ denote the sequence of points defined in this subsubsection, with initial point $p_0 = (a,b)$. 
Then, by Lemma~\ref{lem: D of path model} and~\eqref{eq: STR d_k}, the value $\xi_{a,b}$ is equal to the minimal integer $e \ge 1$ such that $\cS(p_e) \eqfr R_{(0,0),(a,b)}$.

To characterize such points $\cS(p_e)$, we analyze when a lattice point $(x,y) \in \mathcal{Z}$ satisfies $\cS(x,y) \eqfr R_{(0,0),(a,b)}$. 
This condition is equivalent to the following
\begin{equation}\label{eq: return condition}
\begin{cases}
x \equiv a \mod t, \\[2pt]
y \equiv b \mod (n - t + 1), \\[2pt]
\kappa(x,y) = (a,b).
\end{cases}
\end{equation}
If a point $(x,y) \in \mathcal{P}$ satisfies~\eqref{eq: return condition}, then there exists $u \ge 1$ such that $(x,y) = (x(u), y(u))$. 
We now determine the condition on $u$ under which $(x(u), y(u))$ satisfies~\eqref{eq: return condition}.

To proceed, we distinguish the following four types of $u$ depending on its parity and the value of $\overline{D(u)}$:
\begin{enumerate}[label=(\text{T}\alph*)]
    \item $u$ is even and $\overline{D(u)} = 0$,
    \item $u$ is even and $\overline{D(u)} = a$,
    \item $u$ is odd and $\overline{D(u)} = 0$,
    \item $u$ is odd and $\overline{D(u)} = a$.
\end{enumerate}
By inspecting the expressions for $x(u)$ and $y(u)$ given in Lemma \ref{lem: x(u) and e} (i) and~\eqref{eq: y(u)}, we deduce the following characterization for when $(x(u), y(u))$ satisfies the condition~\eqref{eq: return condition}:
\begin{equation}\label{eq: i j condition}
\begin{aligned}
&\text{If } a \ne a^c \text{ and } b \ne b^c, &&\text{then } (x(u), y(u)) \text{ satisfies } \eqref{eq: return condition} \text{ if and only if } u \text{ is of (Ta)}. \\[2pt]
&\text{If } a \ne a^c \text{ and } b = b^c, &&\text{then } (x(u), y(u)) \text{ satisfies } \eqref{eq: return condition} \text{ if and only if } u \text{ is of (Ta) or (Tc)}. \\[2pt]
&\text{If } a = a^c \text{ and } b \ne b^c, &&\text{then } (x(u), y(u)) \text{ satisfies } \eqref{eq: return condition} \text{ if and only if } u \text{ is of (Ta) or (Tb)}. \\[2pt]
&\text{If } a = a^c \text{ and } b = b^c, &&\text{then } (x(u), y(u)) \text{ satisfies } \eqref{eq: return condition} \text{ if and only if } u \text{ is of one of (Ta)--(Td)}.
\end{aligned}
\end{equation}

Note that type (Ta) arises in all four cases listed in~\eqref{eq: i j condition}. Let $d := \gcd(n - t + 1, t)$. Then the minimal value of $u$ such that $(x(u), y(u))$ is of type (Ta) is
\begin{equation}\label{eq: minimal u case a}
u = 2 \cdot \frac{t}{d}.
\end{equation}

We now analyze the following cases:

\smallskip
\noindent
\textbf{Case 1.} Suppose $t = 1$ or $t = n$. In this case, no mutable variable appears, and hence $\PD(x_t) = 1$.

%\smallskip
\noindent
\textbf{Case 2.} 
Suppose $n \ge 4$ and $2 < t < n - 1$. 
Then the pair $(a,b) = (1,1)$ satisfies $a \ne a^c$ and $b \ne b^c$. 
For this pair, the minimal value of $u$ such that $(x(u), y(u))$ satisfies the condition~\eqref{eq: return condition} is given by~\eqref{eq: minimal u case a}, due to~\eqref{eq: i j condition}.
By Lemma~\ref{lem: x(u) and e}~(ii) and (iii), if $p_e = (x(u), y(u))$, then
\[
e = 2 \cdot \frac{t}{d} + 2 \cdot \frac{n - t + 1}{d} = 2 \cdot \frac{n + 1}{d},
\]
which is the minimal integer such that $\cS(p_e) \eqfr R_{(0,0),(1,1)}$. 
Hence, we obtain $\xi_{1,1} = 2 \cdot \frac{n + 1}{d}$.

For all other pairs $(a,b) \ne (1,1)$, 
$\cdual^{\xi_{1,1}} (d_{a,b}) \eqfr d_{a,b}$  by~\eqref{eq: i j condition}
since type (Ta) appears in every case.
Thus the values $\xi_{a,b}$ divide $2 \cdot \frac{n + 1}{d}$, which says 
\[
\PD(x_t) = 2 \cdot \frac{n + 1}{d}.
\]
Finally, since $d = \gcd(n - t + 1, t) = \gcd(n + 1, t)$, the claim follows.

%\smallskip
\noindent
\textbf{Case 3.} Suppose $n \ge 4$ and $t = 2$. 
Then every pair $(a,b)$ satisfies $a = a^c = 1$. 
In particular, the pair $(a,b) = (1,1)$ satisfies $b \ne b^c$, since $n - t + 1 \ge 3$.
For this pair, the minimal value of $u$ such that $(x(u), y(u))$ satisfies the condition~\eqref{eq: return condition} is determined as follows:
\begin{itemize}
    \item If $n$ is odd, then type (Ta) yields $u = 2$, while type (Tb) does not occur.
    \item If $n$ is even, then type (Ta) yields $u = 4$, while type (Tb) yields $u = 2$.
\end{itemize}
In either case, the minimal $u$ is $2$. 
By Lemma~\ref{lem: x(u) and e}~(ii) and (iii), if $p_e = (x(2), y(2))$, then $e = 2 + (n - 1) = n + 1$, so $\xi_{1,1} = n + 1$.

For all other pairs $(a,b) \ne (1,1)$, we again have $\xi_{a,b} \mid n + 1$ by~\eqref{eq: i j condition}. 
Therefore,
\[
\PD(x_t) = n + 1.
\]

%\smallskip
\noindent
\textbf{Case 4.} Suppose $n \ge 4$ and $t = n - 1$. This is analogous to Case~3, and again we obtain $\PD(x_t) = n + 1$.

%\smallskip
\noindent
\textbf{Case 5.} Suppose $n = 3$ and $t = 2$. A direct computation using Lemma~\ref{lem: D of path model} shows that $\PD(x_t) = 2$.

%\smallskip
This completes the proof.
\end{proof}

\begin{remark} \label{rem: periodicity for A}
When $\g$ is of type $A_n$, the quantum unipotent coordinate ring $\qA_q(\n(x_t))$ can be understood as a quantum-analogue of the coordinate ring of the Grassmannian $\Gr(t, n+1)$ of $t$-dimensional subspaces in an $(n+1)$-dimensional space (see \cite{GLS08}). In Grassmannian cluster categories, the twist automorphism $\eta_{x_t}$ can be interpreted in terms of Auslander-Reiten translations and there is a result on its finite periodicity (\cite[Section 3.5]{Baur21}). Thus Theorem \ref{thm: period of minuscule type A} recovers the finite periodicity result \cite[Proposition 3.15]{Baur21} in terms of crystals. This periodicity is also related to Keller's result \cite{Keller13} on the periodicity conjecture (see \cite[Remark 3.16]{Baur21}). 
\end{remark}

\subsection{Further observations} \label{Sec: FO}
In this subsection, we assume that $\g$ is of finite type. 
Based on Theorem \ref{Thm: main}, one can compute $\cdual_w$ for arbitrary elements $w \in \weyl$ using \textsc{SageMath}~\cite{sage}. In this subsection, we present a collection of conjectures on the periodicity $\PD(w)$, motivated by patterns observed in the computational data.

\subsubsection{Minuscule or cominuscule cases}
For general finite type, we define a subset $I_{\mathrm{MC}} \subset I$ consisting of indices that are either minuscule or cominuscule:
\begin{align*}
    I_{\mathrm{MC}} = 
    \begin{cases}
        I & \text{if the Weyl group is of type $A_n$}, \\
        \{ 1, n \} & \text{if the Weyl group is of type $B_n$ or $C_n$}, \\
        \{ 1, n-1, n \} & \text{if the Weyl group is of type $D_n$}, \\
        \{ 1, 6 \} & \text{if the Weyl group is of type $E_6$}, \\
        \{ 7 \} & \text{if the Weyl group is of type $E_7$}, \\
        \emptyset & \text{if the Weyl group is of type $E_8$, $F_4$, or $G_2$}.
    \end{cases}
\end{align*}
We call an index $t \in I_{\mathrm{MC}}$ a \emph{minuscule or cominuscule index}.  
For type $A$, Theorem~\ref{thm: period of minuscule type A} determines $\PD(x_t)$ for all such $t$.  
For other types, we conjecture the following values of $\PD(x_t)$ based on computational data.

\begin{conjecture}
    Let $t \in I$. Then $\PD(x_t)$ is finite if and only if $t \in I_{\mathrm{MC}}$.  
Moreover, if $t \in I_{\mathrm{MC}}$, then the following hold:
    \bni
        \item When the Weyl group is of type \( B_n \) or \( C_n \):
        \begin{enumerate}
            \item If \( t = 1 \), then \(\PD(x_1) = 2\).
            \item If \( t = n \), then
            $
            \PD(x_t) = 
            \begin{cases}
                2 & \text{if } n = 2, \\[6pt]
                4 & \text{otherwise}.
            \end{cases}
            $
        \end{enumerate}
        
        \item When the Weyl group is of type \( D_n \):
        \begin{enumerate}
            \item If \( t = 1 \), then \(\PD(x_1) = 2\).
            \item If \( t \in \{ n-1, n \} \), then
            $
            \PD(x_t) = 
            \begin{cases}
                2 & \text{if } n = 4, \\[6pt]
                8 & \text{if } n \text{ is odd}, \\[6pt]
                4 & \text{otherwise}.
            \end{cases}
            $
        \end{enumerate}
        
        \item When the Weyl group is of type \( E_6 \) and \( t \in \{1, 6\} \), we have \(\PD(x_t) = 6\).
        
        \item When the Weyl group is of type \( E_7 \) and \( t = 7 \), we have \(\PD(x_t) = 4\).
    \ee
\end{conjecture}
This conjecture is supported by computations using \textsc{SageMath}. For all types $B_n$, $C_n$, $D_n$, $E_6$, and $E_7$ with rank $n \le 10$, we have verified that $\PD(x_t)$ agrees with the stated values when $t \in I_{\mathrm{MC}}$, and that $\PD(x_t) > 1000$ when $t \notin I_{\mathrm{MC}}$.

\subsubsection{Coxeter element cases}
We investigate the periodicity of a certain Coxeter element. We write $I = I_o \sqcup I_e$ such that $a_{i,j} = 0$ for all $i \in I_o$ and $j \in I_e$, and define the Coxeter element
\[
c := \left(\prod_{i\in I_o} s_i\right) \left(\prod_{j\in I_e} s_j \right) \in \weyl
\]
We now consider two classes: class $(A)$ consists of types $A_n$, $D_n$ (with $n$ odd), and $E_6$, and class $(B)$ consists of types $B_n$, $C_n$, $D_n$ (with $n$ even), $F_4$, and $G_2$.

\begin{conjecture}
Let $m \in \mathbb{Z}_{>0}$ be such that $m \cdot \ell(c) \le \ell(w_\circ)$. Then
\[
\PD(c^m) = 
\begin{cases}
h+m & \text{if } m = 2 \text{ and } \weyl \text{ is in class } (A), \\[6pt]
\dfrac{2(h+m)}{\gcd(h,m)} & \text{if } m > 2 \text{ and } \weyl \text{ is in class } (A), \\[6pt]
\dfrac{h+m}{2} & \text{if } m = 2 \text{ and } \weyl \text{ is in class } (B), \\[6pt]
\dfrac{h+m}{\gcd(h/2, m)} & \text{if } m > 2 \text{ and } \weyl \text{ is in class } (B),
\end{cases}
\]
where $h$ is the Coxeter number of $W$. 
\end{conjecture}
This conjecture has also been confirmed using \textsc{SageMath} for all types of rank $n \le 10$, in agreement with the previous conjecture.

\subsubsection{Type $A$ parabolic cases}
Let $J \subset I$. Define $\weyl_J$ to be the subgroup of $\weyl$ generated by the simple reflections $s_i$ for $i \in J$.
We denote by $x_J$ the longest element of the set of minimal-length coset representatives for $\weyl_J \backslash \weyl$.

In the following tables, we present the computed values of $\PD(x_J)$ for $\mathfrak{g} = A_4$, $A_5$, and $A_6$, obtained using \textsc{SageMath}.
For $\mathfrak{g} = A_5$, we have verified that any subset $J \subset I$ not listed in Table~\ref{table: A5 parabolic} satisfies $\PD(x_J) > 10{,}000$.
Similarly, for $\mathfrak{g} = A_6$, any subset $J \subset I$ not listed in Table~\ref{table: type A6 parabolic} satisfies $\PD(x_J) > 100$. These observations say that the subsets not appearing in the tables correspond to cases where $\PD(x_J)$ is infinite.

\begin{table}[h!]
\centering
\renewcommand{\arraystretch}{1}
\begin{tabular}{|c||c|c|c|c|c|c|c|c|}
\hline
$J$ & $\emptyset$ & $\{1\}$ & $\{2\}$ & $\{3\}$ & $\{4\}$ & $\{1,2\}$ & $\{1,3\}$ & $\{1,4\}$ \\
\hline
$\PD(x_J)$ & $6$ & $10$ & $8$ & $8$ & $10$ & $6$ & $7$ & $4$ \\
\hline
\hline
$J$ & $\{2,3\}$ & $\{2,4\}$ & $\{3,4\}$ & $\{1,2,3\}$ & $\{1,2,4\}$ & $\{1,3,4\}$ & $\{2,3,4\}$ & $\{1,2,3,4\}$ \\
\hline
$\PD(x_J)$ & $2$ & $7$ & $6$ & $1$ & $5$ & $5$ & $1$ & $1$ \\
\hline
\end{tabular}
%\begin{center}
%Table 1. Values of $\PD(x_J)$ for each $J \subset \{1,2,3,4\}$ in type $A_4$
%\end{center}
\caption{Values of $\PD(x_J)$ for each $J \subset \{1,2,3,4\}$ in type $A_4$}
\end{table}

\begin{table}[h!]
\centering
\small
\renewcommand{\arraystretch}{1}
\begin{adjustbox}{max width=\textwidth}
\begin{tabular}{|c||c|c|c|c|c|c|c|c|}
\hline
$J$ & $\emptyset$ & $\{1,2\}$ & $\{1,4\}$ & $\{2,3\}$ & $\{2,5\}$ & $\{3,4\}$ & $\{4,5\}$ & $\{1,2,3\}$ \\
\hline
$\PD(x_J)$ & $6$ & $10$ & $16$ & $10$ & $16$ & $10$ & $10$ & $7$ \\
\hline
\hline
$J$ & $\{1,2,4\}$ & $\{1,2,5\}$ & $\{1,3,4\}$ & $\{1,4,5\}$ & $\{2,3,4\}$ & $\{2,3,5\}$ & $\{2,4,5\}$ & $\{3,4,5\}$ \\
\hline
$\PD(x_J)$ & $14$ & $14$ & $9$ & $14$ & $2$ & $9$ & $14$ & $7$ \\
\hline
\hline
$J$ & $\{1,2,3,4\}$ & $\{1,2,3,5\}$ & $\{1,2,4,5\}$ & $\{1,3,4,5\}$ & $\{2,3,4,5\}$ & $\{1,2,3,4,5\}$ & & \\
\hline
$\PD(x_J)$ & $1$ & $6$ & $4$ & $6$ & $1$ & $1$ & & \\
\hline
\end{tabular}
\end{adjustbox}
\caption{Values of $\PD(x_J)$ for each $J \subset \{1,2,3,4,5\}$ in type $A_5$}
\label{table: A5 parabolic}
\end{table}

\begin{table}[h!]
\centering
\small
\renewcommand{\arraystretch}{1}
\begin{adjustbox}{max width=\textwidth}
\begin{tabular}{|c||c|c|c|c|c|c|}
\hline
$J$ & $\emptyset$ & $\{1,2\}$ & $\{1,5\}$ & $\{2,4\}$ & $\{2,6\}$ & $\{3,5\}$ \\
\hline
$\PD(x_J)$ & $6$ & $10$ & $16$ & $14$ & $16$ & $14$ \\
\hline
\hline
$J$ & $\{5,6\}$ & $\{1,2,5\}$ & $\{1,2,6\}$ & $\{1,3,5\}$ & $\{1,5,6\}$ & $\{2,3,4\}$ \\
\hline
$\PD(x_J)$ & $10$ & $14$ & $14$ & $20$ & $14$ & $12$ \\
\hline
\hline
$J$ & $\{2,4,6\}$ & $\{2,5,6\}$ & $\{3,4,5\}$ & $\{1,2,3,4\}$ & $\{1,2,3,6\}$ & $\{1,2,4,5\}$ \\
\hline
$\PD(x_J)$ & $20$ & $14$ & $12$ & $8$ & $16$ & $4$ \\
\hline
$J$ & $\{1,3,4,5\}$ & $\{1,4,5,6\}$ & $\{2,3,4,5\}$ & $\{2,3,4,6\}$ & $\{3,4,5,6\}$ & $\{1,2,3,4,5\}$ \\
\hline
$\PD(x_J)$ & $11$ & $16$ & $2$ & $11$ & $8$ & $1$ \\
\hline
\hline
$J$ & $\{1,2,3,4,6\}$ & $\{1,2,3,5,6\}$ & $\{1,2,4,5,6\}$ & $\{1,3,4,5,6\}$ & $\{2,3,4,5,6\}$ & $\{1,2,3,4,5,6\}$ \\
\hline
$\PD(x_J)$ & $7$ & $14$ & $14$ & $7$ & $1$ & $1$ \\
\hline
\end{tabular}
\end{adjustbox}
\caption{Values of $\PD(x_J)$ for each $J \subset \{1,2,3,4,5,6\}$ in type $A_6$}
\label{table: type A6 parabolic}
\end{table}

\vskip 2em 

\FloatBarrier

\end{document}